\documentclass[11pt]{article}

\usepackage[a4paper,headinclude=false,footinclude=false,DIV=10]{typearea}
\usepackage{setspace}

\usepackage{enumitem}

\usepackage{amsmath,amsthm,amssymb}
\usepackage{tikz}
\usepackage{tikz-cd}
\usepackage{stmaryrd}

\usepackage{epsfig,graphicx}
\usepackage[colorlinks=false]{hyperref}
\usepackage[english]{babel}

\usepackage[utf8]{inputenc}  

\usepackage[dvipsnames]{xcolor}

\usepackage{fourier} 
\usepackage[scaled=0.875]{helvet} 

\newtheoremstyle{mystyle}{}{}{\rmfamily}%
{}{\normalfont\bfseries}{ }{ }{} 

\newtheorem{theorem}{Theorem}[section]
\newtheorem{proposition}[theorem]{Proposition}
\newtheorem{lemma}[theorem]{Lemma} 
\newtheorem{corollary}[theorem]{Corollary}
\newtheorem{defn}[theorem]{Definition}
\theoremstyle{mystyle}
\newtheorem{remark}[theorem]{Remark}
\newtheorem{example}{Example}

\newcommand{\R}{\mathbb{R}}
\newcommand{\N}{\mathbb{N}}
\renewcommand{\H}{\mathbb{H}}
\newcommand{\Z}{\mathbb{Z}}

\newcommand{\cW}{\mathcal{W}}

\newcommand{\cH}{\mathcal{H}}

\newcommand{\V}{\mathbf{V}}
\newcommand{\E}{\mathbf{E}}
\newcommand{\G}{\mathcal{G}}

\newcommand{\te}{{\theta}}

\newcommand{\Om}{{\Omega}}

\newcommand{\Sig}{{\Sigma}}

\newcommand{\1}{\mathbf{1}}
\newcommand{\with}{\;:\;}

\DeclareMathOperator{\id}{id}

\begin{document}

\vspace*{1cm}

\begin{center}
\setstretch{2}
{\Huge \bfseries  Amenable graphs and the spectral radius of extensions of Markov maps}\\

Johannes Jaerisch\textsuperscript{(1)}, Elaine Rocha\textsuperscript{(2)} and Manuel Stadlbauer\textsuperscript{(3)}\\ 
\bigskip
{\scriptsize	
\textsuperscript{(1)}
Graduate School of Mathematics, Nagoya University, 
Furocho, Chikusaku, Nagoya, 464-8602, Japan \\[-.2cm] 
\textsuperscript{(2)}
Universidade Federal do Vale do São Francisco,
Rua António Figueira, 134.
56000-000 Salgueiro (PE), Brazil\\[-.2cm]
\textsuperscript{(3)}
Departamento de Matemática, Universidade Federal do Rio de Janeiro,
Ilha do Fundão,  21941-909 Rio de Janeiro (RJ), Brazil\\[-.2cm]
}
\bigskip \bigskip
{\small \today}
\end{center}

\begin{abstract} We discuss relations between the amenability of a graph and spectral properties of a random walk driven by a dynamical system. In order to include graphs which are not locally compact, we introduce the concept of amenability of weighted graphs, which generalises the usual notion as the new definition is shown to be equivalent to F\o lner's condition. 
As a first result, we obtain the following generalisation of Kesten's amenability criterion to graphs and non-independent increments: If the random walk is driven by a full-branched Gibbs-Markov map, the graph is amenable with respect to the weight induced by the random walk if and only if the spectral radius of the associated Markov operator is equal to one. By employing inducing schemes, one then obtains criteria for amenability through Markov maps with less regularity.\\
\indent  We conclude the paper with the following applications to Schreier graphs. If the random walk is driven by a uniformly expanding map with non-Markovian increments {\color{black} or a Sinai billiard}, then, under certain conditions, the Schreier graph is amenable if the probability of a return in time $n$ does not decay exponentially in $n$.
Furthermore, in the context of geometrically finite Kleinian groups, one obtains a version of Brooks's amenability criterion for not necessarily normal subgroups. \\[3mm]
\end{abstract}

\section{Introduction and statement of main results}

The notion of amenability goes back to von Neumann who referred to a locally compact group $G$ as amenable if there exists a finitely additive probability measure which is invariant under translations by  elements of $G$. If, in addition, the group $G$ is countable, it is known from the seminal contributions by F{\o}lner and Kesten, that this abstract condition can be detected either by the growth of $gK \cap K$, for fixed $g\in G$, as the cardinality of $K\subset G$ tends to infinity (\cite{Folner--On-Groups-With-Full--MS1955}) or the exponential decay of the probabilities of returning in time $n$ of a simple random walk on $G$ (\cite{Kesten--Full-Banach-Mean-Values--MS1959}). In both cases, these criteria can be rephrased in terms of the Cayley graph of $G$, giving rise to definitions of amenability in the context of graphs through the growth of the boundary of finite subsets (see \cite{Gerl--Amenable-Groups-And-Amenable--1988}) or the spectral radius of a Markov operator of the random walk on a discrete semigroup (see \cite{Day--Convolutions-Means-And-Spectra--IJM1964}).

The aim of this note is to extend the concepts of amenability and random walks and relate this new form of amenability with the spectral theory of a Markov operator. Since we are interested in graphs which  might contain vertices with infinitely many adjacent edges, we introduce a notion of amenability for weighted graphs. That is, we refer to a graph $\G$ with vertices $\V$ and edges $\E$  as a \emph{weighted graph} with weight $p:\E\to [0,1] $ if for all $v \in \V$, we have $\sum_{u \in \V} p((v,u)) =1$. The $\epsilon$-boundary of $K$, for $\epsilon > 0$ and $K \subset \V$, is then defined
as
\[  \partial^\epsilon K:= \{ v \in K: \exists e \in \E \hbox{ s.t. } s(e) =v, t(e) \notin K, p(e) > \epsilon \}. \]
We then refer to the weighted graph $\G$ with weight $p$ as \emph{$p$-amenable} if
\begin{equation}\label{eq:isoperimetric}
\lim_{\epsilon \to 0} \inf\left\{  {|\partial^\epsilon K|}/{|K|}: {K \subset \V, |K|< \infty} \right\} =0. \end{equation}
Or, in other words, a graph is $p$-amenable if and only if, for all $\epsilon > 0$, the graph obtained by removing the edges of weight smaller than $\epsilon$ is amenable. In particular, if $p$ is uniformly bounded from below, then $p$-amenability coincides with Gerl's notion of amenability in \cite{Gerl--Amenable-Groups-And-Amenable--1988}. Observe that \eqref{eq:isoperimetric} is an asymptotic isoperimetric inequality which can be rephrased through F{\o}lner sequences (see Proposition \ref{prop:mu-amenable} below).

On the other side, we are interested in an object which provides more flexibility than a classical random walk on a graph. That is, we are interested in dynamical systems of the form
\begin{equation} \label{def:1st-skew}
T : X \times \V \to X \times \V, (x,v) \mapsto (\theta(x),\kappa_x(v)),
\end{equation}
where $\theta: X \to X$ is a sufficiently well-behaved dynamical system and $x \mapsto \kappa_x$ is a map from $X$ to the set of bijections of $\V$ such that $(v,\kappa_x(v)) \in \E$. Hence, for any $x \in X$ and $v \in \V$, the evolution of the second coordinate of $T^n(x,v)$ corresponds to a walk on $\V$ along the edges $\E$ of $\G$. Moreover, if one chooses $x$ according to a $\theta$-invariant probability measure $\mu$ on $X$ (i.e. $\mu \circ \theta^{-1} = \mu$), one obtains a stationary random walk on $\V$ with not necessarily independent increments as indicated in the applications at the end of this introduction.  We refer to $T$ in \eqref{def:1st-skew} as an extension by the graph $\G=(\V,\E)$ through $\kappa$ (see Definition \ref{defn:graph_extension}). 

We now specify $\theta: X \to X$ for the first main result in a slightly simplified setting in order to avoid defining Markov maps in this introduction. Assume that $(X,\theta)$ is a full shift with an at most countable 
 alphabet $\mathcal{A}$, i.e. $\theta$ acts on $X:= \{ (x_0,x_1,\ldots) : x_i \in \mathcal{A} \hbox{ for } i \geq 0\}$ through $\theta: (x_0,x_1,\ldots)  \mapsto (x_1,x_2, \ldots)$.
Moreover, we assume that $(X,\theta)$ comes with a Borel probability measure $\mu$ on $X$ such that $\mu \circ \theta^{-1} = \mu$ and $\log {d\mu}/{d\mu \circ \theta}$ is Hölder continuous, where $ {d\mu}/{d\mu \circ \theta}$ stands for the Radon-Nikodym derivative with respect to regions of injectivity (for a definition without this detail, see Definition \ref{def:gm-map}). Moreover, we have to assume for Theorem A and B below  that $\kappa \mapsto \kappa_x$ only depends on the first coordinate $x_0$ of $x\in X$.  We say that the graph $\G=(\V,\E)$ is \emph{ $\mu$-amenable} if it is $p$-amenable with respect to $p(u,v):= \mu(\{x \in X: \kappa_x(u)=v\})$. A further relevant definition  is the notion of \emph{uniform loops} which refers to the existence of  a finite set $\mathcal{J} \subset X$ such that  $v \in \{ \kappa_x(v) : x \in \mathcal{J}\}$ for all $v \in \V$ (cf. Definition \ref{defn:graph_extension_uniform_loops}).

In order to state the result, it remains to introduce the transfer operator $\widehat{T}$ of $T$, which is the unique operator acting on the $L^1$-space on $X \times \V$ with respect to the product of $\mu$ and the counting measure $m_\V$ on $\V$ such that
$ \int (f \circ T) g d\mu \otimes m_\V = \int f \widehat{T}(g) d\mu \otimes m_\V $
for all $f \in L^\infty$ and $g \in L^1$. Note that it follows from general ergodic theory, that
\[ \widehat{T}(g)(x,v) = \sum_{T(y,w)=(x,v)} \tfrac{d\mu}{d\mu \circ \theta}(y)g(y,w).\]

\noindent\textbf{Theorem A} (cf. Theorem \ref{theo:main_result} {\color{black} and Corollary \ref{cor:bip}})\textbf{.}
\textit{
 Let   $(X,\theta,\mu)$ and $\kappa$ be as above and assume that the extension $T$ by the graph $\G=(\V,\E)$  through $\kappa$ is topologically transitive and has uniform loops.    
Then the following are equivalent.
\begin{enumerate}
 \item The graph $\G$ is $\mu$-amenable.
 \item The spectral radius of $\widehat{T}$, acting on $\left\{ f: X\times \V \to \R | \sum_{v\in \V} (\| f(\cdot , v)\|_\infty)^2< \infty \right\}$, is equal to $1$.
  \item For each $\epsilon > 0$, there exists a finite subset $A \subset \V$ such that
 \[  \int | \widehat{T}(\1_{X \times A}) - \1_{X \times A} | d\mu \otimes m_\V \leq \epsilon \cdot m_\V(A).\]
\end{enumerate}}
We remark that the proof relies on methods developed in \cite{Stadlbauer--An-Extension-Of-Kestens--AM2013} and \cite{Jaerisch--Group-Extended-Markov-Systems--PMS2015}, and that Theorem A extends results in there to  graph extensions.
{\color{black} However, it is worth noting that Theorem \ref{theo:main_result}  applies only to
Markov maps with full branches. While this is the natural framework from the perspective of random walks, it also allows a  more accessible proof than that given  in \cite{Stadlbauer--An-Extension-Of-Kestens--AM2013}. To extend this result to the broader class of Markov maps with big images and preimages and include a sufficient condition for amenability in terms of pressure in Corollary \ref{cor:bip}, we employ an inducing scheme. In particular, this generalises the main results for  Markov maps with big images and preimages in \cite{Stadlbauer--An-Extension-Of-Kestens--AM2013,Jaerisch--Group-Extended-Markov-Systems--PMS2015} to graph extensions.}

On the other hand, under a certain weak condition on the symmetry of $\mu$,  referred to as \emph{symmetric} here (cf. \eqref{eq:weakly-symmetric}), it is possible to add a further equivalence in flavour of Kesten's result for symmetric random walks. Namely, $\G$ is  $\mu$-amenable if and only if
\begin{align*} \label{eq:1st-growth}
   \limsup_{n \to \infty} \sqrt[n]{ \mu\left(\left\{x \in X : T^n(x,v) \in X \times \{v\}\right\}\right) } = 1
\end{align*}
for some (and hence any) $v \in \V$ (cf. Proposition \ref{prop:symmetry}). Here, it is worth noting that this condition can be rephrased in terms of the Gurevich pressure (see Proposition \ref{prop:gurevic-pressure}). Namely, if $T$ is symmetric, then $\mu$-amenability is equivalent to $T$ having Gurevich pressure 0. Even though these results are of interest as they generalize the seminal results of Kesten (\cite{Kesten--Full-Banach-Mean-Values--MS1959}) and Day (\cite{Day--Convolutions-Means-And-Spectra--IJM1964}) for groups and semigroups to a walk on $\G$ driven by $\theta$, the motivation behind Theorem A is to use it as a tool for analysing graph extensions over a map $\theta: X \to X$ which admit an induced map or inducing scheme which can be modeled as a full shift.
As any change of the inducing scheme of $\theta$ also affects the skew product $T$, it was necessary to work with the quite technical notion of a \emph{Markov map with adequately embedded  Gibbs-Markov structure} (cf. Definition \ref{def:embedded Gibbs-Markov structure}) in order to provide the necessary tools for comparing the exponential growth rates of $T$ with its induced counterpart (cf. Proposition \ref{prop:pressure_T_vs_S}).
In particular, we were able to show that a non-exponential decay of the measure of certain returns implies amenability,  which is considered to be the hard direction in Kesten's amenability criterion. 

\medskip
\noindent\textbf{Theorem B} (cf. Theorem \ref{theo:main theorem - embedded GM structure})\textbf{.}
\textit{Suppose that $T$ is topologically transitive with   adequately embedded  Gibbs-Markov structure $(\Omega,\sigma)$  such that the induced graph extension satisfies the hypotheses of Theorem A . Moreover, assume that the inducing time decays exponentially  and that $\hat\kappa$ finitely covers $\kappa$ (cf. Definition  \ref{def:finitely-covers}). Then $\G$ is $\mu$-amenable
if for some $v\in \V$,  \[  \limsup_{n \to \infty} \sqrt[n]{ \mu\left(\left\{x \in \Omega : T^n(x,v) \in X \times \{v\}\right\}\right) } = 1 .\]
}

\medskip
{\color{black} As a first application, we obtain Corollary \ref{cor:bip}, announced earlier, as well as an application to maps with an infinite invariant measure (Corollary \ref{cor:nue}).}

The remainder of this paper is devoted to applications of Theorems A and B to random walks on graphs and groups and to Schreier graphs  whose construction we recall now. Let $G$ be a discrete group, $H$ a subgroup of $G$ and $\mathfrak{g} \subset G$ a generating set of $G$. The \emph{Schreier graph} $\mathcal{G}=(\V,\E)$ associated with $\mathfrak{g}$ is then defined as the graph whose vertices are the cosets $\V=\{ Hg : g \in G\}$ and edges $\E =  \{ (Hg,Hgh) : g \in G, h \in   \mathfrak{g} \}$ are given by the right action of $\mathfrak{g}$ on $\V$. It is worth noting that
$\mathcal{G}$ coincides with the Cayley graph of $G/H$ if $H$ is a normal subgroup of $G$.

In order to define a graph extension of the Markov map $(X,\theta)$, it now suffices to specify a map $\gamma: X \to \mathfrak{g}$, $x \mapsto \gamma_x$ and define
\begin{equation} \label{eq:def-Schreier-graph-extension} T:  X \times \V \to X \times \V, \; (x,Hg) \mapsto (\te x, Hg\gamma_x). \end{equation}
If $\gamma$ is measurable with respect to the partition of the Markov map $\theta$, we say that the extension has \emph{Markovian increments}. In this case, the flexibility provided by embedded Markov maps and Theorem B allows to obtain the following sufficient conditions. For $H_0 := \bigcap_{g \in G} gHg^{-1}$, which is the maximal normal subgroup contained in $H$, define $T_0$ as in \eqref{eq:def-Schreier-graph-extension}.
For example, if $\gamma(X)$ is finite and $T_0$ is topologically transitive, then Theorem B holds for any adequately embedded Gibbs-Markov structure with exponentially decaying inducing time (cf. Theorem \ref{theo:Schreier - embedded GM structure}).
Provided that the base map is uniformly expanding, an application of 
Theorem \ref{theo:Schreier - embedded GM structure} then allows to obtain an amenability criterion for the Schreier graph through extensions with non-Markovian increments.
 
\medskip
\noindent\textbf{Theorem C}
(cf. Theorem \ref{theo:Non-Markovian} and Remark \ref{remark:local-diffeo})\textbf{.}
\textit{Assume that $X$ is a connected {\color{black} and compact} Riemannian manifold, that $\theta$ is a surjective and $C^2$-local diffeomorphism with $\|D(\theta)^{-1}\| < 1$, and that
   $H$ is a subgroup of the finitely generated  discrete group $G$. Furthermore, assume that
   $\gamma : X \to G$ is a map such that the following holds.
\begin{enumerate}
\item The image $\gamma(X)$ of $\gamma$ is finite and the interiors of $\gamma^{-1}(\{g\})$ are non-empty for $g \in \gamma(X)$.
\item For all open subsets $U,V \subset X$ and $g \in G$, there exist $n \in \N$ and $x \in X$ such that $x \in U \cap \theta^{-n}(V) \neq \emptyset$ and $(\gamma_x \cdots \gamma_{\theta^{n-1}(x)}) g^{-1} \in \bigcap_{h \in G} hHh^{-1}$.
\item The set $ \Delta := \bigcup_{n \geq 0} \bigcup_{g \in \gamma(X)} \theta^n(\partial(\gamma^{-1}(\{g\})) )$ is not dense.
\item
{\color{black}
 \[ \limsup_{n \to \infty} \sqrt[n]{ \mathrm{Leb}\left(\left\{ x \in X   :  \gamma_x \cdots \gamma_{\theta^{n-1}(x)} \in H  \right\}\right) } =1.\]}
\end{enumerate}
Then the Schreier graph of $H$ with respect to $\mathfrak{g}  = \gamma({\color{black}X})$ is amenable.}\\

{\color{black}
On the other hand, if the cocycle is sufficiently regular, it is also possible to consider skew products with a non-uniformly hyperbolic base that admit a Young tower. In here, we only discuss billiard maps with finitely many, strictly convex scatterers and a finite horizon on the two-torus. The associated walk on the  graph then is defined by following an edge in the graph at each collision. In the context of Schreier graphs, this means that this walk along edges is given by a map $\kappa$ from the set of  {scatterers} to $G$.

\medskip
\noindent\textbf{Theorem D} (cf. Theorem \ref{theo:sinai-billiard})\textbf{.}
\textit{
Assume that $f: (X,m) \to (X,m)$ is the billiard map, that $m$ is the canonical invariant probability measure, and that $\Lambda$ is a base of a Young tower. Then the Schreier graph of $H$ with respect to $\mathfrak{g} = \kappa(X)$ is $m$-amenable if the following holds.
\begin{enumerate}
 \item
For all open subsets $U,V \subset X$ and $g \in G$, there exist $n \in \N$ and $x \in X$ such that $x \in U \cap f^{-n}(V) \neq \emptyset$ and $(\kappa_x \cdots \kappa_{f^{n-1}(x)}) g^{-1} \in \bigcap_{h \in G} hHh^{-1}$.
\item
We have that
 $ \limsup_{n \to \infty} \sqrt[n]{ m \left(\left\{ x \in \Lambda \cap f^{-n} \Lambda :  \kappa_x \cdots \kappa_{f^{n-1}(x)} \in H  \right\}\right) } =1$.
\end{enumerate}
}
}

The remaining application to Schreier graphs in Section \ref{subsec:kleinian} is of more classical flavour. In there, the above theory is applied to non-regular covers of a class of geometrically finite hyperbolic manifolds. In case of surfaces, the main result of this paragraph is as follows.

\medskip
\noindent\textbf{Theorem E} (cf. Theorem \ref{theo-fuchsian})\textbf{.}
\textit{Assume that $\mathbb{H}/G$ is a  geometrically finite hyperbolic surface and that $H$ is a subgroup of $G$ such that $H_0$ is non-trivial. Then the Schreier graph of $H$ is amenable if and only if $\delta(G) = \delta(H)$, where $\delta(G)$ and $\delta(H)$ refer to the abscissas of convergence of $G$ and $H$, respectively.}
\medskip

We remark that Theorem E is a special case of a recent result by Coulon, Dougall, Schapira and Tapie in \cite{Coulon-Dougall-Tapie--Twisted-Patterson-sullivan-Measures-And--PA2018} which was obtained in a purely geometric context using a twisted Patterson-Sullivan measure,  which establishes
a connection to unitary, positive representations. However, we would like to point out that our method in here is different  and that the result provides an example with an inducing scheme without exponential  tails (see Remark \ref{rem:non-necessary}).

As a last application, we establish a connection to random walks on graphs which allows to compare Theorem A for not necessarily independent increments with the classical results by Day and Gerl for the independent case.

\section{Markov maps, graph extensions and amenability}
We begin by recalling the definition of Markov maps (or Markov fibred systems) from \cite{Aaronson-Denker-Urbanski--Ergodic-Theory-For-Markov--TMS1993} (see also \cite{Aaronson--An-Introduction-To-Infinite--1997}). 
\begin{defn}\label{def:Markov_map} Suppose that $(X,\mathcal{B},\mu)$ is a standard probability space and $\alpha$ is an at most countable partition of $X$ into measurable sets of strictly positive measure. We refer to $(X,\theta,\mu,\alpha)$ as a {Markov map} if, for all $a,b \in \alpha$, 
\begin{enumerate}
\item  $\theta|_a : a \to \theta(a)$ is invertible, bimeasurable and non-singular,  
\item  either $\mu(a \cap \theta(b)) = 0$ or $\mu(a \cap (\theta(b))^\mathbf{c})=0$, 
\end{enumerate}
and, for $\alpha_n:= \left\{ a_1 \cap \theta^{-1}a_2 \cdots \cap \theta^{n-1}a_{n} : a_i \in \alpha,  i=1, \ldots, n \right\}$, the $\sigma$-algebra generated by $\{\alpha_n:n>0\}$ is equal to $\mathcal{B}$ up to sets of measure $0$.
\end{defn}
Note that each Markov map comes with an associated topological Markov chain. This object, whose construction we recall now, is an effective tool for handling the preimage structure. Set $\cW^1 := \alpha$, and for $w_i \in \cW^1$ ($i=1,\ldots,n$) we say that  $w = (w_1 \ldots w_n)$ is an \emph{admissible word of length} $n$ if $\theta(w_i) \supset w_{i+1}$ for $i=1, \ldots, n-1$.
The {set of admissible words of length $n$} will be denoted by $\cW^n$, the length of $w \in \cW^n$ by $|w|$ and the set of all admissible words by $\cW^\infty = \bigcup_n \cW^n$. 
As it easily can be verified,
\begin{equation}\label{eq:correspondence_cylinders_words} \cW^n \to \alpha_n,\quad  (w_1 \ldots w_n) \mapsto [w_1 \ldots w_n] := \bigcap_{k=1}^{n} \theta^{-k+1}(w_k) \end{equation}
defines a bijection between $\cW^n$ and $\alpha_n$.
Each $w \in \cW^n$ can be identified with an inverse branch of $\theta^n$ as follows. Since  $\theta^n$ maps ${[w]}$ injectively onto its image, its inverse $\tau_w: \theta^n([w]) \to [w]$ is well defined and by \emph{(i)} in Definition \ref{def:Markov_map}, 
\[0 <  \varphi_w(x) := \frac{d \mu \circ \tau_w}{d \mu}(x) < \infty \]   
for $\mu$-a.e. $x \in \theta^n([w])$.
The associated topological Markov chain is defined by $(\Sig,\tilde\theta)$, with 
\[\Sig := \left\{(w_1w_2 \ldots) : w_kw_{k+1} \hbox{ admissible for } k=1,2, \ldots \right\}\]  
and $\tilde\theta$ referring to the left shift. The relevance of this object is twofold. First, the identification in (\ref{eq:correspondence_cylinders_words}) gives rise to a measure $\tilde{\mu}$ on $\Sig$ by setting $\tilde{\mu}(\{(w_1\ldots w_n v_1 v_2 \ldots) \in \Sig: v_i \in \cW^1\}) := \mu([w_1\ldots w_n])$. By an argument based on the last condition in Definition \ref{def:Markov_map}, it is then easy to construct a measure theoretic bijection between $(\Sig, \tilde{\mu})$ and $(X,\mu)$ such that $\tilde\theta$ and $\theta$ commute. For ease of notation, we will make use of $\theta$ 
and $\mu$, for $\tilde\theta$ and $\tilde\mu$, respectively.

Furthermore, $\Sig$ comes with a canonical topology generated by $\{[w] : w \in \cW^\infty\}$ which coincides with the topology induced by the metric  $d_r$ defined by, for any $r \in (0,1)$, 
\begin{equation}\label{def: d_r metric}
  d_r((x_i),(y_i)) := r^{\min\{i: x_i \neq y_i\}}.
\end{equation}
This metric will play a crucial role for the definition of invariant function spaces. Furthermore, it gives rise to topological irreducibility conditions. We will refer to $(X,\theta,\mu,\alpha)$ as a \emph{topologically transitive} Markov map 
 if for all $a,b \in \alpha$, there exists $n_{a,b}\in \N$ such that  $\mu(\te^{n_{a,b}}(a)\cap b)>0$ and 
 as \emph{topologically mixing} if for all $a,b \in \alpha$, there exists $N_{a,b}\in \N$ such that $\mu(\te^{n}(a)\cap b)>0$ for all $n \geq N_{a,b}$.
  
\begin{defn}  $(X,\theta,\mu,\alpha)$ is a Gibbs-Markov map with full branches
if 
\begin{enumerate}\label{def:gm-map}
\item $\theta([w])=X$ mod $\mu$, for all $w \in \cW^1$, and 
\item there exists $C>0 $, $r \in (0,1)$ such that, for all $w \in \cW^\infty$ and a.e. $x,y \in X$,
\[ \left|\log{\varphi_w(x)} - \log{\varphi_w(y)}\right| \leq C d_r(x,y).\] 
\end{enumerate}
\end{defn}
Furthermore, for each Gibbs-Markov map with full branches, there always exists an equivalent, $\theta$-invariant probability measure $\nu$ such that $\log d\nu/d\mu$ satisfies the above Hölder property and  $(X,\theta,\nu,\alpha)$ is a Gibbs-Markov map with full branches (see \cite{Aaronson--An-Introduction-To-Infinite--1997}).
However, only some Markov maps have this property, even though in many cases, there is an embedded Gibbs-Markov map. The following definition makes this remark precise. 
\begin{defn}\label{def:embedded Gibbs-Markov structure}
We refer to $(X,\theta,\mu,\alpha)$ as a Markov map with embedded Gibbs-Markov structure $(\Omega,\beta,\eta)$ if $\Omega \subset X$ and there exist $\beta \subset \cW^\infty$ and $\eta:\beta \to \N$  such that  
\begin{enumerate}
\item $\Omega$ is a finite union of elements of $\alpha$ {\color{black} up to a set of measure zero,} and the set $\left\{[u]: u \in \beta\right\}$ is a partition of\; $\Omega$ mod $\mu$,
\item the Markov map $(\Omega,\sigma,\nu,\beta)$ defined by $\sigma(x) := \theta^{\eta(u)}(x)$ for all $x \in [u]$ and $u \in \beta$ is a Gibbs-Markov map with full branches, where $\nu:=\left(\mu(\Omega)\right)^{-1}\mu{|_\Omega}$.
\end{enumerate}
Furthermore, we refer to the embedding as adequate if 
\begin{enumerate}
\item there exists a sequence $(C_n)$ such that $\lim_n C_n/n =0$ and
{\color{black}
\[| \log{\varphi_{w}(x)} - \log {\varphi_{w}(y)}|  \leq C_n,\]}
for all $[w] \in \alpha_{n}$ and $a \in \alpha$ and {\color{black} almost every pair $x,y \in [a]$ } such that $[wa]\neq  \emptyset$ and  $[w], [a]  \subset \Omega$ (in particular, $\theta^n([wa]) = [a] \subset \Omega$). In here, $d_\sigma$ refers to the metric with respect to $(\Omega,\sigma)$.
\item there exists an almost surely finite function $\eta^\dagger:\Omega \to \N$ such that, for almost every $x \in \Om$ and $l=0, \ldots , \eta(x)-1$ with $\theta^l(x)\in \Om$, we have that $\eta(x)- l \leq \eta^\dagger(\theta^l(x))$. 
\end{enumerate}
\end{defn}
\begin{remark} \label{remark:embedded vs immersed}
These rather technical definitions are motivated by dynamical systems that admit a tower construction with Markov partition, and the  
the significance of adequately embedded Gibbs-Markov systems will become visible in Proposition \ref{prop:pressure_T_vs_S} where the decay of return probabilities of extensions of the original and the embedded system are compared. Moreover, it is worth noting that the first condition of adequacy is related and motivated by the notions of weak Gibbs measures and medium variation (see \cite{Yuri--Thermodynamic-Formalism-For-Certain--ETDS1999,Kessebohmer--Large-Deviation-For-Weak--N2001}).
The tower constructions we have in mind range from first return maps, jump transformations and Schweiger collections (see \cite{Aaronson-Denker-Urbanski--Ergodic-Theory-For-Markov--TMS1993}) to Young towers and Pinheiro's general construction for expanding measures (\cite{Pinheiro--Expanding-Measures--AIHPANL2011}). 

For example, if $\sigma$ is the first return to a cylinder, then (i) in the definition of an adequate embedding follows from the Gibbs-Markov property of $\sigma$  with $C_n=C$, and (ii) always holds for $\eta^\dagger=\eta$ as $\eta(x) = l + \eta(\theta^l x)$ 
for first return times. In particular, if $\eta$ is the first return, then the embedding automatically is adequate.
In the context of hyperbolic and zooming times as defined in \cite[Def. 5.5]{Pinheiro--Expanding-Measures--AIHPANL2011}, the first condition  might not be always satisfied, but the  second condition holds as above for $\eta^\dagger=\eta$ by the construction of hyperbolic and zooming times through nested sets.

 However, it is sometimes advantageous to consider situations where $\eta^\dagger \geq \eta$, as, e.g., in the proof of Theorem \ref{theo:main theorem - embedded GM structure} or by considering a constant jump time, that is $\sigma = \theta^n$ for some $n \in \mathbb{N}$. In conclusion, it is worth noting that (ii) is a mild and technical condition, even though the property does not always hold. For example, if 
\[X = \Omega =  \{0,1\}^\mathbb{N}, \quad \beta :=\left\{ [0_n 1 w] : n\in \mathbb{N}  \cup \{0\},  w \in \cW^n \right\}, \]
where $0_n$ is the word $(00\ldots 0)$ of length $n$, then $\beta$ is a partition of $\Omega$ modulo the $(1/2,1/2)$-Bernoulli measure on $X$ and $\eta|_{[0_n 1 w ]} := 2n +1$ defines an embedded Gibbs-Markov map on $\Omega$. By construction, for any $x\in [0_n 1]$, we have $\eta(x) - (n+1) = n$. On the other hand, as $\theta^{n+1} ([0_n 1]) = \Omega$, the only function $\eta^\dagger$ with $\eta(x)- (n+1) \leq \eta^\dagger(\theta^{n+1}(x))$  is the constant function $\eta^\dagger = \infty$. {\color{black} Moreover, as the $(1/2,1/2)$-Bernoulli measure of $\{x: \eta(x) = 2n+1 \}$ decays exponentially in $n$, it follows that $\eta$ has exponential tails (see Definition \ref{def:exponential-tails} below). In particular, $\eta$ is a non-adequate jump time with exponential tails.
}
\end{remark}

\subsection{Extensions by graphs}
A (directed) graph is an ordered pair $\G = (\V,\E)$, where $\V$ is an at most countable set of \emph{vertices} and $\E \subset \V \times \V$ the set of  \emph{(directed) edges}. An edge $e= (g_1,g_2)\in \E$ might be seen as a link between the vertices $g_1,g_2 \in \V$. In this context, $s(e):=g_1$ is called the source and $t(e):=g_2$ is the target of $e$. This notion gives rise to the notions of paths and loops in $\G$. That is, $p:=(e_1 e_2\ldots e_n) \in \E^n$ is called a \emph{path of length $n$} from $s(p):=s(e_1)$ to $t(p):=t(e_n)$ if $t(e_i) = s(e_{i+1})$ for all $i =1,\ldots,n-1$. If $s(p)=t(p)$, then $p$ is called a \emph{loop}.

\begin{defn}\label{defn:graph_extension} Suppose that $(X,\te,\mu,\alpha)$ is a {Markov map} and $\G = (\V,\E)$ a graph. We refer to $\kappa:  X \times \V \to \V$, $(x,g) \mapsto \kappa_x(g)$ as a nearest neighbour cocycle if 
for all $x \in X$ and $g \in \V$, $\kappa_x : \V \to \V$ is a bijection,  $(g, \kappa_x(g)) \in \E$, and $\kappa_x$ is constant on cylinders, that is $\kappa_x = \kappa_y$ for all $x,y \in [v]$ and all $v \in \cW^1$.
The  extension $(Y,T,\kappa)$ by $\G$ through the nearest neighbour cocycle $\kappa$ is defined by, with $Y:= X \times \V$,
\[ T:  Y \to Y, \; (x,g) \mapsto (\te x, \kappa_x(g)).\]
\end{defn}
The space $Y$ is equipped with the canonical measure $\mu \otimes m_{\V}$, where $m_{\V}$ denotes the  counting measure on $\V$. Furthermore,  for $n \in \N$ and $w \in \cW^n$ and $x \in [w]$, we will write
\[ \kappa_w \equiv \kappa^n_x := \kappa_{\theta^{n-1}(x)} \circ \kappa_{\theta^{n-2}(x)} \circ \cdots \circ \kappa_{x}.  \]
Observe that an extension of a Markov map by a graph implicitly defines a coloring of $\E$ by referring to $w \in \cW$ as the \emph{color} of $(g,\kappa_w (g)) \in \E$. 
A fundamental concept in the proof of our main result is the existence of a loop of length one at each vertex, and that these loops can be chosen from a finite set of colors.  This condition is formalized as follows.

\begin{defn}\label{defn:graph_extension_uniform_loops}
The extension $(Y,T,\kappa)$ of a  Markov map  $(X,\theta,\mu,\alpha)$  with full branches has uniform loops if there exists a finite set 
such that $\mathcal{J} \subset \cW^\infty$ with $v \in \{ \kappa_x(v) : x \in \mathcal{J}\}$ for all $v \in \V$.
\end{defn}

It is worth noting that this condition appears to be quite restrictive. However, since Theorem \ref{theo:main_result} only requires the existence of an $n \in \mathbb{N}$ such that $T^n$ has uniform loops, it is actually sufficient to have loops of uniform length and color. Furthermore, in the context of embedded Gibbs-Markov structures, even the requirement of having loops of uniform length can be relaxed. 
We would also like to note that the motivation behind this condition is to allow us to work with general graphs. For example, if $G$ is the Cayley graph of a discrete semigroup, then it is sufficient to require the existence of a finite word $w$ such that $\kappa_w$ is a right unit (see Section \ref{subsec:Day}). In particular, in this case, the cardinality of $\mathcal{J}$ is one. 

\subsection{Amenability}

Before introducing amenability of a weighted graph, we recall the definition by Gerl (\cite{Gerl--Random-Walks-On-Graphs--JTP1988}, see also \cite{Gerl--Amenable-Groups-And-Amenable--1988}). 
Assume that $\G=(\V,\E)$ is a graph such that there exists a uniform bound on the number of adjacent edges of a vertex. For $K \subset \V$, $|K|< \infty$, the boundary of $K$ defined by
\[  \partial K:= \{ v \in K: \exists e \in \E \hbox{ s.t. } s(e) = v, t(e) \notin K   \} \] 
is then always  finite. The graph $\G$ is referred to as an \emph{amenable graph} if 
\[\inf\left\{   \frac{|\partial K|}{|K|}: {K \subset \V, |K|< \infty} \right\} =0.\] 
Since we also want to consider graphs who might contain vertices with infinitely many adjacent edges, we introduce the notion of amenability for weighted graphs. We refer to a graph $\G = (\V,\E)$  as a \emph{weighted graph} with weight $p:\E\to [0,1] $ if for all $v \in \V$, we have $ \sum_{e:s(e)=v} p(e) =1$. This then gives rise to the following boundary definition. For $\epsilon >0$ and $K \subset \V$,   the $\epsilon$-boundary of $K$ is defined by 
\[  \partial^\epsilon K:= \{ v \in K: \exists e \in \E \hbox{ s.t. } s(e) =v, t(e) \notin K, p(e) > \epsilon \}. \]  
\begin{defn} The weighted graph
$\G=(\V,\E)$ with weight $p$ is $p$-amenable if 
\[ \lim_{\epsilon \to 0} \inf\left\{   \frac{|\partial^\epsilon K|}{|K|}: {K \subset \V, |K|< \infty} \right\} =0. \]
\end{defn}
Observe that there is the following, equivalent definition. For a weighted graph $\mathcal{G}$ and $\epsilon >0$, let $\mathcal{G}_\epsilon =(\V,\E_\epsilon)$ be the graph with the same set of vertices $\V$ and edges $\E_\epsilon := \{e : p(e)> \epsilon\}$. We then have that $\G$ is an amenable weighted graph if and only if $\mathcal{G}_\epsilon$ is an amenable graph for all $\epsilon >0$. 
In particular, if $p$ is uniformly bounded away from $0$, then the  number of adjacent edges of a vertex is uniformly bounded from above and the notions of amenability for graphs and weighted graphs coincide. Furthermore, it is worth observing that by construction, edges with zero weight are irrelevant for $p$-amenability and might be removed from the graph without changing $p$-amenability.  

In the context of an extension of a Markov map with $\mu(X)=1$, we refer to the weight 
defined by $p(e):= \mu\{ x \in \Sig: \kappa_x(s(e))=t(e) \}$ as 
the \emph{canonical weight}. Moreover, if $\G$ is amenable with respect to this weight, we refer to  $\G$ as \emph{$\mu$-amenable}.
 Note that, if $|\cW^1|<\infty$ and if for each $e \in \E$, there exists $x \in X$ with $\kappa_x(s(e)) = t(e)$, then the notions of amenability for graphs and graphs with respect to the canonical weight  coincide since $p(e)$ is uniformly bounded away from $0$. 
We now show that amenability in fact only depends on $\kappa$. In order to do so, we use the idea of F\o lner sequences.
That is, we refer to a sequence $(K_n)$ of finite subsets of\/ $\V$ as a \emph{F\o lner sequence} with respect to $\kappa$ (or $\kappa$-F\o lner sequence) if, for all $v \in \cW^\infty$, 
\[ \lim_{n \to \infty} \frac{| \kappa_v(K_n)\setminus K_n|}{ | K_n|} =0.\] 

\begin{proposition} \label{prop:mu-amenable}
Suppose that $(Y,T,\kappa)$ is a graph extension of a topologically transitive Markov map $(X,T,\mu,\alpha)$.
Then $\G$ is $\mu$-amenable if and only if there exists a $\kappa$-F\o lner sequence. 
\end{proposition}

\begin{proof} We begin with the proof of the existence of a $\kappa$-F\o lner sequence. In order to do so, observe that, for $\epsilon >0$, we have $ |\partial^\epsilon K|  \geq |\kappa_v(K) \setminus  K| $
for each  $K \subset \V$ and  $v \in \cW^\infty$ with  $ \mu([v]) > \epsilon$. 
Hence, if $\G$ is $\mu$-amenable, then there exists $F_\epsilon\subset \V$ finite such that  
\[ {| \kappa_v(F_\epsilon)\setminus F_\epsilon|} \leq { |F_\epsilon|} \,\epsilon \quad \forall v \in \cW^1 \hbox{ with } \mu([v]) > \epsilon. \]
Set $K_n:= F_{1/n}$. As $\mu([v])>0$ for all $v \in \cW^1$ by topological transitivity,  $\lim_n  | \kappa_v(K_n)\setminus K_n|/| K_n |=0$ for all $v \in \cW^1$. We now prove by induction that  $(K_n)$ is a $\kappa$-F\o lner sequence. If the above property holds for finite words $u,v \in \cW^\infty$ and if  $w:=uv \in  \cW^\infty$,  then
\begin{align*}
 \frac{|\kappa_{w}(K_n)\setminus K_n|}{| K_n|} & \leq  \frac{|\kappa_{uv}(K_n)\setminus \kappa_v(K_n)|}{| K_n|} + \frac{|\kappa_{v}(K_n)\setminus K_n|}{| K_n|} \\
& \leq  \frac{|\kappa_{v}(\kappa_u(K_n)\setminus  K_n)|}{| K_n|} + \frac{|\kappa_{v}(K_n)\setminus K_n|}{| K_n|} \\
 & =  \frac{|\kappa_u(K_n)\setminus  K_n|}{|K_n|} + \frac{|\kappa_{v}(K_n)\setminus K_n|}{|K_n|}  \xrightarrow{n\to \infty} 0.
\end{align*}
Hence, by induction, $(K_n)$ is a $\kappa$-F\o lner sequence. On the other hand, if $(K_n)$ is $\kappa$-F\o lner sequence, then
\begin{align*}
 \frac{|\partial^\epsilon K_n|}{|K_n|} &= 
  \frac{\left|\left\{   g \in K_n: \exists v \in \cW^1 \hbox{ with } \mu([v])>\epsilon, \kappa_v(g) \notin K_n \right\}\right|}{|K_n|}
 \\
&   \leq \sum_{v \in \cW^1:  \mu([v])>\epsilon} \frac{\left|\kappa_{v}(K_n)\setminus K_n\right|}{| K_n|}
  \xrightarrow{n\to \infty} 0.
\end{align*}
Thus, $\G$ is $\mu$-amenable.
\end{proof}
\begin{remark} F\o lner's classical condition is given in terms of the symmetric difference of sets defined by  $A \triangle  B := (A\setminus B) \cup  (B\setminus A)$. In order to compare the definition above with  F\o lner's, first observe that for finite sets of the same cardinality, we always have that $ |A\setminus B| =|A| - |A\cap B| = |B\setminus A|$. As $\kappa_v$ is always a bijection on the set of vertices, it hence follows that  $(K_n)$ is a {F\o lner sequence} as defined above if and only if, for all $v$, 
\[ \lim_{n \to \infty} \frac{| \kappa_v(K_n)\triangle K_n|}{ | K_n|} =0.\] 
In particular, if $\kappa$ defines a topologically transitive extension by a Cayley graph of a finitely generated discrete group $G$, the notions of $\mu$-amenability for graphs and the classical notion of amenability of groups coincide. 
\end{remark}

\section{Spectral radius and amenability for full Markov maps}
We now relate amenability with the spectral radius  
for extensions of Gibbs-Markov maps with full branches, following closely  ideas in \cite{Stadlbauer--An-Extension-Of-Kestens--AM2013, Jaerisch--Group-Extended-Markov-Systems--PMS2015}. Therefore, we assume throughout  this section that $(X,\theta,\mu,\alpha)$ is a full branched {Gibbs-Markov map} with invariant probability $\mu$
which already satisfies the \emph{uniform loop property} as defined above.

We start with preparatory estimates based on these uniform loops for the norm $\| \cdot \|_2$ on $\ell^2(\V)$. 
In order to do so, for $f : \V \to \R$ and $w \in \cW^n$, set $f_w := f \circ \kappa_w^{-1}$. As $\kappa_w$ is a bijection, it follows that  $\|f\|_2 = \|f_w\|_2$  for all $f \in \ell^2(\V)$. 
\begin{lemma}\label{lemma:normdrop} Suppose that $T$ has uniform loops and let $f \in \ell^2(\V)$ with $\|f\|_2=1$. Then, 
\begin{equation}\label{eq:normdop}
\left\|  f-{\textstyle \sum_{u \in \mathcal{J}} f_u}  \right\|_2 \leq  \#\mathcal{J} - 1 .
\end{equation}
\end{lemma}
\begin{proof} 
Since $T$ has uniform loops,  for each $h\in \V$ there exists $u_0\in \mathcal{J}$ such that $\kappa_{u_0}^{-1}(h)=h$. Let  $ n:= \#\mathcal{J}-1$ and enumerate the  elements  $\kappa^{-1}_u(h)$, $u\in \mathcal{J} \setminus \{u_0\}$,  by  $h_i \in V$,  $i\in \Z/n\Z$.
Hence,  for each $h\in \V$,
\[\left(\sum_{u \in \mathcal{J}} f_u -f\right)(h)=\sum_{i\in \Z/n\Z} f(h_i).\] Consequently, we have
\begin{equation*}
\left\|  f-{\textstyle \sum_{u \in \mathcal{J}} f_u}  \right\|_2^2 =\sum_{h\in \V} \left(\sum_{i\in \Z/n\Z}f(h_i)\right)^2.
\end{equation*}

Define $\tilde{f}\in \ell^2(\V\times \Z/n\Z)$ given by $\tilde{f}(h,i):=f(h_i)$,  for $h\in \V$ and $i\in \Z/n\Z$. Then, 
we have
\begin{equation*}
\sum_{h\in \V} \left(\sum_{i\in \Z/n\Z}f(h_i)\right)^2=\sum_{h\in \V}\sum_{i\in \Z/n\Z} \sum_{k\in \Z/n\Z}\tilde{f}(h,i) \tilde{f}(h,i+k).
\end{equation*}
By the Cauchy-Schwarz inequality in $\ell^2(\V\times \Z/n\Z)$ we have
\begin{align*}
&\quad \sum_{h\in \V}\sum_{i\in \Z/n\Z} \sum_{k\in \Z/n\Z}\tilde{f}(h,i) \tilde{f}(h,i+k) \\
&\le \sum_{k\in \Z/n\Z}\sqrt{\left(\sum_{h\in \V}\sum_{i\in \Z/n\Z} \tilde{f}(h,i)^2\right)\left(\sum_{h\in \V}\sum_{i\in \Z/n\Z} \tilde{f}(h,i+k)^2\right)}\\
& = n \langle \tilde f,\tilde f \rangle. 
\end{align*}
Since each vertex  appears at most $n$-times in the family $(h_i)_{h\in \V, i\in \Z/n\Z}$,
we have 
$\left\langle \tilde{f},\tilde{f}\right\rangle \le 
 n \langle f,f \rangle$,
which completes the proof of the lemma.
\end{proof}
 
The following simple lemma plays a key role for a bound on the spectral radius of $\widehat{T}$ as it allows to replace an estimate from below of a difference by an estimate from above of a sum. 

\begin{lemma}\label{lemma:rotundity} Suppose that $T$ has uniform loops and that there exist $\epsilon>0$, $f \in \ell^2(\V)$ with $\|f\|_2=1$ and $w \in \cW$ such that $\|f - f_w\|_2 > \epsilon$. Then, for $\delta:= 2 - \sqrt{4-\epsilon^2}$, we have 
\begin{equation}\label{eq:local_contraction}
\left\| f_w + {\textstyle \sum_{u \in \mathcal{J}} f_u}  \right\|_2 \leq 1+ \#\mathcal{J} - \delta .
\end{equation}
\end{lemma}

\begin{proof} By the parallelogram law in $\ell^2(\V)$, we have
\[ \|f + f_w\|_2^2 \leq 2(\|f\|_2^2 + \|f_w\|_2^2) - \|f - f_w\|_2^2 \leq 4 - \epsilon^2 = (2-\delta)^2. \]
Lemma \ref{lemma:normdrop} then implies that
\[  \left\| f_w + {\textstyle \sum_{u \in \mathcal{J}} f_u}  \right\|_2 \le  \left\| f_w +f  \right\|_2 + \left\|  \left( {\textstyle \sum_{u \in \mathcal{J}} f_u}\right)  -f  \right\|_2  \le 2-\delta+ \#\mathcal{J} - 1= 1+ \#\mathcal{J} - \delta. \]
\end{proof}

After these basic but fundamental considerations, we now focus on the functional analytic properties of the transfer operator. Recall that the  transfer operator  in ergodic theory is defined as the dual of the Koopman operator.
That is, the transfer operator
$\widehat{\te} : L^1(\mu) \to L^1(\mu)$ is defined by
\[ \int \widehat{\te}(f) g d\mu = \int f (g\circ \te) d\mu, \quad \forall f \in L^1(\mu), g \in L^\infty(\mu). \]   
In case of a Markov map, it is well known that the transfer operator can be identified with Ruelle's operator for the potential 
$\varphi=  d\mu/d\mu\circ \te$, that is, for $f_w := f \circ \tau_w$ 
\[\widehat{\te}(f)(x) = \sum_{w \in \cW^1}  \varphi_w(x) f_w(x).  \] 
Observe that $\widehat{\te}(1)=1$ as $\mu$ is $\te$-invariant. The transfer operator $\widehat{T}$ is defined in the same way and, as the product of $\mu$ and the counting measure is $T$-invariant, also satisfies $\widehat{T}(1)=1$, where $1$ this time is the constant function one on $Y$. 
We now introduce the relevant function spaces with respect to the spectrum of $\widehat{T}$. 
In order to do so, set ${X_g}:= \{ (x,g) : x \in X \}$, for $ g \in \V$.  Furthermore, with $\| \cdot \|_p$ referring to the $L_p(X,\mu)$-norm and for $p=1,\infty$, $f:Y \to \R$ measurable, $g \in \V$, set 
\begin{align*} 
 \llbracket f\rrbracket_p:= \sqrt{{\textstyle\sum_{g \in \V} (\|f(\cdot\ ,g)\|_p)^2 } } \quad \hbox{ and } \quad   \cH_p:= \{f: Y \to \R\with \llbracket f \rrbracket_p < \infty \}.
\end{align*}
Furthermore, set $\cH_c := \{f \in \cH_\infty \with f \hbox{ is constant on } {X_g} \forall g \in \V\}$. Define  \[ \Lambda_k := \sup\left(\left\{  \llbracket \widehat{T}^k(f)  \rrbracket_1 /  \llbracket f \rrbracket_1  \with f \geq 0, f\in \cH_c \cap \cH_1 \right\}\right). \] 
By the same arguments as in \cite{Stadlbauer--An-Extension-Of-Kestens--AM2013} and \cite[Lemma 3.2(3)]{Jaerisch--Group-Extended-Markov-Systems--PMS2015}, one obtains the following for the action of $\widehat{T}$ on $\cH_1$ and $\cH_\infty$. 
\begin{proposition}\label{prop:operator_on_cH} The function spaces $(\cH_1, \llbracket \cdot \rrbracket_1)$ and $(\cH_\infty, \llbracket \cdot \rrbracket_\infty)$ are Banach spaces, the operators $\widehat{T}^k: \cH_\infty \to \cH_\infty$ are bounded and there exists $C\geq1$ such that   $\llbracket\widehat{T}^k\rrbracket_\infty \leq C$ for all $k\in \N$.
Furthermore, $\Lambda_k \le 1$ for all $k\in \N$ and $\lim_{k \to \infty} (\Lambda_k)^{1/k} =\rho( \widehat{T})$, with $\rho( \widehat{T})$ referring to the spectral radius of  $\widehat{T}: \cH_\infty \to \cH_\infty$.
\end{proposition}
Hence, by the above, $\widehat{T}$ acts continuously on $\cH_\infty$, by Gelfand's formula for the spectral radius, $\rho(\widehat{T}) \leq 1$ and   $\widehat{T}: \cH_c \to \cH_1$ is continuous with norm smaller than or equal to 1. However, $\widehat{T}(\cH_c) \not\subset \cH_c$ and $\widehat{T}$ does not necessarily act on $\cH_1$ as a bounded operator. Furthermore, by definition as transfer operator, $\widehat{T}$ acts on $L^1(Y,\mu)$ as an isometry and, in particular, the spectral radius on this space always is equal to one. The following theorem relates amenability of $\mathcal{G}$ with the action of $\widehat{T}$ on these spaces.  
 
\begin{theorem} \label{theo:main_result}
Suppose that  $(Y,T,\kappa)$ is a topologically transitive extension of a Gibbs-Markov map  $(X,\te,\mu,\alpha)$ with full branches and invariant probability $\mu$. Moreover, assume that there exists $n \in \N $ such that $T^n$ has uniform loops. Then the following assertions are equivalent.\begin{enumerate}
\item The graph  $\G = (\V,\E)$ is $\mu$-amenable.
\item The spectral radius $\rho(\widehat{T})$ of $\widehat{T}: \cH_\infty \to \cH_\infty$ is equal to 1.
\item For each $\epsilon >0$, there exists  $A \subset \V$ finite such that 
$ \int | \widehat{T}(\1_{X \times A}) - \1_{X \times A} | d \mu \leq \epsilon \cdot \#(A)$.
\item For each $\epsilon >0$, there exists  $A \subset \V$ finite such that 
$   \llbracket \widehat{T}(\1_{X \times A}) - \1_{X \times A} \rrbracket_1 \leq \epsilon \llbracket  \1_{X \times A} \rrbracket_1 $.
\end{enumerate}
\end{theorem}

\begin{proof} We begin with the proof of the theorem for $n=1$, that is $T$ has uniform loops and deduce the general case from this result in Step 7 below. The principal part of the proof is to show
 that $\rho(\widehat{T})=1$ implies that for each $\epsilon >0$, there exists a $A \subset \V$ finite such that $\int | \widehat{T}(\1_{X \times A}) - \1_{X \times A} | d \mu \leq \epsilon \cdot \#(A)$. In order to do so, we first show that, for each $\epsilon >0$, there exists $f \in \cH_c$, $\llbracket f \rrbracket_1 =1$ and $f \geq 0$ such that $\llbracket \widehat{T}(f) - f \rrbracket_1< \epsilon$. In order to prove this by contradiction, let $\mathcal{J}\subset \cW$ be given by the uniform loop property and suppose that 
\begin{equation}  \label{eq:no_almost_eigenfunction} \delta_1 := \inf\left\{\llbracket \widehat{T}(f) - f \rrbracket_1 /\llbracket f \rrbracket_1 \with f \in \cH_c, f\geq 0, f \neq 0 \right\} >0. \end{equation}
\noindent \textsc{Step 1}. We begin with an application of Lemma \ref{lemma:rotundity}, that is, we construct a finite subset $\cW^\ast \subset \cW$  such that, for $f \in \cH_c$ with $\llbracket f \rrbracket_1=1$, there exists $w \in \cW^\ast$ satisfying inequality (\ref{eq:local_contraction}) together with a control of the distortion.  
In order to do so, choose $\cW^\ast \subset \cW^1$, $\mathcal{J} \subset \cW^\ast$ such that $\sum_{w \in \cW \setminus \cW^\ast} \mu([w]) \leq {\delta_1}/{4}$.   
Using $\widehat{T}(1)=1$ and the $\triangle$-inequality then gives 
\begin{align*}
\delta_1 & \leq \llbracket \widehat{T}(f) - f \rrbracket_1 = 
\llbracket \sum_{w \in \cW} \Phi_w (f_w - f) \rrbracket_1 \\
& \leq  \sum_{w \in \cW^\ast} \llbracket \Phi_w (f_w - f) \rrbracket_1 + 
\llbracket \sum_{w \notin \cW^\ast}  \Phi_w f_w \rrbracket_1 +  \llbracket \sum_{w \notin \cW^\ast} \Phi_w f \rrbracket_1  \\
& \leq   \sum_{w \in \cW^\ast} \mu([w]) \llbracket (f_{w} - f) \rrbracket_1 + \delta_1/2.
\end{align*}
 Hence there exists $w \in \cW^\ast$ such that 
$\delta_1/2 \leq  \llbracket (f_{w} - f)  \rrbracket_1 
$. 
Identifying $\cH_c$ with $\ell^2(\V)$,  Lemma \ref{lemma:rotundity} implies  the inequality (\ref{eq:local_contraction}) with respect to 
 $\delta_2:= 2 - \sqrt{4-\delta_1^2}$. 
That is, for each $f \in \cH_c$, $f\geq 0$, there exists $w_f \in \cW^\ast$ with     
\begin{equation*}
\llbracket f_{w_f} + {\textstyle \sum_{u \in \mathcal{J}} f_u}  \rrbracket_1 \leq (1+ \#\mathcal{J} - \delta_2) \llbracket f \rrbracket_1.
\end{equation*}
As $\theta$ has full branches, this gives rise to a uniform estimate with respect to $\cW^\dagger := \cW^\ast \cup \mathcal{J}$ and $K:= \#(\cW^\dagger)$. Namely, for $f \in \cH_c$, $f\geq 0$, we have  
\begin{equation}\label{eq:local_contraction_2}
\llbracket {\textstyle \sum_{u \in \cW^\dagger} f_u}  \rrbracket_1 \leq 
\llbracket f_{w_f} + {\textstyle \sum_{u \in \mathcal{J}} f_u}  \rrbracket_1  + \llbracket {\textstyle \sum_{u \in  \cW^\dagger \setminus (\mathcal{J} \cup \{w_f\})} f_u}  \rrbracket_1
\leq (K - \delta_2) \llbracket f \rrbracket_1 .
\end{equation}
In order to control distortion, note that by Hölder continuity of $ \log \Phi_w$, there exists $m \in \N$
such that, for all $x,y \in X$, $a \in \cW^m$, and $w \in \cW^\infty$, 
\begin{equation}
\label{eq:absorbing_hoelder_coeficients}
  \sqrt{ 1- \delta_2/K }  \leq \frac{\Phi_w(\tau_{a}(x))}{\Phi_w(\tau_{a}(y))} \leq \left(  \sqrt{ 1- \delta_2/K }  \right)^{-1}. 
\end{equation} 
We now fix $a \in \cW^m$ and set $\cW^\ddagger := \left\{ (au) \in \cW^{m+1}: u \in \cW^\dagger \right\}$. By \eqref{eq:local_contraction_2}, we have
\begin{equation}\label{eq:local_contraction_pt2}
\left \llbracket  \sum_{v \in \cW^\ddagger} {f}_{v} \right \rrbracket_1 
= \left\llbracket \sum_{u \in \cW^\dagger}  ({f_a})_u  \right\rrbracket_1
\leq (K - \delta_2) \llbracket f \rrbracket_1. 
\end{equation}

\noindent \textsc{Step 2}. We deduce pointwise exponential decay of $\llbracket \widehat{T}^{m+1}(f) \rrbracket_1$, for $f\in \cH_c$,  from (\ref{eq:absorbing_hoelder_coeficients}) and (\ref{eq:local_contraction_pt2}). 
For a given $f \in \cH_c$ and $(k+1)$ finite words $w_i \in \cW^{n_i},\,\,n_i \ge 1$ ($i=0,1,\ldots,k$), we define for  $j=0,1,\ldots,k$, 
\begin{align*}
\cW_{j} := \left\{ (w_0 v_1 w_1 v_2 \ldots v_j w_{j})   \with  v_i \in \cW^\ddagger \hbox{ for } i=1, \ldots, j  \right\}, \quad 
f_{j} := \sum_{w \in \cW_{j}} f_w.
\end{align*} 

Using $\llbracket f_v\rrbracket_1 = \llbracket f\rrbracket_1$ and $f_{vw}=(f_v)_w$  for $v,w\in \cW^\infty$ and $f\in \cH_c$,  it follows inductively from (\ref{eq:local_contraction_pt2})  that
\begin{align} 
\nonumber 
\left\llbracket f_k  \right\rrbracket_1 &= 
\left\llbracket \sum_{w \in \cW_k} f_w \right\rrbracket_1  =  
\left\llbracket \sum_{w \in \cW_{k-1}, v \in \cW^\ddagger} (f_{wv})_{w_k} \right\rrbracket_1   = 
\left\llbracket \sum_{w \in \cW_{k-1}, v \in \cW^\ddagger} f_{wv} \right\rrbracket_1 \\
\label{eq:decay-along-orbits}
& = 
\left\llbracket \sum_{v \in \cW^\ddagger} (f_{k-1})_{v} \right\rrbracket_1 
\leq (K - \delta_2) \left\llbracket f_{k-1}   \right\rrbracket_1 \leq ( K - \delta_2)^k  \llbracket  {f}    \rrbracket_1.
\end{align}

As $ \cW^\ddagger$ is finite, $\alpha:= \inf\{\Phi_w(x): w \in \cW^\ddagger, x \in X \} >0$. By dividing  
each $v \in \cW^\ddagger$ into two words $v_1,v_2$ of length $m+1$ and defining ${\Phi}_{m+1}(x):= \alpha $ for $x\in [v_1]$ and ${\Phi}_{m+1}(v_2x) :=  {\Phi}_{m+1}(vx) - \alpha$ for $(v_2x) \in [v_2]$, we may and do assume in the next step that $\Phi_w=\alpha$ for all $w \in \cW^\ddagger$.  The above estimates in \eqref{eq:absorbing_hoelder_coeficients} and \eqref{eq:decay-along-orbits} now imply that, for $l:=m+1$, $\delta_3 := 1 - \sqrt{1- {\delta_2}/{K}}$ and each $x \in X$ and $f \in \mathcal{H}_c$ with $f \geq 0$ and $\llbracket f \rrbracket_1< \infty$,

\begin{align}
\nonumber 
 \left(\sum_{v\in \V} \left(\widehat{T}^{l n}(f)(x,v)\right)^2 \right)^{1/2}
 = & \left\llbracket \sum_{J \subset \{1,\ldots,n\}} \;
 \sum_{\forall j \notin J: w_j \notin \cW^\ddagger} \; \sum_{\forall j \in J: w_j \in \cW^\ddagger}  
 {\Phi}_{w_1\ldots w_n}(x) 
 f_{w_1\ldots w_n}   \right\rrbracket_1 \\
\nonumber
\leq & \sum_{\genfrac{}{}{0pt}{}{J \subset \{1,\ldots,n\}}{\forall j \notin J: w_j \notin \cW^\ddagger} }  
\left\llbracket \sum_{\forall j \in J: w_j \in \cW^\ddagger}   {\Phi}_{w_1\ldots w_n}(x) f_{w_1\ldots w_n}  \right\rrbracket_1\\
\nonumber
\leq &   
 \sum_{\genfrac{}{}{0pt}{}{J \subset \{1,\ldots,n\}}{\forall j \notin J: w_j \notin \cW^\ddagger} } 
 \left(\max_{\forall j \in J: w_j \in \cW^\ddagger} {\Phi}_{w_1\ldots w_n}(x) \right)
 \left\llbracket \sum_{\forall j \in J: w_j \in \cW^\ddagger} f_{w_1\ldots w_n}  \right\rrbracket_1 \\
\nonumber
\leq &   
 \sum_{\genfrac{}{}{0pt}{}{J \subset \{1,\ldots,n\}}{\forall j \notin J: w_j \notin \cW^\ddagger} } 
 \left(K^{\#J} \max_{\forall j \in J: w_j \in \cW^\ddagger} {\Phi}_{w_1\ldots w_n}(x) \right)
( 1 - \delta_2/K)^{\#J} \llbracket f \rrbracket_1\\ 
\nonumber
\leq & 
 \sum_{\genfrac{}{}{0pt}{}{J \subset \{1,\ldots,n\}}{\forall j \notin J: w_j \notin \cW^\ddagger} } 
 \sum_{\forall j \in J: w_j \in \cW^\ddagger} {\Phi}_{w_1\ldots w_n}(x) 
(1 - \delta_3)^{\#J}   \llbracket f \rrbracket_1  \\
\nonumber
 = &\sum_{J \subset \{1,\ldots,n\}} \sum_{\genfrac{}{}{0pt}{}{\forall j \notin J: w_j \notin \cW^\ddagger}{\forall j \in J: w_j \in \cW^\ddagger} } {\Phi}_{w_1\ldots w_n}(x) (1 - \delta_3)^{\#J}  \llbracket f \rrbracket_1 \\
 = &\sum_{J \subset \{1,\ldots,n\}} (1-\alpha)^{n - \#J} \alpha^{\#J} (1 - \delta_3)^{\#J}  \llbracket f \rrbracket_1
\label{eq:estimate-contraction} 
 = (1- \alpha \delta_3)^n \llbracket f \rrbracket_1.
\end{align}

It now follows  from Jensen’s inequality that $\llbracket \widehat{T}^{ln}(f)\rrbracket_1 \leq (1- \alpha \delta_3)^n \llbracket f \rrbracket_1$. As $\llbracket f \rrbracket_1 = \llbracket f \rrbracket_\infty$ for $f \in \mathcal{H}_c$, it follows from   Lemma \ref{lemmaTn}(i) and Gelfand's formula that $\rho(\widehat{T})\leq ({1-\alpha\delta_3})^{1/l}<1$, which is a contradiction.

\medskip
\noindent \textsc{Step 3}. We now fix $\epsilon>0$.  It follows from the above that there exists 
$f \in \cH_c$ with $\llbracket f \rrbracket_1=1$,  $f \geq 0$, and 
 $ \llbracket \widehat{T} (f) - f \rrbracket_1 \leq \epsilon$. Observe that, as $\llbracket f \rrbracket_1<\infty$, we may assume without loss of generality, that there exists a finite set $B\subset \V$ such that $f$ is supported on $X\times B$. This implies that there exists a finite, increasing sequence of finite subsets $A_1 \subset A_2 \subset \cdots \subset A_k$ of $\V$ and $\lambda_1, \lambda_2 , \ldots, \lambda_k >0 $ such that 
 $
 f= \sum_{i=1}^k \lambda_i \1_{X \times A_i}
 $. 
By monotonicity, we have 
\[(X \times A_i\setminus T^{-1}(X \times A_i)) \cap (T^{-1}(X \times A_j) \setminus (X \times A_j)) = \emptyset,\]
 for all $1 \leq i,j \leq n$. Also note that $ \tau_v(x) \in T^{-1} (X \times A_i)$ for a $v \in \cW^1$  implies that $ \tau_w(x) \in T^{-1} (X \times A_i)$ for all $w \in \cW$. Hence, with $\triangle$ referring to the symmetric difference, 
\begin{align*}
|\widehat{T}(f) - f| &= \left| {\sum}_i \lambda_i \left(\1_{X \times A_i}\widehat{T}\left(\1_{X \times A_i} - \1\right) +  \1_{X \times A_i^c} \widehat{T}\left(\1_{X \times A_i}\right)\right)  \right|\\
& = \left| {\sum}_i \lambda_i \left(  
 \widehat{T}\left(\1_{(X \times A_i) \setminus T^{-1}(X \times A_i)}\right)
 - \widehat{T}\left(\1_{T^{-1}(X \times A_i) \setminus (X \times A_i) }\right)
  \right)  \right| \\
& = {\sum}_i \lambda_i 
 \widehat{T}\left(\1_{(X \times A_i) \triangle T^{-1}(X \times A_i)}\right) =  \widehat{T}\left(|f - f\circ T|\right) .	 
\end{align*}
Hence, for each $v \in \cW$, we have that $\mu([v]) \llbracket  {f}_v -  {f} \rrbracket_1 \leq 
 \llbracket \widehat{T}\left( f \right) - f  \rrbracket_1 \leq \epsilon$. In particular, for each finite subset $\cW^\#$ of $\cW$, there exists $f$ such that  $\llbracket  {f}_v -  {f} \rrbracket_1$ is uniformly arbitrarily small for $v \in \cW^\#$ in this finite subset. We now choose $\cW^\#$ finite such that $\sum_{w \notin \cW^\#}\mu([w])< \epsilon/4$, and $f$ as above with the additional property that $ \llbracket  {f}_v -  {f} \rrbracket_1 \leq \epsilon/2$ for all $v \in \cW^\#$. However, we now suppose that the $\lambda_i$ are chosen such that  $f^2= \sum_{i=1}^k \lambda_i \1_{X \times A_i}$. For $f\in \cH_c$ we use  $\hat{f}:\V \to \V $ to denote the function given by $\hat{f}(g):=f(x,g)$.  By the same argument based on $A_i \subset A_{i+1}$ as above, one obtains that for each $v \in \cW$,
 \begin{align} \nonumber
 \|\hat{f}^2 \circ \kappa_v^{-1} -  \hat{f}^2 \|_{1} &
 = \sum_{g\in \V} \left| \sum_{i=1}^n \lambda_i \1_{\kappa_v(A_i)}(g) - \1_{A_i}(g)   \right|
 = \sum_{g\in \V}  \sum_{i=1}^n \lambda_i \1_{\kappa_v(A_i)\triangle A_i}(g) \\
 \label{eq:counting_symmetric_diferences}
 &  =  \sum_{i=1}^n \lambda_i \#(\kappa_v(A_i)\triangle A_i).
 \end{align}
On the other hand, observe that 
 $\|\hat f_1^2 - \hat f_2^2\|_1 \leq \|\hat f_1 + \hat f_2\|_2 \cdot \|\hat f_1 - \hat f_2\|_2$ is an easy consequence of the Cauchy-Schwarz inequality, which implies by the choice of $f$ that the right hand side of \eqref{eq:counting_symmetric_diferences} is smaller than or equal to $\epsilon$  for all $v \in \cW^\#$. It is then easy to see that 
\begin{align*}
    \sum_{i=1}^n  \lambda_i \sum_{v \in \cW}  \mu([v]) \; \#({\kappa_v(A_i)\triangle A_i}) \leq 2\epsilon.
\end{align*}
It now follows from $1 = \llbracket f \rrbracket_1^2 = \sum_{i} \lambda_i \#A_i $  that there has to exist $A \in \{A_1, \ldots, A_n\}$ with
\[   \sum_{v \in \cW}  \mu([v]) \; \#({\kappa_v(A)\triangle A }) \leq  2 \epsilon (\# A ). \]
In order to deduce (iii) from the estimate, note that it follows from the above, that 
\begin{align} \label{eq:integrating_out}
\sum_v \varphi_v(x) |\1_{\kappa_v(A)}(g) - \1_{A}(g)| = |\widehat{T}(\1_{X \times A})(x,g) - \1_{X \times A}(x,g)|\end{align}
for all $(x,g) \in X \times \V$. By integrating over $X \times \V$, we hence obtain part (iii) of the theorem, that is  $\int | \widehat{T}(\1_{X \times A}) - \1_{X \times A} | d \mu \leq \epsilon \cdot \#(A)$.

\medskip
\noindent \textsc{Step 4: (i) $\iff$ (iii)}. We now show that amenability and part (iii) are equivalent. In order to do so, first note that \eqref{eq:counting_symmetric_diferences} and \eqref{eq:integrating_out}  imply that   
\begin{align} \label{eq:relating-l_1-and-counting}
\int | \widehat{T}(\1_{X \times A}) - \1_{X \times A} | d \mu =  \sum_{v \in \cW}  \mu([v]) \; \#({\kappa_v(A)\triangle A }),
\end{align} 
and that, by bijectivity of $\kappa_v$,  $\#(A \setminus \kappa_v(A)) = \#(\kappa_v(A) \setminus A)$. Furthermore, observe that
 $e \in \partial A$ if and only if, by definition $s(e) \in A$ and $t(e) \notin A$, which is equivalent to the existence of $v \in \cW$ such that, for $g=s(e) \in A$, $t(e)=\kappa_v(g)$ and $\kappa_v(g) \in \kappa_v(A) \setminus A$. 
Hence, for $\epsilon >0$,  
\[2 \# \partial^\epsilon A \leq 2 \sum_{v : \mu([v])>\epsilon}   \# (\kappa_v(A) \setminus A) = \sum_{v : \mu([v])>\epsilon}   \# (\kappa_v(A) \triangle A).\] 
In particular, part (iii) of the theorem implies amenability. In order to obtain the reverse direction, observe that  
weighted amenability allows to find $\epsilon$ and $A$ such that the right hand side of \eqref{eq:relating-l_1-and-counting} divided by $\#A$ is arbitrarily small.    

\medskip
\noindent \textsc{Step 5: (iii) $\Rightarrow$ (iv)}. Suppose that $\int | \widehat{T}(\1_{X \times A}) - \1_{X \times A} | d \mu \leq \epsilon \cdot \#(A)$. As $| \widehat{T}(\1_{X \times A}) - \1_{X \times A} | = \widehat{T}(\1_{(X \times A) \triangle T^{-1}(X \times A)}) \leq 1$, we have by Jensen's inequality
\begin{align*} 
\epsilon \llbracket \1_{X \times A}) \rrbracket_1^2 & = \epsilon  \cdot \#(A) \geq \int | \widehat{T}(\1_{X \times A}) - \1_{X \times A} | d \mu 
 = \sum_{g \in \V} \int_{X_g}  | \widehat{T}(\1_{X \times A}) - \1_{X \times A} | d \mu  \\
 & \geq  \sum_{g \in \V}  \int_{X_g}  | \widehat{T}(\1_{X \times A}) - \1_{X \times A} |^2 d \mu \ge \llbracket \widehat{T}(\1_{X \times A}) - \1_{X \times A} \rrbracket_1^2.
\end{align*}
Hence, $\llbracket \widehat{T}(\1_{X \times A}) - \1_{X \times A} \rrbracket_1 \leq \sqrt{\epsilon}   \llbracket \1_{X \times A}) \rrbracket_1 $.  

\medskip
\noindent \textsc{Step 6: (iv) $\Rightarrow$ (ii)}.
 For each $\epsilon>0$ there exists  $f \in \cH_c$ with $\llbracket f \rrbracket_1 =1 $
 and $\llbracket  \widehat{T}(f) -  f \rrbracket_1 \leq \epsilon$. By Lemma 3.5 (i) there exists a uniform constant $ C\ge 1$ such that $\llbracket  \widehat{T}(f) -  f \rrbracket_\infty \leq C \epsilon$. It follows that $\widehat{T}(f) -  f$ has no bounded inverse in $\cH_\infty$ and therefore, $\rho(\widehat{T})\ge 1$. By Proposition 3.3 this implies $\rho(\widehat{T})= 1$.

\medskip
\noindent \textsc{Step 7: The case $n\geq 2$}. Note that only Steps 1 and 2 of the proof above rely on 
uniform loops at time $n=1$. Hence, in order to prove the theorem for $n \geq 2$, it remains to show that $\rho(\widehat{T})=1$ implies that (iii) holds. In order to do so, first observe that $\rho(\widehat{T}^n) = \rho(\widehat{T})^n$ by Gelfand's formula for the spectral radius and 
hence, that $\rho(\widehat{T})=1$ implies that $\rho(\widehat{T}^n)=1$.

Furthermore, as  $(Y,S)$ is topologically transitive, there exists $p \in \N$, the period of $S$, and a decomposition $Y_1, \ldots Y_p$ of $Y$, measurable with respect to $\alpha$ such that $T(Y_{k}) = Y_{k+1}$
and $T^p: Y_k \to Y_k$ is topologically mixing, for any $k\in \Z/p\Z$. However, as $\theta$ is the full shift, it follows that each $Y_k$ is of the from $X \times \V_k$, where $\V_1, \ldots \V_p$ is a decomposition of $\V$. Furthermore, if 
$T^n$ has uniform loops, then $Y_k = Y_{k+n}$, which implies that $n$ is a multiple of $p$. In particular, the result for uniform loops at time 1 applied to $T^n: X \times \V_0 \to X \times \V_0$ implies that, for a given $\epsilon > 0$, there exists $A_0 \subset \V_0$ such that $\int | \widehat{T}^n(\1_{X \times A_0}) - \1_{X \times A_0} | d \mu \leq \epsilon \cdot \#(A_0)$. 

Now assume that $k+l =n$. We now approximate $\widehat{T}^k(\1_{X \times A_0)}$ by $\1_{X \times A_k} $ for some suitable $A_k \subset \V_k$. In order to do so, first observe that  \eqref{eq:counting_symmetric_diferences} applied to a single set $X \times A_0$ implies
\begin{align*}
 \epsilon \cdot \#(A_0) & \geq \int \left| \widehat{T}^n(\1_{X \times A_0}) - \1_{X \times A_0} \right| d \mu  
 =   \int  \widehat{T}^{l+k} (|\1_{(X \times A_0) \triangle T^{-k-l}(X \times A_0}| ) d \mu \\
 & 
  =  
   \int  \widehat{T}^k(|\1_{(X \times A_0)  \triangle T^{-k-l}(X \times A_0}| ) d \mu 
  =   \int \left| \widehat{T}^k(\1_{X \times A_0} -   \1_{T^{-l}(X \times A_0)} )\right| d\mu.
\end{align*}
Hence, in average, $\widehat{T}^k(\1_{X \times A_0})$ behaves almost like an indicator function. In order to show, that this indicator function might be chosen to be of the form $\1_{X \times A_k}$ define, for $\delta >0$,
\begin{align*}
A_k & := \left\{g \in\V: \widehat{T}^k(\1_{X \times A_0})(x,g) \; \geq \; 1-\delta \forall x \in X \right\}, \\
B_k & := \left\{g \in\V: \widehat{T}^k(\1_{X \times A_0})(x,g) \; \leq \; \delta \forall x \in X \right\}.  
\end{align*}
By choosing $ \delta < 1/2$, we obtain that  $A_k \cap B_k = \emptyset$.
Furthermore, as $\theta$ is a Gibbs-Markov map, it follows that $ \widehat{T}^k(\1_{X \times A_0})(x,g)=c^{\pm 1}  \widehat{T}^k(\1_{X \times A_0})(y,g)$ for some $c>0$. Combining this fact with $ \widehat{T} \1 = \1$, it follows  for any $x \in X$ and  $g \in (A_k \cup B_k )^c$ that $ \delta/c < \widehat{T}^k(\1_{X \times A_0})(x,g) < 1- \delta/c$. By dividing the integral above into three parts, it follows that 
\begin{align*}
 \epsilon \cdot \#(A_0) & \geq \sum_{g \in A_k} \int \left| \widehat{T}^k(\1_{X \times A_0}) -   \1_{T^{-l}(X \times A_0)} \right| d\mu  +  \sum_{g \in B_k} \int \left| \widehat{T}^k(\1_{X \times A_0}) -   \1_{T^{-l}(X \times A_0)} \right| d\mu \\
  & \phantom{\geq} +  \sum_{g \notin A_k \cup B_k} \int \left| \widehat{T}^k(\1_{X \times A_0}) -   \1_{T^{-l}(X \times A_0)} \right| d\mu \\
  & \geq  (1-\delta) \mu( (X\times A_k) \cap T^{-l}((X\times A_0)^c)) + 
  (1-\delta) \mu( (X\times B_k) \cap T^{-l}(X\times A_0)) \\
    & \phantom{\geq}  + \frac{\delta}{c} \mu(X\times (A_k \cup B_k)^c).
 \end{align*}
Furthermore, the above estimate implies for $\tilde{c} := ({2+c})/{\delta}$ that 
\begin{align*} 
 \int |\1_{X \times A_k} -  \1_{T^{-l}(X \times A_0)}| d\mu  & 
\leq \mu((X\times A_k)  \cap  T^{-l}(X\times A_0^c) )  + \mu((X\times B_k)  \cap  T^{-l}(X\times A_0) ) \\
& \phantom{\leq } +  \mu(X\times (A_k \cup B_k)^c) 
\\ 
& \leq \frac{\epsilon}{1-\delta} \#(A_0) + \frac{\epsilon}{1-\delta} \#(A_0) + \frac{c \epsilon}{\delta} \#(A_0) \leq 
\tilde{c} \epsilon \cdot \#(A_0).
\end{align*} 
In particular, 
$\#(A_k) = \mu(X\times A_k) =  \mu(T^{-l}(X\times A_0)) \pm  \tilde{c}\epsilon \cdot \#(A_0) = \left(1 \pm  \tilde{c} \epsilon \right)  \#(A_0)$. This then implies from the estimate below that 
$ f:= \sum_{k=0}^{n-1} \1_{X\times A_k}$ is an almost eigenfunction. That is, considering $k$ as element of $\Z/n\Z$ and employing $ \widehat{T}(g\circ T ) = g  \widehat{T}(\1) = g$, 
\begin{align*}
\int |\widehat{T}(f) - f|d\mu & \leq \sum_{k=0}^{n-1} \int |\widehat{T}(\1_{X\times A_k}) - \1_{X\times A_{k-1}}|d\mu \\
& \leq  \sum_{k=0}^{n-1} \int |\widehat{T}(\1_{X\times A_k} - \1_{X\times A_0}\circ T^{n-k})| + |\1_{X\times A_0}\circ T^{n-k+1} - \1_{X\times A_{k-1}}|d\mu\\
& \leq 2n   \tilde{c} \epsilon \cdot \#(A_0) \leq 2n \tilde{c} \epsilon  \left(1 +   \tilde{c} \epsilon \right)
 \sum_{k=0}^{n-1} \#(A_k)  \ll \epsilon  \int |f| d\mu. 
\end{align*}
By applying Step 3 of the proof to $f$ as above, it follows that for any $\epsilon > 0$ there exists $A \subset V$ finite such that 
$\int |\widehat{T}(\1_{X\times A}) - \1_{X\times A}|d\mu \leq \epsilon   \#(A)$. This finishes the proof of the theorem. 
\end{proof}

In order to complete the picture, we now analyze the exponential decay rate  of the return probabilities to a fixed vertex $\mathbf{o} \in \V$.  In order to do so, for $n \in \N$ and $x \in X$, set $\kappa_x^n:= \kappa_{\theta^{n-1}(x)} \circ \cdots \circ \kappa_x$. The associated decay rate is for $\mathbf{o} \in \V$ defined by 
\begin{align*}
R(T,\mathbf{o}) &:= \limsup_{n \to \infty} \sqrt[n]{ \mu(\{x \in X : \kappa_x^n(\mathbf{o})=\mathbf{o} \}) }.
\end{align*}
We now relate the exponential decay rate $R(T,\mathbf{o})$ with the Gurevich pressure $P_G(T,\varphi)$ of the topological Markov chain $T$ and the potential $\varphi:= \log d\mu / d\mu \circ \theta$, whose definition we recall now. For $A\subset Y$, set
\[
P_G(T,\varphi,A) :=  \limsup_{n \to \infty}  \frac{1}{n} \log \sum_{T^n x = x,\;  x \in A} e^{\sum_{k=0}^{n-1} \varphi(T^k(x))}.
\]
Then, provided that $T$ is topologically transitive and $\varphi$ is locally Hölder continuous, it follows that $P_G(T,\varphi,A) = P_G(T,\varphi,B)$ whenever
$A$ and $B$ are cylinders of length 1 (see \cite{Sarig--Thermodynamic-Formalism-For-Countable--ETDS1999}). In particular, the \emph{Gurevich pressure} of $T$ and $ \varphi$ is defined by $P_G(T,\varphi):=P_G(T,\varphi,A)$, where $A$ is a cylinder of length 1.

\begin{proposition} \label{prop:gurevic-pressure}
$P_G(T,\varphi) = \log R(T,\mathbf{o})$.
\end{proposition}

\begin{proof}
As $\varphi$ is by assumption Hölder continuous, it is well known that Bowen's property holds. That is,
\[ \sum_{k=0}^{n-1} \varphi(\theta^k(x)) \asymp \sum_{k=0}^{n-1} \varphi(\theta^k(y))
\]
 whenever $x,y$ are in the same cylinder of length $n$. Hence, as $\theta$ has full branches and as $\widehat{T}$  is the transfer operator associated with the product of $\mu$ and counting measure,
\begin{align*}
 \log R(T,\mathbf{o}) &= \limsup_{n \to \infty} \frac{1}{n} \log \mu(\{x \in X : \kappa_x^n(\mathbf{o})=\mathbf{o} \})
      =  \limsup_{n \to \infty} \frac{1}{n} \log \int_{X \times \{ \mathbf{o}\}} \widehat{T}^n(\mathbf{1}_{X \times \{ \mathbf{o}\}})  d\mu
 \\ &
      = \limsup_{n \to \infty} \frac{1}{n} \log \sum_{w \in \cW^n : \kappa_w(\mathbf{o}) = \mathbf{o} }  \int  \tfrac{d\mu \circ \tau_w}{d\mu}  d\mu
      =  \limsup_{n \to \infty}  \frac{1}{n} \log \sum_{T^n (x, \mathbf{o}) = (x, \mathbf{o})} e^{\sum_{k=0}^{n-1} \varphi(\theta^k(x))}
 \\ & \geq  P_G(T,\varphi).
\end{align*}
Now fix  $a \in \cW^1$. Then, by topological transitivity of $T$, there exists $\ell \in \N$ and $b \in \cW^\ell$ such that $\kappa_{ab}(\mathbf{o})=\mathbf{o}$. Hence,
again by bounded distortion, with respect to any $ y \in [ab]$,
\begin{align*}
 \sum_{\genfrac{}{}{0pt}{}{T^{n+\ell +1} (x, \mathbf{o}) = (x, \mathbf{o}),}{x \in [a]}} e^{\sum_{k=0}^{n + \ell } \varphi(\theta^k(x))}
 & \geq \sum_{\genfrac{}{}{0pt}{}{T^{n+\ell +1} (x, \mathbf{o}) = (x, \mathbf{o}),}{x \in [ab]}} e^{\sum_{k=0}^{n + \ell } \varphi(\theta^k(x))}\\
 & \asymp e^{\sum_{k=0}^{\ell} \varphi(\theta^k(y))} \sum_{T^{n} (x, \mathbf{o}) = (x, \mathbf{o})} e^{\sum_{k=0}^{n -1} \varphi(\theta^k(x))}.
\end{align*}
It then follows by taking the limit as $n \to \infty$ and from the same argument as above that $P_G(T,\varphi) \geq \log R(T,\mathbf{o})$.
\end{proof}

Note that, if $(Y,T)$ is topologically transitive, then $R(T,\mathbf{o})$ is independent of the vertex $\mathbf{o}$, and we denote the common value by $R(T)$.

In order to relate $R(T)$ with the spectral radius $\rho(\widehat{T})$, we will employ a certain weak notion of symmetry. That is, we refer to $(Y,T,\mu)$ as \emph{symmetric} if there exist $(C_n)$ and $(N_n)$ such that $\lim_{n \to \infty} C_n^{1/n}=1$, $\lim_{n \to \infty} N_n/n=0$ and, for all $v,w \in \V$,
 \begin{align}\label{eq:weakly-symmetric} \mu \left( \left\{ x \in X: \kappa^n_x(v)=w \right\} \right) \leq C_n \sum_{k=n-N_n}^{n+N_n}\mu \left( \left\{ x\in X: \kappa^k_x(w)=v  \right\} \right).\end{align}

 Define $A:\cH_\infty \rightarrow  \cH_c$ given by $Af(v):=\int f(\cdot,v) d\mu$ and $T_n : \cH_c \rightarrow \cH_c$ given by $T_n:=A \widehat{T}^n$.  The following lemma can be proved as in \cite[Lemma 3.2]{Jaerisch--Group-Extended-Markov-Systems--PMS2015}.
 \begin{lemma} \label{lemmaTn} Suppose that
 $(Y,T,\kappa)$ is a  topologically transitive extension of a Gibbs-Markov map 
 $(X,\te,\mu,\alpha)$ with full branches. Then the following holds. 
 \begin{enumerate}[ref=(\roman*)]
 \item \label{item1}  There exists $C\ge 1$ such that for every $f \in \cH_c$ and all $n\in \N$, $C^{-1} \llbracket\widehat{T}^n(f)\rrbracket_\infty \le \llbracket\widehat{T}^n(f)\rrbracket_1  = \Vert T_n(f) \Vert_2  \le \llbracket\widehat{T}^n(f)\rrbracket_\infty $.
 \item  \label{item2}$\lim_{n\rightarrow \infty} \Vert T_n \Vert_2^{1/n}=\rho(\widehat{T})$
 \item  \label{item3}$\limsup_{n\rightarrow \infty} \langle T_n \1_{X\times \{\mathbf{o} \}},\1_{X\times \{\mathbf{o} \}}\rangle^{1/n}=R(T)$ for any $\mathbf{o} \in \V$.
 \end{enumerate}
 \end{lemma}
 With this Lemma at hand, we are now in position to show that $\rho(\widehat{T})=R(T)$ in case of a symmetric extension.
\begin{proposition} \label{prop:symmetry}  
Assume that $(Y,T,\kappa)$ is a topologically transitive  extension of a Gibbs-Markov map $(X,\te,\mu,\alpha)$ with full branches.
Then $R(T)\leq \rho(\widehat{T}) \leq 1$. If $(Y,T,\mu)$ is symmetric then $\rho(\widehat{T})=R(T)$. 
\end{proposition}
\begin{proof}
The first assertion follows from Lemma  \ref{lemmaTn} \ref{item2} and \ref{item3} in tandem with the Cauchy-Schwarz inequality.
Now suppose that $(Y,T,\mu)$ is symmetric.  Denote by $T_n^{*}$ the adjoint operator of $T_n$. By  Lemma  \ref{lemmaTn}  \ref{item2} we conclude that  
$$\rho(\widehat{T}) =  \lim_{n\rightarrow \infty} \Vert T_n^{*}T_n  \Vert_2^{1/(2n)}.  $$ Since $T_n^{*}T_n$ is self-adjoint,  it is well known that for each $n\in \N$,  $$\Vert T_n^{*}T_n \Vert_2 =\limsup_{k\rightarrow \infty} ((T_n^{*}T_n)^k \1_{X\times \{\mathbf{o} \}},\1_{X\times \{\mathbf{o} \}})^{1/k}.$$  
Since $\varphi$ is Hölder continuous, we conclude that for all $f_1, f_2 \in \mathcal{H}_c$ and $n,m\in \N$,
\[
\langle T_nT_m(f_1),f_2 \rangle \asymp \langle T_{n+m}(f_1),f_2 \rangle.
\] 
Moreover, by the symmetry assumption, we have 
\[
\langle T_n(f_1),f_2 \rangle \le C_n  \sum_{k=n-N_n}^{k+N_n} \langle T_{k}^{*}(f_1),f_2 \rangle.
\]
Now, we can prove $\rho(\widehat{T})=R(T)$ as in \cite[Proposition 1.5]{Jaerisch--Group-Extended-Markov-Systems--PMS2015}.
\end{proof}


\section{Amenability and embedded Gibbs-Markov structures}

We now relate decay rates, the spectral radius and amenability of extensions of Markov maps with  embedded Gibbs-Markov structure. Recall that,  for a graph extension $T$ with an embedded Gibbs-Markov structure $\sigma: \Om \to \Om$ as in definition \ref{def:embedded Gibbs-Markov structure}, 
$S:\Om \times \V \to \Om \times \V$, $S(x,g)=T^{\eta(x)}(x,g)$ denotes the associated graph extension whose base is a Gibbs-Markov map with full branches. The corresponding graph extensions are related as follows.
\[
\begin{tikzcd}
X \times \V   \arrow{rd} \arrow{rrr}{T} & & &  X \times \V \arrow{dl}\\
& X \arrow{r}{\theta} & X  & \\
& \Om \arrow[hook,dashed]{u}  \arrow{r}{\sigma}
& \Om \arrow[hook,dashed]{u}  & \\
\Om \times \V \arrow{ru} \arrow{rrr}{S} \arrow[hook,dashed]{uuu} & & &  \Om \times \V  \arrow{ul}\arrow[hook,dashed]{uuu}
\end{tikzcd}
\]
In here, the dashed arrows stands for a tower construction and therefore, the corresponding parts of the diagram do not necessarily commute with respect to inclusion. We begin with comparing the decay rates of the return probabilities of $S$ and $T$ to a fixed vertex $\mathbf{o} \in \V$. In order to do so, we introduce the following notation. For $n \in \N$ and $x \in X$, set $\kappa_x^n:= \kappa_{\theta^{n-1}(x)} \circ \cdots \circ \kappa_x$. Moreover, for $x \in \Omega$, define   $\eta_n := \sum_{j=0}^{n-1} \eta(\sigma^j(x))$ and 
$\hat{\kappa}_x^n:= \kappa_x^{\eta_n(x)}$. Note that with these definitions, $T^n(x,g)$ and $S^n(x,g)$  can be written as $(\theta^n(x),\kappa_x^n(g))$ and $(\sigma^n(x),\hat{\kappa}_x^n(g))$. The associated decay rates are now defined by 
\begin{align*}
R(T) &:= \limsup_{n \to \infty} \sqrt[{n}]{ \mu(\{x \in X : \kappa_x^n(\mathbf{o})=\mathbf{o} \})}, \\
R_\Omega(T) &:= \limsup_{n \to \infty} \sqrt[{n}]{ \mu(\{x \in \Om \cap \theta^{-n}(\Omega): \kappa_x^n(\mathbf{o})=\mathbf{o} \})}, \\
R(S) &:= \limsup_{n \to \infty} \sqrt[{n}]{ \nu(\{x \in \Om : \hat{\kappa}_x^n(\mathbf{o})=\mathbf{o} \})}. 
\end{align*}
Observe that, in the examples we have in mind, the logarithm of these rates coincides with the Gurevich pressures of $T$ and $S$. In order to relate these decay rates, the notion of an adequate embedding from Definition \ref{def:embedded Gibbs-Markov structure} will be crucial. Recall that this provides the existence of $(C_n)$ with $\lim_n C_n/n =0$ and 
\begin{align}\label{eq:medium-variation} | \log{\varphi_{w}(x)} - \log {\varphi_{w}(y)}|  \leq C_n d_\sigma(x,y) \leq C_n,\end{align}
for all $[wa] \in \alpha_{n+1}$ and $x,y \in \theta^n[w]$ with $[wa] \subset \Omega$, $[a] \subset \Omega$, $[a]\in \alpha$, and the existence 
of an almost surely finite function $\eta^\dagger:\Omega \to \N$ such that, for almost every $x \in \Om$ and $l=0, \ldots , \eta(x)-1$  with $\theta^\ell(x)\in \Om$, we have that $\eta(x)- l \leq \eta^\dagger(\theta^l(x))$.
In order to relate the decay rates, we introduce the following condition which also only depends on the embedded Gibbs-Markov structure.

\begin{defn}  \label{def:exponential-tails}
The embedded Gibbs-Markov structure has exponential tails if 
\[ \limsup_{n \to \infty} \sqrt[n]{ \mu\left\{ x \in \Omega : \eta(x)=n   \right\} } < 1.\]
\end{defn} 
\begin{proposition}  \label{prop:pressure_T_vs_S}
Assume that $(X,\theta,\mu, \alpha)$ is a Markov map with embedded Gibbs-Markov structure $\sigma$. 
Then, we have  $R(S) \leq R_\Omega(T) \leq R(T)$. Furthermore, 
if
$\sigma$ is  adequately embedded  and has exponential tails, then 
 $R_\Omega(T)=1$ implies that $R(S)=1$. 
\end{proposition} 

\begin{proof} As $\sigma$ is a Gibbs-Markov map, there exists a $\sigma$-invariant probability $m$ on $\Om$ with $ 1/C < dm/d\nu <C$ for some $C>0$ (see \cite{Aaronson-Denker-Urbanski--Ergodic-Theory-For-Markov--TMS1993}). 
 Clearly, we have $R(S)\leq 1$. 
Let $A_{n}:= \left\{x \in X : \kappa_x^n(\mathbf{o})=\mathbf{o} \right\}$ and  $\hat{A}_{n}:= \left\{x \in \Om : \hat\kappa_x^n(\mathbf{o})=\mathbf{o} \right\}$. This gives rise to the following estimate for $s\geq 1$,
\begin{align*} 
   \sum_{k=1}^\infty s^k  \mu(A_k) 
 & \geq   \sum_{k=1}^\infty s^k \mu(\{x \in \Om : \kappa_x^k(\mathbf{o})=\mathbf{o},\, \theta^k(x) \in \Om \})
   \asymp  \sum_{k=1}^\infty s^k \nu(\{x \in \Om : \kappa_x^k(\mathbf{o})=\mathbf{o},\, \theta^k(x) \in \Om \}) \\
 & \stackrel{(\ast)}{\geq}   \sum_{k=1}^\infty \sum_{l=1}^k s^k \nu(\{x \in \Om : \eta_l(x) =k,  \hat\kappa_x^l(\mathbf{o})=\mathbf{o} \}) 
   =  \sum_{l=1}^\infty  \int_{\hat{A}_l} s^{\eta_l(x)} d \nu 
  \geq  \sum_{l=1}^\infty s^{l} \nu(\hat{A}_l).
\end{align*}
Hence, the radius of convergence $1/R(S)$ of the last series is bigger than or equal to the radius of convergence $1/R_\Omega(T)$ of the second series which itself is bigger than or equal to the radius of convergence $1/R(T)$ of the first series. This proves the first assertion.

\medskip 
\noindent \textsc{Step 1.} Now assume that the Gibbs-Markov system is adequately embedded. We now show that the series adjacent to $(\ast)$ have the same radius of convergence. 
Set 
\[
\xi_n(x) := \min \left(\left\{ \eta_k(x) \geq n : \ k=1,2,3, \ldots \right\}\right).
\]  
That is, $\xi_n(x)$ is the time of the next return after time $n$ to $\Om$ with respect to $\sigma$. Now let $\eta^\dagger$ be given by  Definition \ref{def:embedded Gibbs-Markov structure}. As $\Om$ is a finite union of elements of $\alpha$, there exists 
 $K$ such that  $\nu\left(\left\{x \in [a] : \eta^\dagger(x) \leq K \right\}\right) \geq  \nu\left(\left\{x \in [a]: \eta^\dagger(x) > K \right\}\right)$ for all $a \in \alpha$ and $[a]\subset \Omega$.
 With $\ast$ standing for $w\in \left\{v \in \cW^{n+1}: \kappa_x^n(\mathbf{o})=\mathbf{o}, x\in  [v] \subset \Omega, \theta^n([v]) \subset \Om\right\}$, 
\begin{align*}
 (\ast\ast) & := \nu\left(\left\{x \in \Om : \kappa_x^n(\mathbf{o})=\mathbf{o}, \theta^n(x) \in \Omega \right\}\right) = \sum_{\ast} \nu([w]) \\
 & = \sum_{\ast} \left(
  \nu\left(\left\{x \in [w] : \eta^\dagger(\te^{n}(x)) \leq K \right\}\right) +  \nu\left(\left\{x  \in [w]  : \eta^\dagger(\te^{n}(x)) > K \right\}\right)\right)\\
& \leq  \sum_{\ast} \sup_{y \in  \theta^n([w])} \varphi_w(y) 
\left(
\nu\left(\left\{y \in \theta^n([w]) : \eta^\dagger(y) \leq K \right\}\right) +   \nu\left(\left\{y \in \theta^n([w]): \eta^\dagger(y) > K \right\}\right) \right)\\
& \leq 2  \sum_{\ast} \sup_{y \in  \theta^n([w])} \varphi_w(y) \, \nu\left(\left\{x \in \theta^n([w]) : \eta^\dagger(x) \leq K \right\}
\right).  
\end{align*}
Furthermore, again by using that $\Om$ is a finite union of elements of $\alpha$, 
there exists $B \subset \V$ finite such that, for all $[a] \in \alpha$, $[a] \subset \Omega$, 
\begin{align*} &\nu\left(\left\{x \in [a]: \eta^\dagger(x) \leq K, \kappa_x^k(\mathbf{o})\in B  \;\forall 1\leq k \leq K \right\}\right) \\
& \geq  \nu\left(\left\{x \in [a]: \eta^\dagger(x) \leq K,  \;\exists 1\leq k \leq K \hbox{ s.t. } \kappa_x^k(\mathbf{o})\notin B   \right\}\right).\end{align*}
Moreover, it follows from \eqref{eq:medium-variation} that for $v \in \cW^\infty$ with $[v] \subset \Omega$ and  $\theta^{|v|-1}([v]) \subset \Omega$, we have $\varphi_v(x)/\varphi_v(y)\leq \exp {C_n} $ for all $x,y \in \theta^{|v|-1}([v])$. Hence, as
 $\xi_n(x) -n \leq \eta^\dagger(\te^{n}(x))$ (cf. condition (ii) of an adequate embedding in Definition \ref{def:embedded Gibbs-Markov structure}), 
\begin{align*}
(\ast\ast) & \leq 
 4 \sum_{\ast} \sup_{x \in  \theta^n([w])} \varphi(x) \, 
 \nu\left(\left\{x \in  \theta^n([w]) : \eta^\dagger(x) \leq K, \kappa_x^k(\mathbf{o})\in B  \;\forall 1\leq k \leq K \right\}\right) \\
& \leq 4 e^{C_n} \nu\left(\left\{x \in \Omega: \xi_n(x)-n \leq K, \kappa_x^{\xi_n(x)}(\mathbf{o}) \in B \right\}\right).
\end{align*}
Hence, as $B$ and $K$ are independent of $n$, 
one obtains for $s \leq 1$ that 
\begin{align*}
 &   \sum_{n=1}^\infty s^n \mu(\{x \in \Om : \kappa_x^n(\mathbf{o})=\mathbf{o},\, \theta^n(x) \in \Om \}) \\
  \leq &  4  \sum_{n=1}^\infty s^n e^{C_n} \nu\left(\left\{x \in \Omega: \xi_n(x)-n \leq K, \kappa_x^{\xi_n(x)}(\mathbf{o}) \in B \right\}\right) \\
   \leq & 4Ks^{-K} \sum_{k=1}^\infty s^{k} e^{C_k} \nu\left(\left\{x \in \Omega: \exists l \hbox{ s.t. } \eta_l(x)=k,  \hat\kappa_x^l(\mathbf{o}) \in B \right\}\right).
\end{align*}
By Hadamard's formula, the radius of convergence of the right hand side does not depend on the factor $e^{C_k}$ as $\lim_k (\exp C_k)^{1/k} =1$. Moreover, as $\theta$ is transitive and again using the decay of $C_k/k$, one may replace the finite set $B$ with $\{\mathbf{o}\}$ without changing the radius of convergence.

\medskip
\noindent \textsc{Step 2.} Now assume the embedding has exponential tails.  Then $\int s^\eta d\nu < \infty$ for $s\in [1,1+\epsilon]$, for $\epsilon$ sufficiently small. By applying Sarig's version of Ruelle's theorem to the potential 
\[ f_s (x) := \log \frac{d\nu}{d\nu\circ \sigma}(x) + \eta(x)\log s, \] 
it follows that 
$\int_{\{\eta=\ell\}} s^\eta d\nu \asymp \lambda_s^l$, where $\lambda_s = \lim_{l \to \infty} (\int s^{\eta_l} d\nu)^{1/l}$. We now show that $s \mapsto \lambda_s$ is continuous. 
By H\"older's inequality we have for $t \in [0,1]$ and $a,b$ with $\lambda_{e^a}, \lambda_{e^b}<\infty$ that 
\begin{align*}
\log \lambda_{e^{ta+(1-t)b}} & =\lim_{l\to \infty}\frac{1}{l}\log\int e^{\left(ta+(1-t)b\right)\eta_{l}}d\nu  \le \lim_{l\to \infty} \frac{1}{l}\log\left(\left(\int e^{a\eta_{l}}d\nu\right)^{t}\cdot\left(\int e^{b\eta_{l}}d\nu\right)^{1-t}\right)\\
 & =t\log \lambda_{e^a}+(1-t)\log \lambda_{e^b}.
\end{align*}
This shows that $s \mapsto \log \lambda_s $ is convex, and hence continuous. 

Now assume that $R_\Omega(T)=1$ and $R(S)<1$. Then there exists $t<1$ such that $ \nu(\hat{A}_l) \ll t^l$ for all $l$. 
Therefore, the Cauchy-Schwarz inequality implies that
\begin{align*}    \int_{\hat{A}_l} s^{\eta_l(x)} d \nu \leq       \sqrt{\nu(\hat{A}_l)  \cdot  \int s^{2\eta_l} d\nu}
\ll    t^{l/2} \lambda_{s^2}^{l/2}
\end{align*}
As $\lambda_1 =1$, it follows from continuity that for $s>1$ sufficiently close to 1, $t \lambda_{s^2}  < 1$. For this choice of $s$, we hence have that  
\[ 
 \sum_{l=1}^\infty  \int_{\hat{A}_l} s^{\eta_l(x)} d \nu  \ll  \sum_{l=1}^\infty (t \lambda_{s^2})^{l/2} < \infty,
\]
which is a contradiction to $R_\Omega(T)=1$. 
\end{proof}

We now relate $\mu$-amenability with $\nu$-amenability. As a first result in this direction, it follows from Proposition \ref{prop:mu-amenable} that there exists a $\kappa$-F\o lner sequence. As a $\kappa$-F\o lner sequence is also a $\hat\kappa$-F\o lner sequence, $\G$ is $\nu$-amenable.

\begin{defn} \label{def:finitely-covers} We say that $\hat\kappa$ finitely covers $\kappa$ if there exists a finite set $ \mathcal{K}  \subset \hat{\cW}^\infty$ such that, for all $v \in \cW^1$ and $g \in \V$, there exists $w \in \mathcal{K}$ such that $ \kappa_v(g) = \hat{\kappa}_{w}(g)$.
\end{defn}

\begin{proposition} \label{prop:amenability_T_vs_S} If $\hat\kappa$ finitely covers $\kappa$, then 
  $\nu$-amenability and  $\mu$-amenability are equivalent.  
\end{proposition}

The proof of this proposition is easy and therefore omitted. The above results are summarized in the following diagram and the main result below for Markov maps with embedded Gibbs-Markov structure immediately follows from these.
\begin{equation}\label{diagram}
\begin{tikzcd}
R(T)=1   \arrow[Leftarrow]{r}  & R_{\Omega}(T)=1 	
\arrow[Rightarrow,shift right=+1.5ex]{d}[left]{\ref{prop:pressure_T_vs_S}} 
&  
& \mathcal{G} \,\, \mu\hbox{-amenable} \arrow[Rightarrow,shift right=+1.5ex]{d}[left]{\ref{prop:mu-amenable}} 
\\
 & R(S)=1 \arrow[Rightarrow,shift right=+1ex]{r} \arrow[Leftarrow,shift right=-0.5ex]{r}{\ref{prop:symmetry}}  \arrow[Rightarrow]{u}
& \rho(\widehat{S}) =1 
\arrow[Leftrightarrow]{r}{\hbox{\scriptsize\ref{theo:main_result}}} 
& \mathcal{G} \,\, \nu\hbox{-amenable}
\arrow[Rightarrow]{u}[right]{\ref{prop:amenability_T_vs_S}}
\end{tikzcd}
\end{equation}

\begin{theorem}\label{theo:main theorem - embedded GM structure}   Let $(X,\theta,\mu,\alpha)$ be a Markov map with $\theta$-invariant probability measure $\mu$. Suppose that $(Y,T,\kappa)$ is a graph extension of $(X,\theta,\mu,\alpha)$ with embedded Gibbs-Markov structure 
such that the induced graph extension $(Y,S,\widehat \kappa)$ is topologically transitive and has uniform loops. Then the following holds.
\begin{enumerate} 
\item If $\mathcal{G}$ is $\mu$-amenable and $(\Omega,\sigma,\nu)$ is symmetric then $R_\Omega(T)=1$.
\item Suppose that $T$ is topologically transitive, $\hat\kappa$ finitely covers $\kappa$ and the embedding is adequate and  has exponential tails. Then  
 $R_\Omega(T)=1$ implies that $\G$ is $\mu$-amenable.
\end{enumerate}\end{theorem}

{\color{black}
We show how to apply Theorem \ref{theo:main theorem - embedded GM structure} and the implications in Diagram \ref{diagram} for two classes of examples. With respect to the first, recall that a Markov map $(X,\theta,\mu,\alpha)$ is said to have \emph{big images and big preimages (b.i.p.)} if there exists a finite set $\mathcal{J} \subset \cW^1$ such that for every $a \in \cW^1$, there exist $b,c \in \mathcal{J}_\theta$ with $[bac] \in \in \cW^3$. Furthermore, we refer to the Markov map as a \emph{Gibbs-Markov map with b.i.p.} if condition (ii) in Definition \ref{def:gm-map} and the b.i.p.-property hold. Also recall it is well known that this property implies that one may assume without loss of generality that $\mu$ is a $\theta$-invariant probability measure.

We now recall known results for group extensions with respect to a base with b.i.p. So assume that $G$ is a countable group with generating set $\mathfrak{S}$, and let $\kappa : X \to \mathfrak{g}$ be a function which is constant on elements of $\alpha$. This induces an extension $T$ of $(X,\theta)$ by the Cayley graph of $G$ with respect to $\mathfrak{g}$. The principal result of \cite{Stadlbauer--An-Extension-Of-Kestens--AM2013} now states that in case of a Gibbs-Markov map with b.i.p. and a transitive extension $T$, $R(T)=1$ implies that $G$ is amenable. As a consequence of the proof, one also obtains that  $\rho(\widehat{T}) =1$ implies amenability. This result subsequently was extended in \cite{Jaerisch--Group-Extended-Markov-Systems--PMS2015} by showing that $\rho(\widehat{T}) =1$ is equivalent to amenability of $G$.

We now return to extensions by general graphs over a Gibbs-Markov map $(X,\theta,\mu,\alpha)$ with big images and big preimages. In order to employ Theorem \ref{theo:main theorem - embedded GM structure} it is now necessary to refine the existence of finite loops in case of a full shift to returns to $\mathcal{J}_\theta$ as follows. That is, we say that $(Y,T,\kappa)$ has \emph{finite $\mathcal{J}_\theta$-loops} if there exists a finite subset $\mathcal{J}_T \subset \cW^\infty$ such that for every pair $a,b \in \mathcal{J}_\theta$ and $v \in \V$, there exists $\mathfrak{c} = \mathfrak{c}_{a,b,v}  \in \mathcal{J}_T$ such that $\mathfrak{c} = ac b  \in \cW^\infty$ for some $c$ and $\kappa_{\mathfrak{c}}(v)=v$.

\begin{corollary} \label{cor:bip}
Suppose that  $(Y,T,\kappa)$ is a topologically transitive extension of a Gibbs-Markov map  $(X,\te,\mu,\alpha)$ with big images and big preimages and invariant probability $\mu$ and that $T$ has finite $\mathcal{J}_\theta$-loops.  Then the following holds.
\begin{enumerate}
\item If $R(T) =1$, then $\G = (\V,\E)$ is $\mu$-amenable.
\item The following are equivalent.
\begin{enumerate}
 \item $\G = (\V,\E)$ is $\mu$-amenable.
 \item  The spectral radius $\rho(\widehat{T})$ of $\widehat{T}: \cH_\infty \to \cH_\infty$ is equal to 1.
 \item For all $\epsilon>0$,   there exists a finite set $A$ such that $   \llbracket \widehat{T}(\1_{X \times A}) - \1_{X \times A} \rrbracket_1  \leq \epsilon  \llbracket \1_{X \times A} \rrbracket_1$.
\end{enumerate}
\end{enumerate}

\end{corollary}

\begin{proof}
We begin with the construction of the embedded Gibbs-Markov map. To this end, observe that
$(X \times \V,T)$ is a topologically transitive Markov chain and therefore admits the following decomposition. There exists $p \in \N$ and a partition $\{ \mathfrak{M}_1, \ldots,  \mathfrak{M}_p\}$ of $X \times \V$ into $\alpha \times \V$-measurable sets such that $T(\mathfrak{M}_j) = \mathfrak{M}_{j+1 \mod p}$ and that $T^p : \mathfrak{M}_j \to \mathfrak{M}_j $ is topologically mixing for all $j = 1, \ldots ,p$. It follows that for each pair $a, b \in \mathcal{J}_\theta$ with $ba \in \cW^2$, and $v \in \V$, the length $|\mathfrak{c}_{a,b,v}|$ of the word $\mathfrak{c}_{a,b,v}$ is a multiple of $p$. By considering repeated words of the form
$\mathfrak{c}_{a,b,v} \mathfrak{c}_{a,b,v} \cdots \mathfrak{c}_{a,b,v}$ and setting $\mathfrak{n}$ to be the least common multiple of all lengths, we may assume without loss of generality that $|\mathfrak{c}_{a,b,v}| = \mathfrak{n}$ for all  $a, b \in \mathcal{J}_\theta$ with $ba \in \cW^2$, and $v \in \V$.

Moreover, we may assume without loss of generality that the cardinality  of $\alpha$ is at least two, since otherwise  there is nothing to show. Hence, by eventually modifying $\mathcal{J}_\theta$, $\mathcal{J}_T$ and $\mathfrak{n}$, we may assume that for fixed  $a, \overline{a} \in \mathcal{J}_\theta$ with $\overline{a}a \in \cW^2$, there exists for each $ b \in \mathcal{J}_\theta  \cup \{ \mathfrak{o}\}$ a map $ \mathfrak{c}^\ast_{b, \; \cdot \;}:  \V \to \cW^\mathfrak{n}$
such that $\mathfrak{c}^\ast_{b,v}a \in \cW^{\mathfrak{n}+1}$, $\kappa_{\mathfrak{c}^\ast_{b,v}}(v) = v$  and $\mathfrak{c}^\ast_{b,v} \neq \mathfrak{c}^\ast_{\overline{b},v}$ for all $b \neq \overline{b}$. The return time $\eta$ of the embedded Gibbs-Markov map is now defined as follows.
\begin{enumerate}
 \item If $x  \in [\mathfrak{c}^\ast_{\mathfrak{o},v} a ]$ for some $v \in V$, then $\eta(x) =|\mathfrak{c}^\ast_{\mathfrak{o},v} |$.
 \item If $x  \in [\mathfrak{c}^\ast_{b,v} \, \mathfrak{c}_{a,b,v} \, c \,  \mathfrak{c}_{d,a,\kappa_c(v)}  ]$ for some $v \in V$ and $b,d \in \mathcal{J}_\theta$, $c \in \cW^1$ such that $b \neq \mathfrak{o}$
 and  $bcd  \in \cW^3$, then
$\eta(x) =\left|\mathfrak{c}^\ast_{b,v} \, \mathfrak{c}_{a,b,v} \, c \,  \mathfrak{c}_{d,a,\kappa_c(v)}\right|-1$.
 \item If $x \in [a]$ is not in any of the cylinders of part (i) or (ii), let $\eta$ be the $k$-the return to $[a]$, where $k = \max \left\{ \left|\mathfrak{c}^\ast_{b,v} \, \mathfrak{c}_{a,b,v} \, c \,  \mathfrak{c}_{d,a,\kappa_c(v)}\right|-1  : b,d \in \mathcal{J}_\theta, c \in \cW^1, v \in \V \right\}$.
\end{enumerate}
\noindent{\textsc{Proof of the first assertion.}}
Note that
$\sigma: [a] \to [a]$, $x \mapsto \theta^{\eta(x)}(x)$ is by construction a Markov map with full branches, say with respect to the partition $\beta$. As $\theta$ is a Gibbs-Markov map, condition (ii) in Definition \ref{def:gm-map} passes over to $\sigma$. Furthermore, as $\eta$ is bounded by the $k$-th return time, it follows that $\sum_{[b] \in \beta} \mu([b]) < \infty$. Hence, there is a unique $\sigma$-invariant probability measure $\nu$, which is absolutely continuous with respect to $\mu$. Moreover, $\log d\nu / d \mu$ is uniformely bounded from above and below. Furthermore, as a consequence of the Gibbs-Markov property of $\theta$ and the fact that $\eta$ is bounded by the $k$-th return time, the embedding is adequate.
Hence, we have shown that $\sigma$ is an adequately embedded Gibbs-Markov map with full branches.

Moreover, as $\theta$ is a Gibbs-Markov map with the b.i.p. property, it is well-known (see, e.g., \cite{Sarig--Existence-Of-Gibbs-Measures--PMS2003}) that $(\widehat{\theta} - \id)$ acts as a strict contraction on the space of Hölder continuous functions. As in the proof of Theorem \ref{theo:Non-Markovian} below, this implies that $k$-th return to $[a]$ has exponential tails. In particular, $\eta$ also has this property.

We now turn our attention towards the topological properties of $S: [a] \times \V \to [a] \times \V$. As the cylinders in (i) and (ii) are pairwise disjoint, it follows from (i) that the map $S$ has uniform loops. Furthermore, it follows from (ii) that the topological transitivity of $T$ implies that of $S$. Moreover, due to
 \[ \mu([c]) \asymp \nu([\mathfrak{c}^\ast_{b,v} \, \mathfrak{c}_{a,b,v} \, c \,  \mathfrak{c}_{d,a,\kappa_c(v)} ]),\] it follows immediately that $\nu$-amenability implies $\mu$-amenability. Finally, by using the b.i.p. property, one easily proves that
 $R(T)=1$ if and only if $R_{[a]}(T)=1$.
 By following the implications in Figure \ref{diagram}, one obtains the first assertion of the theorem. That is, $\mathcal{G}$ is $\mu$-amenable if $R(T)=1$.

 \medskip
 \noindent{\textsc{Proof of the remaining assertions.}}
 So assume that $\rho(\widehat S) =1$. Then, by Theorem \ref{theo:Schreier - embedded GM structure}, the graph is $\mu$-amenable. Furthermore, observe that by (ii) in the construction of $S$ there exists $C > 0$ such that
 \[
  {\varphi_v(x)} \leq C \Phi_{\mathfrak{c}^\ast_{b,v} \mathfrak{c}_{a,b,v} \, c \,  \mathfrak{c}_{d,a,\kappa_c(v)}  }(y)
 \]
 for all $x \in X$ and $y \in [a]$.
 Furthermore, by Theorem \ref{theo:main_result} there exists, for any $\epsilon > 0$, a finite set $A  \subset \V$ such that
$   \llbracket \widehat{S}(\1_{X \times A}) - \1_{X \times A} \rrbracket_1  \leq  \epsilon  \llbracket \1_{X \times A} \rrbracket_1$. By (ii) in the construction of $S$, for $x \in X$, $y \in [a]$ and with $\triangle$ referring to the symmetric difference,
 \begin{align*}
   |\widehat{T}(\1_{X \times A}) - \1_{X \times A}|(x,g) & = \sum_{v \in \cW^1: g \in A \triangle  \kappa_v(A)} \varphi_v(x)
  \leq C \sum_{v \in \cW^1: g \in A \triangle  \kappa_v(A)} \varphi_{v}(x) \Phi_{\mathfrak{c}^\ast_{b,v} \mathfrak{c}_{a,b,v} \, c \,  \mathfrak{c}_{d,a,\kappa_c(v)}  }(y) \\
  & \ll C  |\widehat{S}(\1_{X \times A}) - \1_{X \times A}|(y,g),
 \end{align*}
 where the comparability in the last line is a consequence of the change to the $S$-invariant probability measure $\nu$. Hence,
 $   \llbracket \widehat{T}(\1_{X \times A}) - \1_{X \times A} \rrbracket_1  \ll \epsilon  \llbracket \1_{X \times A} \rrbracket_1$. As $\epsilon$ can be chosen arbitrarily, $1$ is an element of the spectrum of $\widehat T$ and, in particular, $\rho(\widehat T) =1$.
 Hence, we have shown so far that $\rho(\widehat S) =1$ implies statements (a), (b) and (c) in (ii).

 Now assume that $\rho(\widehat S) < 1$. Then $(\id - \widehat S)^{-1} = \sum_{n \geq 0} {\widehat S}^n$ is a well-defined and bounded operator with respect to $ \llbracket \, \cdot \, \rrbracket_\infty$. Fix $\lambda \in \mathbb{C}$ with $|\lambda| =1$ and $f : X \times \V \to \mathbb{C}$ with $\llbracket f \rrbracket_\infty < \infty$. Furthermore, set ${f^\ast}(x,g) := \sup_{x\in X} |f(x,g)|$. Then, for $x \in X$ and $y \in [a]$, a repetition of part (ii) in the construction of $S$ shows that
 \begin{align*}
   (\ast) :=  \left|  \sum_{n \geq 0} \lambda^n {\widehat T}^n(f) (x,g) \right|  \leq
   \sum_{n \geq 0}  {\widehat T}^n\left(f^\ast\right) (x,g)
   \ll \sum_{n \geq 0}  {\widehat T_{[a] \times \V}}^n\left(f^\ast\right) (y,g),
  \end{align*}
  where $T_{[a] \times \V}$ refers to the first return map to ${[a] \times \V}$. Hence, with $k$ as in (iii) of the construction of $S$ and by substituting $\varphi_v = d \mu \circ \tau_v / d \mu$ with  $\varphi^\ast_v = d \nu \circ \tau_v / d \nu$,
  \begin{align*}
   (\ast)  & \ll \sum_{x \in [a], n \geq 0: \theta^n x =y} \prod_{j=0}^{n-1} \varphi^\ast(\theta^j x) f^\ast(x, \kappa_x^{-n}(g))
    \\
   & \asymp  \sum_{\ell=0}^\infty \quad \sum_{x \in [a], n \geq 0: \theta^n x =y, \atop \eta_\ell(x) \leq n <  \eta_{\ell + 1}(x)} \prod_{j=0}^{\ell(x)-1} \varphi^\ast(\theta^j x) \prod_{\ell(x)}^{n-1} \varphi(\theta^j x) f^\ast(x, \kappa_x^{-n}(g)) \\
   & \leq \sum_{\ell=0}^\infty  {\widehat S}^\ell    \left( \textstyle \sum_{j=0}^{k-1}{\widehat T_{[a] \times \V}}^j(f^\ast) \right)(y,g).
  \end{align*}
  As  $\sum_{\ell}  {\widehat S}^\ell$ is a bounded operator, the  above estimate implies that  $(\id - \lambda {\widehat T})^{-1} = \sum_{n} \lambda^n {\widehat T}^n$ also is a bounded operator. In particular, $\rho({\widehat T})< 1$. Hence, we have shown that $\rho({\widehat T})= 1$ implies $\rho({\widehat S})= 1$.

  Furthermore, since it is easy to see that (c) implies (b), and since $\mu$-amenability implies  $\rho({\widehat S})= 1$ by following the implications in Figure \ref{diagram}, the remaining assertions follow.
  \end{proof}

  In contrast to the above class of examples, the maps in the following class are not uniformly expanding. More specifically, we consider maps $\theta: [0,1] \to [0,1]$ with indifferent fixed points such that the restriction to each sufficiently small interval adjacent to such a  fixed point is conjugated to a continuous  map $F_{\epsilon,\alpha,\beta} = F :   [0,\epsilon] \to [0,1] $  with  $F(0) = 0$ and $F \in C^2(0,\epsilon)$ such that $F'(x) = 1 + cx^{\alpha} + o(x^\alpha)$, $F'(x) > 1$ and $|F''(x)|  \ll  x^{- \beta}$, for some $\epsilon > 0$, $\alpha > 0$ and $\beta \in [0,1)$ and $x>0$. With $F_{\epsilon,\alpha,\beta}$ at hand, we are now in position to define the following class of non-uniformly expanding maps, whose definition is inspired by Thaler (\cite{Thaler--Transformations-On-01-With--IJM1983}).

\begin{defn}
    Assume that $X= [0,1]$, that $0 = a_0 < a_1 < \ldots < a_k =1$, and that for each $j=1,\ldots, k$, there is a homeomorphism  $\theta_j: [a_{j-1},a_j] \to [0,1]$ with  $\theta_j  \in C^1(a_{j-1},a_j)$. We then refer to $\theta := \sum_{1 \leq j< k} \mathbf{1}_{[a_{j-1},a_j)} \theta_j + \mathbf{1}_{[a_{k-1},a_k]} \theta_k$ as a \emph{Thaler-like map}, if the following holds, for each $j = 1, \ldots, k$.
  \begin{enumerate}
   \item There exists a unique $x_j \in  [a_{j-1},a_j]$ with $\theta_j(x_j) = x_j$, and $\theta \in C^2((a_{j-1},a_j)\setminus \{x_j\})$.
   \item If $\inf (\{ |\theta'(x)| : x \in (a_{j-1},a_j) \}) > 1$, then  $\sup_{x \in (a_{j-1},a_j), x \neq x_j} |\theta''(x)(\theta'(x))^{-2}| < \infty$.
   \item If $\inf (\{ |\theta'(x)| : x \in (a_{j-1},a_j) \}) \leq 1$, then there exist $\epsilon,\delta > 0$ such that  $|\theta'(x)|> 1+ \delta $ and $|\theta''(x)|<1/\delta$ for each $x\in  (a_{j-1},a_j)$ with $|x-x_j| > \epsilon$. Furthermore, there exist $\alpha_{\pm},\beta_{\pm} > 0$ with $\beta_{\pm} < 1/\alpha_{\pm}$ such that,
\begin{enumerate}
 \item the restriction of $\theta_j$ to  ${[x_j,x_j + \epsilon]}$ is conjugated by a translation either to $F_{\epsilon,\alpha_+,\beta_+}$ or $  \iota \circ F_{\epsilon,\alpha_+,\beta_+}$, with $\iota(x) := -x$ and provided that $(x_j,x_j + \epsilon)  \cap (a_{j-1},a_j)\neq \emptyset$,  and
 \item the restriction of $\theta_j$ to ${[x_j - \epsilon, x_j]}$ is conjugated by a translation either to $   F_{\epsilon,\alpha_-,\beta_-}\circ \iota$ or $  \iota \circ F_{\epsilon,\alpha_-,\beta_-}\circ \iota$, provided that $(x_j - \epsilon, x_j)  \cap (a_{j-1},a_j)\neq \emptyset$.
\end{enumerate}
\end{enumerate}
\end{defn}

In contrast to the maps introduced in \cite{Thaler--Transformations-On-01-With--IJM1983}, we do not allow infinitely many branches. The reason for that is as follows. If $(Y,T,\kappa)$ is a nearest neighbour extension by a graph of $(X,\theta,\mathrm{Leb})$, then the transfer operator $\widehat{T}$ acts as a bounded operator on $\mathcal{H}_\infty$. On the other hand, we do not assume that $\theta''/(\theta')^2$ is uniformly bounded, which allows us to include the class of Liverani-Saussol-Vaienti maps (see \cite{Liverani-Saussol-Vaienti--A-Probabilistic-Approach-To--ETDS1999}) in our analysis.

Due to the fact that a Thaler-like map has full branches,  $\alpha := \{ [a_{j-1},a_j)  : j = 1, \dots, k \}$ is a Markov partition. For  $a,b \in \alpha$, $a \neq b$, $n \in \N$ and
$x \in [ a^n b ]$,  define $\eta (x) := n + 1$, where $a^n$ stands for the $n$-fold repetition of $a$. As $\eta$ is defined on a set of full measure, the map $\sigma: [0,1] \to [0,1]$, $x \mapsto \theta^{\eta(x)}(x)$, known as
Schweiger's jump transformation, is well defined.
Moreover, as shown below, $\sigma$ is a full branched Gibbs-Markov map.

\begin{corollary} \label{cor:nue}
 Assume  that $(Y,T,\kappa)$ is a topologically transitive graph extension of the Thaler-like map $(X,\theta,\mathrm{Leb},\alpha)$ such that there is a finite set $\mathcal{J} \subset \bigcup_{n \geq 0} \alpha^n$ such that  for any $v \in \V$ and $a \in \alpha$, there is $b \in \mathcal{J}$ such that $\kappa_{ab}(v)=v$. If $\mathcal{G}$ is amenable, then the spectral radius of the action of $\widehat{T}$  on $\mathcal{H}_\infty$ is equal to one.
\end{corollary}

Before passing to the proof of Corollary \ref{cor:nue}, we analyse the distortion properties of Thaler-like maps and their associated jump transformation.

\begin{proposition}
Assume that $(X,\theta,\mathrm{Leb},\alpha)$ is a Thaler-like map. Then, with respect to Le\-besgue measure, $\sigma$ is a Gibbs-Markov map with full branches and $\theta$ is conservative and ergodic. Furthermore, there is a $\theta$-invariant measure, which is absolutely continuous with respect to Lebesgue measure and which is finite if and only if all the exponents $\alpha$ of the normal forms $F_{\epsilon,\alpha,\beta}$ are strictly smaller than one. Moreover, there is $(C_n)$ with $\lim C_n/n =0$  such that, for all $[w] \in \alpha_{n}$ and a.e. $x,y \in [0,1]$,
\begin{align}\label{eq:weak-distortion-whole-taler-map}
 \left| \log\tfrac{d\mathrm{Leb} \circ \tau_{w}}{d\mathrm{Leb}}(x) - \log\tfrac{d\mathrm{Leb} \circ \tau_{w}}{d\mathrm{Leb}}(y)  \right|  \leq C_n.
\end{align}
\end{proposition}

\begin{proof} In the first step of the proof, we analyse the asymptotic behaviour of the pre-orbits attracted to the fixed points. In order to do so, it suffices to consider  $F_{\epsilon,\alpha,\beta} = F$ as  defined as above, for some $\epsilon > 0$, $\alpha > 0$ and $\beta < 1$. As it easily can be seen, this implies that
$(F^{-1})' = 1 - cx^{\alpha} + o(x^\alpha)$ and $|(F^{-1})''(x)|  \ll  x^{- \beta}$. It then follows from Lemma 2 and Corollary in \cite{Thaler--Transformations-On-01-With--IJM1983} that, for $ x \in [F^{-1}(\epsilon), \epsilon]$ and $n > 1$,
\[
n \sim \int^{\epsilon}_{F^{-n}(x)} \frac{1}{t -
F^{-1}(t)} dt \sim  \int^{\epsilon}_{F^{-n}(x)} \frac{\alpha + 1}{ ct^{\alpha+1}} dt \sim
\frac{ \alpha + 1}{c \alpha}  (F^{-n}(x))^{-\alpha}.
\]
Hence, $(F^{-n}(x))^\alpha \sim (\alpha + 1)  /(c \alpha n)$. In particular, for some sequence $ c_j(x) \to 0$ as $j \to \infty$,
\begin{align*}
 \frac{|(F^{-n})''(x)|}{|(F^{-n})'(x)|} & = \sum_{j=1}^n \frac{|(F^{-1})''(F^{-j}x)|}{|(F^{-1})'(F^{-j}x)|} |(F^{-j})'(x)|
 \ll \sum_{j=0}^{n-1} ((F^{-j})(x))^{-\beta} e^{\sum_{i=0}^{j-1} \log |(F^{-1})'(F^{-i}(x))|} \\
 & \ll  \sum_{j=0}^{n-1} j^{\frac{\beta}{\alpha}} e^{-c \sum_{i=0}^{j-1}  (F^{-i}(x))^\alpha + o((F^{-i}(x))^\alpha)  }  =  \sum_{j=0}^{n-1} j^{\frac{\beta}{\alpha}  - \frac{\alpha + 1}{\alpha}  (1+ c_j(x))  }.
\end{align*}
Hence, if $\beta < 1$, then the above sum converges and, as $(\log (F^{-n})')' = (F^{-n})''/(F^{-n})'$, the mean value theorem implies that
\begin{align} \label{eq:gibbs-markov-for-thaler}
 \sup_{x,y \in [F^{-1}(\epsilon), \epsilon]} | \log (F^{-n})'(x) - \log (F^{-n})'(y) | \ll |x-y|.
\end{align}
On the other hand, the above calculation also provides us with the estimate  $|(F^{-n})'(x)| = n^{-{(\alpha+1)}(1+ c_n(x))/\alpha }$, which implies that
\begin{align} \label{eq:weak-gibbs-for-thaler}
 \sup_{x,y \in [0, \epsilon]} | \log (F^{-n})'(x) - \log (F^{-n})'(y) |
= \sup_{x \in [0, \epsilon]} | \log (F^{-n})'(x)|
\ll   \log n.
\end{align}
In the second and final step of the proof, we now consider the global dynamics of $\theta$. In order to do so, observe that outside a neighbourhood of the fixed points $\{x_j\}$, $\theta$ is uniformly expanding and in $C^2$. Moreover, in the neighbourhoods of the  $\{x_j\}$, the asymptotic behaviour is determined by  \eqref{eq:gibbs-markov-for-thaler} and \eqref{eq:weak-gibbs-for-thaler}. In particular, $\sigma$ is a Gibbs-Markov map with full branches by \eqref{eq:gibbs-markov-for-thaler}, and $\theta$ satisfies \eqref{eq:weak-distortion-whole-taler-map}. Moreover, it follows from the general theory of Gibbs-Markov maps that there exists a $\sigma$-invariant measure $\mu$ with $\|\log d\mu/d\mathrm{Leb}\|_\infty < \infty $.

In order to prove that $\theta$ is conservative and ergodic, it suffices to apply the general theory for Markov maps with weak distortion (cf. \cite{Aaronson--An-Introduction-To-Infinite--1997}): as  $\sigma$ is conservative, $\theta$ has to be conservative and therefore is ergodic. Furthermore, by representing  $\theta$ as a tower over $\sigma$, one obtains a canonical $\theta$-invariant measure $m$ as a lift of $\mu$, and $m$ is finite if and only if $\eta$ is integrable with respect to the Lebesgue measure. However, by \eqref{eq:gibbs-markov-for-thaler}, $\int \eta d \mathrm{Leb} < \infty$ is equivalent to $\sum_F \sum |(F^{-n})'(x)| n < \infty$ where $F$ runs through the normal forms of $\theta$.
\end{proof}

\begin{proof}[Proof of Corollary \ref{cor:nue}]
Similar to the proof of Corollary \ref{cor:bip}, we may assume without loss of generality that the words in $\mathcal{J}$ have all the same length $n$. It is now  easy to check that the graph extension $S$ associated to $\theta^n \circ \sigma$ is a Gibbs-Markov map, which finitely covers $T$. In particular, $S$ is topologically transitive.  In analogy to the proof of Corollary \ref{cor:bip}, the remaining assertion is a consequence of $\sum_n s^n \widehat{T}^n(f) \geq \sum_n s^n \widehat{S}^n(f)$, for $s \geq 1$ and $f\geq 0$.
\end{proof}

\begin{remark} As closing remark to this part, it is worth pointing out that the converse of Corollary \ref{cor:nue} does not hold: Assume that $(Y,T,\kappa)$ is a graph extension of Thaler-like map, that $a \in \alpha$ is an atom with $\inf_{x \in a} |\theta'(x)|=1$, and that $\kappa_x =\id $ for all $x \in a$. It now follows from the fact that $\mathrm{Leb}(a^n)$ decays polynomially, that $R(T) = \rho(\widehat{T})=1$. Hence, in this situation, these quantities do not provide a criterion for amenability.
\end{remark}

}

\section{Applications to Schreier graphs} \label{sec:applications}
This section is devoted to the application of Theorems \ref{theo:main_result} and \ref{theo:main theorem - embedded GM structure} as well as the relations indicated in diagram \ref{diagram}  to the specific case of a Schreier graph whose construction we recall now. Let $G$ be a discrete group,
$H$ a subgroup of $G$ and $\mathfrak{g} \subset G$ a generating set of $G$. The Schreier graph $\mathcal{G}=(\V,\E)$ associated with $\mathfrak{g}$ is then defined as the graph whose vertices are the cosets $\V=\{ Hg : g \in G\}$ and whose edges
 $\E =  \{ (Hg,Hgh) : g \in G, h \in   \mathfrak{g} \}$ are given by the right action of $\mathfrak{g}$ on $\V$.
In order to define a graph extension of the Markov map $(X,\theta)$, it now suffices to fix a map $\gamma: X \to \mathfrak{g}$, $x \mapsto \gamma_x$ and consider  the skew product
\[ T:  X \times \V \to X \times \V, \; (x,Hg) \mapsto (\te x, Hg\gamma_x), \]
that is $\kappa$ is defined by $\kappa_x(Hg) := Hg\gamma_x$.  If $\gamma$ is measurable with respect to the Markov partition $\alpha$, we say that the extension has \emph{Markovian increments}, and otherwise that the extension has \emph{non-Markovian increments}.

\subsection{Extensions by Schreier graphs with Markovian increments}
\label{sec:Schreier}
Throughout this section, we assume that $(X,\theta)$ is a Markov map and that $\gamma: X \to \mathfrak{g}$ is constant on cylinders of length 1.
In order to have a similar notation as for nearest neighbour cocycles at hand, set $\gamma_w := \gamma_{w_1}\gamma_{w_2} \cdots \gamma_{w_n}$ for $w = (w_1 \ldots w_n) \in \cW^\infty$. Moreover, for the embedded Gibbs-Markov map, let $\hat{\gamma}_x := \gamma_x \cdots \gamma_{\theta^{\eta(x)-1}(x)}$ and $\hat{\gamma}_w$ accordingly. The conditions of topological transitivity \textit{(tt)}, uniform loops \textit{(ul)} and finite cover \textit{(fc)} of $\kappa$ for extensions by Schreier graphs read as follows.
\begin{description}
\item[\rm\textit{(tt)}] 
For all $g,h \in G$, there exists $w \in \hat{\cW}^\infty$ with  $\hat{\gamma}_w \in gHh$. 
\item[\rm\textit{(ul)}] 
There is a finite subset $\mathcal{J} $ of  $\hat{\cW}^1$ such that $ \forall \, g \in G$, there exists $u \in \mathcal{J}$ with  
$\hat{\gamma}_u \in gHg^{-1}$. 
\item[\rm\textit{(fc)}] There is a finite subset $\mathcal{K} $ of  $\hat{\cW}^\infty$ such that $ \forall\,  g \in G, $   $\forall v \in \cW^1$,  there exists $u \in \mathcal{K} $ with  $\gamma_v\hat{\gamma}_u^{-1} \in gHg^{-1}$. 
\end{description}

In particular, if \textit{(tt)} and  \textit{(ul)} are satisfied and $\theta$ is a full Gibbs-Markov map (in this case, $\hat{\cW}^\infty = {\cW}^\infty$), then Theorem \ref{theo:main_result} provides an amenability criterion in terms of the spectral radius. On the other hand, if \textit{(tt)}, \textit{(ul)} and \textit{(fc)} are satisfied, $\sigma$ is adequately embedded and has exponential tails,
then part (ii) of Theorem  \ref{theo:main theorem - embedded GM structure} is applicable. In this situation, $R_\Omega(T)=1$ implies $\mu$-amenability.

In order to obtain a less abstract criterion  
recall that the normal core of a subgroup is defined by 
\[H_0 := \bigcap_{g \in G} gHg^{-1} \]
and that $H_0$ is the maximal normal subgroup in $G$ which is contained in $H$. Note that the coset space $\{H_0g : g \in G\}$ is isomorphic to the group $G/H_0$ by normality. In particular, by substituting  $H$ by $H_0$ in the construction of $T$, we obtain
\[ T_0:  X \times G/H_0 \to X \times G/H_0, \; (x,H_0g) \mapsto (\te x, H_0g\gamma_x), \]
which is an extension by a group as considered in \cite{Stadlbauer--An-Extension-Of-Kestens--AM2013, Jaerisch--Group-Extended-Markov-Systems--PMS2015}. The advantage of this construction is that the topological transitivity of $T_0$ allows to modify the embedded Gibbs-Markov map such that the embedded map automatically is topologically transitive, has uniform loops and finitely covers $T$. The following result allows to deduce amenability from $R_\Omega(T)=1$ which is considered the hard part of Kesten's amenability criterion for groups.
\begin{theorem}\label{theo:Schreier - embedded GM structure} Let  $(X,\theta,\mu,\alpha)$ be a Markov map with $\theta$-invariant probability measure $\mu$ and with adequately embedded Gibbs-Markov structure with exponential tails. Furthermore, let $H$ be a subgroup of the countable group $G$ and $\gamma:X \to G$ be a map which is constant on cylinders of length 1 such that $\mathfrak{g}:= \gamma(X)$ is finite and that the {Markov map $T_0$} is topologically transitive.
If $R_\Omega(T)=1$, then  the Schreier graph associated with $\mathfrak{g}$ is $\mu$-amenable. 
\end{theorem}

\begin{proof} 
It suffices to construct an embedded map which satisfies \textit{(tt)}, \textit{(ul)}, \textit{(fc)}, has exponential tails and then apply the second part of Theorem \ref{theo:main theorem - embedded GM structure}. In order to do so, choose $u \in \hat{\cW}^1$. By topological transitivity of $T_0$, there exists for each $ h \in \mathfrak{g} \cup \{\id\}$ a word $w_h \in \cW^\infty$  
such that each word $v_h:= w_h u$  is admissible, $[w_h] \subset \Omega$, $\gamma_{v_h} \in h H_0$ and, for $h \neq \tilde{h}$,       
 $[w_h] \cap [w_{\tilde{h}}] = \emptyset$. 
These cylinders  give rise to a further embedded Markov map $\tilde{S} : \Omega \to \Omega$, $x \mapsto \theta^{\tilde{\eta}(x)}(x)$ where the new return time is defined by, for $A:= \bigcup_{ h \in \mathfrak{g} \cup \{\id\} } [v_h]$,
\[ \tilde{\eta} :  \Omega \to \mathbb{N}, \;
x \mapsto \begin{cases}  |v_h| &:\; \exists h \in \mathfrak{g} \cup \{\id\}:\,\,x\in [v_h]\\ 
 \min\left\{|w| :  w \in \hat{\cW}^\infty, x\in[w], [w] \cap  A =\emptyset \right\} &:\; x \in \Omega\setminus A.  \end{cases}
 \] 
 
Observe that $\tilde{S}$ is defined on a set of full measure and that, as 
$H_0 \subset H$, conditions \textit{(ul)} and \textit{(fc)} are immediate from the construction of $\tilde{S}$. Now assume that $g,h \in G$. Then, by topological transitivity of $T_0$, there exist $g_1, \ldots, g_k \in \mathfrak{g}$ such that $g_1\circ \cdots g_k \in ghH_0$. Hence, as $\tilde{S}$ is the full shift, there exists a word $w$ with respect to the new partition of $\Omega$ such that 
$\hat{\gamma}_w \in ghH_0 = gH_0h \subset gHh$. Thus, $\tilde{S}$ also satisfies \textit{(tt)} and it remains to show that $\tilde{S}$ is adequately embedded and has exponential tails.

 By definition, $\tilde{S}$ is a full Markov map and each branch either is an iterate of $S$ or defined on a cylinder of type $[v_h]$. In the first case, as an iterate of a Gibbs-Markov map again is a Gibbs-Markov map with respect to the same constant, the estimate in Definition \ref{def:gm-map} holds with respect to the same constant $C$. In the second case, for each $[v_h]$, the estimate holds for $\max\{C, C_{|v_h|}\}$, where $C_n$ is given by (i) in the definition of an adequate embedding (see Definition \ref{def:embedded Gibbs-Markov structure}). As $|\mathfrak{g} \cup \{\id\}|< \infty$, one obtains a uniform bound and, in particular, $\tilde{S}$ is a Gibbs-Markov map with full branches. In order to prove that $\tilde{S}$ is adequately embedded, first observe that condition (i) is inherited from $S$ and it remains to show that there exists  
 an almost surely finite function $\tilde{\eta}^\dagger:\Omega \to \N$ such that, for almost every $x \in \Om$ and $l=0, \ldots , \tilde{\eta}(x)-1$, we have $\tilde{\eta}(x)- l \leq \tilde{\eta}^\dagger(\theta^l(x))$.

\medskip
\noindent\textsc{Case 1:} 
First assume that  $x \in [v_h]$ for some $h \in \mathfrak{g} \cup \{\id\}$ and $0 \leq l < \tilde{\eta}(x)$ with $\theta^l(x)\in \Om$, we have $\tilde{\eta}(x)-l = |v_h|-l \leq M:= \max\{|v_h| : h \in \mathfrak{g} \cup \{\id\}\}$. 

\noindent\textsc{Case 2:} Now assume that $x \notin A$ and that $0 \leq l < \tilde{\eta}(x)$ with $y:=\theta^l(x)\in \Omega$. Then there exists $m \geq 0$ such that $\eta_m(x) \leq l < \eta_{m+1}(x)$. As $\eta$ satisfies (ii) of the definition of an adequate embedding, it follows that $\eta_{m+1}(x) - l \leq {\eta}^\dagger(y)$ as illustrated in Figure \ref{fig:example}.    
\begin{figure}[htbp] 
   \centering
    \def\svgwidth{0.75\textwidth}
	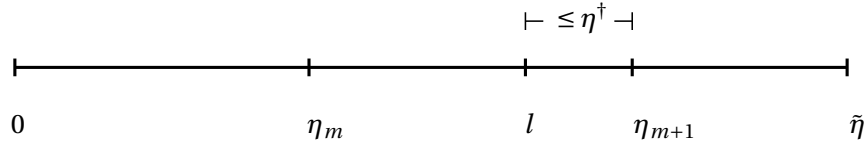
 \caption{Estimate for $\tilde{\eta}(x)-l$}
\label{fig:example}  
\end{figure}
However, by construction of $\tilde\eta$ as a stopping time, it follows that $\tilde{\eta}(x)-\eta_{m+1}(x) \le  \tilde{\eta}(\theta^{\eta_{m+1}(x)}(x)) = \tilde{\eta}(\sigma^{m+1}(x))$. Hence,
\[\tilde{\eta}(x) - l \leq {\eta}^\dagger(y) +  \tilde{\eta}(\sigma^{m+1}(x)) \leq  {\eta}^\dagger(y) + \max\left\{\tilde{\eta}(\theta^{k}(y)) : 0 <  k \leq  {\eta}^\dagger(y) \right\} =: M(y). \]
In particular, $\tilde{\eta}^\dagger(y):= \max \{N, M(y)\}$ satisfies condition (ii) of an adequate embedding. 

It remains to check that $\tilde \eta$ has exponential tails. As $\tilde \eta$ is constructed through a finite choice of elements in $\widehat{\cW^\infty}$, there exists $k \in \N$ such that the Markov partition for $S^k$ is finer than the one for $\tilde S$. Therefore, it suffices to prove that $\eta_k$ has exponential tails, which is a consequence of the  following calculation. Assume that $\eta^\ast,\eta: \Omega \to \N $ have exponential tails. That is, there exists $t \in (0,1)$ such that 
$\mu(\{x: \eta^\ast(x) =n\}) \ll t^n$ and  $\mu(\{x: \eta(x) =n\}) \ll t^n$ for all $n \in \N$. Then, using bounded distortion,
\begin{align*}
\mu\left(\left\{x: \eta(x) + \eta^\ast(S(x)) =n\right\}\right) & \asymp  \sum_{i=1}^n  \mu\left(\left\{x: \eta(x) = i  \right\}\right)  \mu\left(\left\{x:  \eta^\ast(x) =n -i \right\}\right) \ll nt^n \ll t^n.
\end{align*}
In particular,  for $\eta^\ast := \eta_k$ it follows from  $\eta_{k+1} = \eta + \eta_k\circ S$ that $\mu\left(\left\{x : \eta_{k+1}  (x) = n\right\}\right) \ll t^n$. Hence $\eta_k$ has exponential tails for each $k \in \N$.
\end{proof}

\subsection{Extensions by Schreier graphs with non-Markovian increments} \label{subsec:2xmod1}

We show how to apply embedded Gibbs-Markov maps in order to obtain amenability criteria with respect to non-Markovian increments and Ruelle expanding maps (see \cite{Ruelle--The-Thermodynamic-Formalism-For--CMP1989}).

\begin{defn} Let $(X,d)$ be a compact metric space. Then $\theta:X \to X$ is referred to as \emph{Ruelle-expanding} if there exist $a>0$ and  $\lambda \in (0, 1)$
such that the following holds: For any  $x, {y}, \tilde{x} \in X$  with $d(x, {y})<a$ and $\theta(\tilde{x})=x$, there exists a unique $\tilde{y}\in X$  with $\theta(\tilde{y})={y}$ and $d(\tilde{x}, \tilde{y})<a$, and such that this $\tilde y$ satisfies
$d(\tilde{x}, \tilde{y}) \leq  \lambda d(x,y) $.
\end{defn}

We remark that the class of class of Ruelle expanding maps is sufficiently flexible to include one-sided subshifts of finite type as well as distance expanding maps on closed manifolds. Furthermore, if $\theta$ is Ruelle expanding, then it is easy to see that $\theta^n$ also is Ruelle expanding with parameters $a/2$ and $\lambda^n$.
In particular, if $\theta^n(\tilde{x}) =x$, this implies that for each $y$ with $d(x,y) < a/2$, there exists a unique element $\tilde{y} \in \theta^{-n}(\{y\})$ with $d(\tilde{x}, \tilde{y}) \leq  \lambda^n d(x,y)$. Hence, the map defined by $\theta_{\tilde{x}}^{-n}: y \mapsto \tilde{y}$ is Lipschitz continuous, injective and $\theta^n \circ \theta_{\tilde{x}}^{-n}$ is the identity on the open ball $B_{a/2}(x)$ around $x$ with radius $a/2$. Or in other words, each pair $(n,\tilde{x})$ comes with a homeomorphism $\theta_{\tilde{x}}^{-n} : B_{a/2}(x)  \to \theta^{-n}_{\tilde{x}}( B_{a/2}(x))$, referred to as the \emph{inverse branch} of $\theta^n$ at $\tilde{x}$.

In order to employ Theorem \ref{theo:Schreier - embedded GM structure}, we now use thermodynamic formalism to construct our reference measure. That is, by assuming that $\theta$ is topologically mixing and $\varphi: X \to \R$ is Hölder continuous, it is well-known (see, e.g., \cite{Ruelle--The-Thermodynamic-Formalism-For--CMP1989} or \cite{Stadlbauer-Varandas-Zhang--Quenched-And-Annealed-Equilibrium--ETDS-2023} for a more recent exposition in the setting of semigroups) that there exists a unique invariant probability measure $\mu$ which realizes the supremum in the variational principle (i.e., $\mu$ is an equilibrium state). Moreover, by Ruelle's operator theorem, there exists a Hölder continuous and strictly positive function $h: X \to \R$ such that the transfer operator with respect to $\mu$ is of the form
\[
\widehat{\theta} (f) (x) = \sum_{\theta y = x}  e^{\varphi(y) - P(\varphi,\theta) + \log h(y) - \log h(x)} f(y),
\]
where $P(\varphi,\theta)$ refers to the topological pressure.

\begin{example} \label{ex:local-diffeo}
If $X$ is a connected {\color{black} and compact} Riemannian manifold and $\theta$ is  $C^2$-local diffeomorphism with $\|D(\theta)^{-1}\| < 1$, then the last property implies that $\theta$ is Ruelle expanding. Moreover, by combining expansion with the hypothesis that the manifold is pathwise connected, a simple argument shows that $\theta$ in fact is topologically mixing (see, e.g., Example 3.2 in \cite{Stadlbauer-Varandas-Zhang--Quenched-And-Annealed-Equilibrium--ETDS-2023}). Finally,  the $C^2$-regularity implies that $\varphi := \log \det |D(\theta)^{-1}|$ is Lipschitz continuous. Hence, there exists a unique equilibrium state.
However, as a consequence of change of variables, it follows that Lebesgue measure is a so called $\varphi$-conformal measure and that $P(\varphi,\theta)=0$. In particular, this implies that $d\mu = h d\mathrm{Leb}$.
\end{example}

We now provide sufficient conditions in order to conclude $\mu$-amenability of the Schreier graph from an extension of a $(a_{\; \mathrm{Ruelle}},\lambda_{\;  \mathrm{Ruelle}})$-Ruelle-expanding map $\theta$. In order to do so, assume that $H$ refers to a subgroup of a finitely generated  discrete group $G$ and that $\gamma : X \to G$ is a map with the following properties with respect to the equilibrium state $\mu$.
\begin{enumerate}[label=(S\arabic*),ref=(S\arabic*)]
\item\label{item:finite-cover} The image $\gamma(X)$ of $\gamma$ is finite.
\item\label{item:transitivity}
For all open subsets $U,V \subset X$ and $g \in G$, there exist $n \in \N$ and $x \in X$ such that $x \in U \cap \theta^{-n}(V) \neq \emptyset$ and $(\gamma_x \cdots \gamma_{\theta^{n-1}(x)}) g^{-1} \in \bigcap_{h \in G} hHh^{-1}$.
 \item\label{item:embedded-Markov-connected}
The set $\Delta:= \bigcup_{n \geq 0} \bigcup_{g \in \gamma(X)} \theta^n(\partial(\gamma^{-1}(\{g\})) )$ is not dense.
 \item\label{item:non-exp-decay-of-returns}
{\color{black}
 $ \limsup_{n \to \infty} \sqrt[n]{ \mu\left(\left\{ x  \in  X:  \gamma_x \cdots \gamma_{\theta^{n-1}(x)} \in H  \right\}\right) } =1. $}
\end{enumerate}
We now give a brief comment on the ideas behind \ref{item:transitivity} and \ref{item:embedded-Markov-connected}. Condition \ref{item:transitivity} essentially states that the map $T_0$ from Theorem \ref{theo:Schreier - embedded GM structure} is topologically transitive, whereas \ref{item:embedded-Markov-connected} will allow in a general context to construct an adequate embedded Gibbs-Markov structure such that the associated return time is a first return, provided that the ambient space is connected.
{\color{black}
Furthermore, if $\theta$ is topologically transitive or mixing, it is easy to see that \ref{item:embedded-Markov-connected} is equivalent to $\Delta$ being nowhere dense.
}
Using first returns then allows to use exponential decay of correlations of $\mu$ in order to obtain exponential tails.

\begin{theorem} \label{theo:Non-Markovian}
Assume that $\theta: X  \to X$ is Ruelle expanding and topologically mixing, that $X$ is locally connected and that $\mu$ is the equilibrium state associated to the Hölder continuous function $\varphi: X \to \R$. Then, if $H$ is a subgroup of the finitely generated group $G$ and $\gamma : X \to G$  satisfies \ref{item:finite-cover}, \ref{item:transitivity}, \ref{item:embedded-Markov-connected} and \ref{item:non-exp-decay-of-returns},  the Schreier graph with vertices $\{Hg : g \in G\}$ and edges $\{(Hg,Hgh): g \in G, h \in \gamma(X)\}$ is $\mu$-amenable.
\end{theorem}

\begin{proof} {\color{black} We begin with showing that there exists an open set $V \subset \Delta^c$ such that $d({V},\Delta)> 0$ of diameter smaller than $a_{\; \mathrm{Ruelle}}$  and
  \begin{equation} \label{eq:localized-rate-of-decay-of-returns} \limsup_{n \to \infty} \sqrt[n]{ \mu\left(\left\{ x  \in  \theta^{-n}(V):  \gamma_x \cdots \gamma_{\theta^{n-1}(x)} \in H  \right\}\right) } =1. \end{equation}
By compactness of $X$, it follows from \ref{item:non-exp-decay-of-returns} that there is a ball $B$ of diameter smaller than $a_{\; \mathrm{Ruelle}}$ such that \eqref{eq:localized-rate-of-decay-of-returns} holds with respect to $V=B$. Now choose an open set $V_0$ with  $d({V_0},\Delta)> 0$ of diameter smaller than $a_{\; \mathrm{Ruelle}}$. As \ref{item:transitivity} implies that the skew product is topologically transitive, it follows that there exists $k \in \N$ such that $T^{-k}(V_0 \times H) \cap  B \times H \neq \emptyset$. Hence, there exists $x_0 \in   \theta^{-k}(V_0) \cap B$ such that $\gamma(x_0) \cdots \gamma(\theta^{k-1}(x_0)) \in H$. Moreover, as the diameter of $V_0$ is smaller than $a_{\; \mathrm{Ruelle}}$,
$\theta_{x_0}^{-k}(V_0)$ is well defined. Set  $V:= \theta^k(\theta_{x_0}^{-k}(V_0) \cap B)$. Then, due to the fact that $\mu$ is a Gibbs measure, it follows that
\[ \mu(\theta_x^{- n}(B)) \asymp e^{ \sum_{j=0}^{n-1} \varphi \circ \theta^j (x)} \mu(B),
\]
for any $B$ of diameter smaller than $a_{\; \mathrm{Ruelle}}$ and $x \in \theta^{-n}B$ and assuming, without loss of generality, that the pressure of $\varphi$ is zero. Hence,
\begin{align*}
 & \mu\left(\left\{ x  \in  \theta^{-n-k}(V):  \gamma_x \cdots \gamma_{\theta^{n+k-1}(x)} \in H  \right\}\right) \\
 \geq &
  \mu\left(\left\{ x  \in  \theta^{-n}( \theta_{x_0}^{-k} V \cap B) :  \gamma_x \cdots \gamma_{\theta^{n+k-1}(x)} \in H  \right\}\right) \\
  \asymp &
  \sum_{\theta^n y = x_0} e^{ \sum_{j=0}^{n+k-1} \varphi \circ \theta^j (y)} \mu(V)
  \asymp
  \sum_{\theta^n y = x_0} e^{ \sum_{j=0}^{n-1} \varphi \circ \theta^j (y)}
  e^{ \sum_{j=0}^{k-1} \varphi \circ \theta^j (x_0)} \frac{\mu(B)\mu(V)}{\mu(B)}\\
  \asymp  &
   \mu\left(\left\{ x  \in  \theta^{-n}(B) :  \gamma_x \cdots \gamma_{\theta^{n-1}(x)} \in H  \right\}\right)
   \frac{ \mu\left(\theta_{x_0}^{-k} V \cap B\right) }{ \mu(B)}
\end{align*}
In particular, by the choice of $B$, $V$ satisfies \ref{eq:localized-rate-of-decay-of-returns}
}

We now check whether Theorem \ref{theo:Schreier - embedded GM structure} is applicable. {\color{black} In order to do so, note that \ref{item:embedded-Markov-connected} and the fact that $X$ is locally connected implies that there is an open and connected set $U \subset V$. As we may choose the diameter of $U$ smaller than $a_{\; \mathrm{Ruelle}}$, it follows that the inverse branches $\theta_x^{-n}$ are defined on all of $U$ for each $n \in \N$  and $x \in \theta^{-n}(U)$}. So assume that $n \in \N$ and $x \in \theta^{-n}(U)$. By \ref{item:embedded-Markov-connected}, $\theta_x^{-n}(U) \cap \partial(\gamma^{-1}(\{g\})) = \emptyset$ for all $g \in \gamma(X)$.  Hence, $\theta_x^{-n}(U) \subset \bigcup_{g\in \gamma(X)} \mathrm{Int}(\gamma^{-1}(\{g\}))$. As $\theta_x^{-n}(U)$ is connected, it follows that there is a unique $g$ with $\theta_x^{-n}(U) \subset \mathrm{Int}(\gamma^{-1}(\{g\}))$.

Furthermore, note that each Ruelle expanding map admits a finite Markov partition $\alpha$ such that
each  {$a \in \alpha$} satisfies $\overline{a} = \overline{\mathrm{Int}(a)}$ (in particular, $\mathrm{Int}(a) \neq \emptyset$) and that
the diameters of the atoms of the partition $\alpha_{n+1}$ generated by $\alpha, \theta^{-1}(\alpha), \ldots, \theta^{-n}(\alpha)$ tend to zero as $n \to \infty$.
Hence, there are $n \in \N$ and $b \in \alpha_{n+1}$ such that $b \subset U$. Moreover, as $\theta$ is topologically mixing and $\mu$ is equivalent to the conformal measure associated to $\varphi$, $b$ has positive measure. Hence, by Poincaré's recurrence theorem, $\mu$-almost every element in $b$ has infinitely many returns to $b$.

In order to construct the adequately embedded Markov map it remains to define $\Omega := b$, $\eta(x) := \min\{ k \geq 1: \theta^k(x) \in b \}$, $\sigma: b \to b$ as the first return to $b$ and $\beta$ as the countable partition of $\Omega$ modulo $\mu$ given by those elements of $\bigcup_n \alpha_n$ which are contained in $b$ and which are associated to a first return to $b$.
The key observation is  that $b \subset U$ implies that for each $a \in \beta$ and $0 \leq n < \eta(a)$, we have that $\theta^n(a) \subset \mathrm{Int}(\gamma^{-1}(\{g\}))$ for exactly one $g \in \gamma(X)$. Hence, the map $\hat{\gamma}: b \to G$, $x \mapsto \gamma(x) \cdots \gamma(\theta^{\eta(x)-1}x)$ is constant on the atoms of $\beta$. Moreover, by substituting $\alpha$ with $ \{ \theta^n(a) : a \in \beta, 0 \leq n < \eta(a) \}$, one obtains a countable partition of $X$ modulo $\mu$ such that $\gamma$ is constant on cylinders of length 1.

We now show that $\eta$ has exponential tails by using the decay of correlations of $\theta$ with respect to the metric of the topological Markov chain associated with the partition $\alpha$. That is, we make use of the fact, after choosing the $r$ in $d_r$ in \eqref{def: d_r metric} according to the Hölder continuity of $\varphi$, that there exists $C> 0$ and $t \in (0,1)$ such that
\begin{equation} \label{eq:exponential-decay}
\|\widehat{\theta}^k(f) - \mu(f)\|_{\textrm{Lip}} \leq C t^k \textrm{Lip}(f)
\end{equation}
for any Lipschitz continuous function $f$ and $k \in \N$. If $k = \ell m + d$ for some $\ell, m \in \N$ and $1 \leq d \leq \ell$, then
\begin{align}
 \label{eq:estimate-for-exponential-tails}
\mu(\{ x \in \Omega:  \eta (x) \geq k\}) & = \int \mathbf{1}_\Omega \prod_{j =1}^{k-1} \mathbf{1}_{\Omega^c}\circ \theta^j d\mu
\leq \int  \prod_{j =1}^{m}  \mathbf{1}_{\Omega^c}\circ \theta^{\ell j} d\mu\\
\nonumber
 & = \int \widehat{\theta}^{\ell m} \left(\textstyle \prod_{j =1}^{m}  \mathbf{1}_{\Omega^c}\circ \theta^{\ell j}\right) d\mu
  =  \int  \mathbf{1}_{\Omega^c}  \widehat{\theta}^{\ell } \left(  \mathbf{1}_{\Omega^c}  \widehat{\theta}^{\ell } \left(  \mathbf{1}_{\Omega^c}    \cdots \widehat{\theta}^{\ell } \left(  \mathbf{1}_{\Omega^c}  \right) \cdots   \right)   \right) d\mu.
\end{align}
We proceed by induction. Set $f_0 = 1$, $\epsilon_0=0$, $f_{n+1} = \widehat{\theta}^{\ell }(\mathbf{1}_{\Omega^c} f_n)$ and $\epsilon_n = f_n - \mu(f_n)$. Then
\begin{align*}
  f_{n+1}   =   \widehat{\theta}^{\ell }\left( \mathbf{1}_{\Omega^c}   (\mu(f_n) + \epsilon_n)  \right)
   =    \mu(f_n) \mu(f_1)  +  \textstyle \int      \mathbf{1}_{\Omega^c} \epsilon_n d\mu + \epsilon_{n+1}.
 \end{align*}
 Hence
 \[ \mu(f_{n+1}) =  \mu(f_n)\mu(f_1) +  \textstyle \int   \mathbf{1}_{\Omega^c} \epsilon_n d\mu
     \leq  \mu(f_n)\mu(f_1) +  \mu(\Omega^c) \|\epsilon_n\|_\infty
     \leq  \mu(\Omega^c) \left(\mu(f_n) + \textrm{Lip}(\epsilon_n) \right)
     , \]
 where we have used that  $\mu(f_1) = \mu(\Omega^c)$ and $\|\epsilon_n\|_\infty \leq \textrm{Lip}(\epsilon_n)$ as $\mu(\epsilon_n) =0$.
Now, using \eqref{eq:exponential-decay}  and  $\textrm{Lip}(fg) \leq \textrm{Lip}(f)\|g\|_\infty +  \|f\|_\infty\textrm{Lip}(g) $, we obtain that
\begin{align*}
 \textrm{Lip} (\epsilon_{n+1})
 & 
 \leq  C t^\ell\left(  \textrm{Lip}(\mathbf{1}_{\Omega^c}) \mu(f_n)        + \textrm{Lip}(\mathbf{1}_{\Omega^c} \epsilon_n) \right) \\
 &
 \leq  C t^\ell\left(  \textrm{Lip}(\mathbf{1}_{\Omega^c}) \mu(f_n)  +   \textrm{Lip}(\mathbf{1}_{\Omega^c}) \|\epsilon_n\|_\infty +  \textrm{Lip}(\epsilon_n)
 \right).
\end{align*}
By combining the last two estimates, it follows that
\[
\mu(f_{n+1}) + \textrm{Lip} (\epsilon_{n+1}) \leq \mu(f_n) \left( \mu(\Omega^c) + C \textrm{Lip}(\mathbf{1}_{\Omega^c}) t^\ell \right)
+  \textrm{Lip}(\epsilon_n)  \left(  \mu(\Omega^c) + C  t^\ell (1+ \textrm{Lip}(\mathbf{1}_{\Omega^c}))\right).
\]
In particular, as $\mu(\Omega^c) < 1$, we have that  $\mu(f_{n}) + \textrm{Lip} (\epsilon_{n}) \leq \tilde{t}^n$ for all $n \in \N$ and some $\tilde{t} \in (\mu(\Omega^c),1)$ for $\ell$ sufficiently large. Hence, by \eqref{eq:estimate-for-exponential-tails} and for $k =\ell m + d$ with  $1 < d \leq \ell$,
\[
\mu(\{ x \in \Omega:  \eta (x) \geq k\})  \leq \int f_{m} d\mu \leq \tilde{t}^m.
\]
Or, in other words, $(\Omega,\beta)$ has exponential tails.

In order to conclude the proof by applying,
Theorem \ref{theo:Schreier - embedded GM structure}, it remains to check topological  transitivity of $T_0$ and that $R_\Omega(T)=1$.
As \ref{item:transitivity} immediately implies that $T_0$ is topologically transitive as a Markov map with respect to the partition $ \{ \theta^n(a) : a \in \beta, 0 \leq n < \eta(a) \}$, it remains to verify  {\color{black}  that \eqref{eq:localized-rate-of-decay-of-returns} implies $R_\Omega(T)=1$. As the diameter of $V$ is smaller than $a_{\; \mathrm{Ruelle}}$, }
for any $n$ and $x \in \theta^{-n}(V)$, the function $x \mapsto \gamma_x$ is
constant on $ \theta_x^{-n}(V)$. Hence, $\gamma_x \cdots \gamma_{\theta^{n-1}x}$ also is constant on this set. Furthermore, it follows from $\Omega \subset U$  and conformality of $\mu$ that
$\mu(\theta_x^{-n}(V)) \asymp \mu(\theta_x^{-n}(\Omega))$. Hence,
\begin{align*}
   \mu\left(\left\{ x  \in  \theta^{-n}(V):  \gamma_x \cdots \gamma_{\theta^{n-1}(x)} \in H  \right\}\right)
\asymp
\mu\left(\left\{ x  \in  \theta^{-n}(\Omega):  \gamma_x \cdots \gamma_{\theta^{n-1}(x)} \in H  \right\}\right)
 \\
 \gg
 \mu\left(\left\{ x  \in  \Omega \cap \theta^{-n}(\Omega):  \gamma_x \cdots \gamma_{\theta^{n-1}(x)} \in H  \right\}\right).
\end{align*}
On the other hand, there is $m$ such that $\theta^m \Omega = X$. This implies that there are disjoint subsets $\Omega_1, \ldots, \Omega_d \subset \Omega$ such that $\bigcup \theta^m(\Omega_i) = X$, the restrictions $\theta^m|_{\Omega_i}$ are injective and
$\gamma_x \cdots \gamma_{\theta^{m-1}(x)}$ is constant for $x \in \Omega_i$. By transitivity, there are finite words $w_1, \ldots, w_d$ such that for any $i$, $[w_i] \subset \Omega$ and $\theta_{\eta(w_i)} ([w_i]) = \Omega$; and that
$\gamma_x \cdots \gamma_{\theta^{\eta(w_i) + m-1}(x)} \in H$ for any  $x \in \Omega_i^\ast:=  [w_i] \cap \theta^{-\eta(w_i)} (\Omega_i)$.
Hence, for $s \in (1/2,1]$,
\begin{align}
   \label{eq:topological-power-series}
   & \sum_{n=0}^\infty s^n \mu\left(\left\{ x  \in  \theta^{-n}(V):  \gamma_x \cdots \gamma_{\theta^{n-1}(x)} \in H  \right\}\right)
  \\
  \nonumber
  \ll &  \sum_{n=0}^\infty s^n \mu\left(\left\{ x  \in \Omega  \cap \theta^{-n}(\Omega):  \gamma_x \cdots \gamma_{\theta^{n-1}(x)} \in H  \right\}\right)
  \\
  \nonumber
  &  + \sum_{n=0}^\infty \sum_{i=1}^d s^{n + \eta(w_i)}
    \mu\left(\left\{ x  \in \Omega_i^\ast \cap \theta^{-\eta(w_i)}(\Omega^c) \cap \theta^{-\eta(w_i) - n}(\Omega):  \gamma_x \cdots \gamma_{\theta^{\eta(w_i) + n-1}(x)} \in H  \right\}\right)
   \\
   \label{eq:induced-power-series}\leq & (d +1)  \sum_{n=0}^\infty s^n \mu\left(\left\{ x  \in \Omega  \cap \theta^{-n}(\Omega):  \gamma_x \cdots \gamma_{\theta^{n-1}(x)} \in H  \right\}\right)
\end{align}
Hence, it follows {\color{black}
from \eqref{eq:localized-rate-of-decay-of-returns}} that the power series in \eqref{eq:topological-power-series} and \eqref{eq:induced-power-series}
have radius of convergence one. That is, $R_\Omega(T)=1$.
\end{proof}

\begin{remark} \label{remark:local-diffeo}
We would like to emphasize that in the situation of Example \ref{ex:local-diffeo}, one easily obtains a version of Theorem \ref{theo:Non-Markovian}, whose statement is independent of the equilibrium state $\mu$. That is, assume that $X$ is a connected and compact Riemannian manifold, $\theta$ a  $C^2$-local diffeomorphism with $\|D(\theta)^{-1}\| < 1$ and $\gamma:  X \to G$ a map such that \ref{item:finite-cover}, \ref{item:transitivity} and \ref{item:embedded-Markov-connected} hold.
Then, with $\mu$ referring to the equilibrium state associated with $\log |\det D\theta^{-1}|$, $h = d\mu/ d \mathrm{Leb} $ is bounded away from $0$ and infinity. Hence, \ref{item:non-exp-decay-of-returns} is equivalent to
\[ \limsup_{n \to \infty} \sqrt[n]{ \mathrm{Leb}\left(\left\{ x  \in \theta^{-n}(V) : H \gamma_x \cdots \gamma_{\theta^{n-1}(x)} = H  \right\}\right) } =1,\]
where $V$ is an open set bounded away from $\Delta$.
 As Riemannian manifolds are locally connected, Theorem \ref{theo:Non-Markovian} implies that the Schreier graph is $\mu$-amenable. However, if $\mathrm{Int}(\gamma^{-1}(\{g\})) \neq \emptyset$ for all $g \in \gamma(X)$ then $\mu(\gamma^{-1}(\{g\}))>0$ for all $g \in \gamma(X)$. Hence, as $\gamma(X)$ is finite, it follows that the Schreier graph is amenable in the usual sense.
\end{remark}

{\color{black}
\section{Billiards with convex scatterers and finite horizon}  \label{subsec:billiards}

The purpose of this part is to show how our results can be employed in the context of Young towers and billiard maps on the 2-torus. So assume that $\{\mathcal{C}_i : i = 1, \ldots, d \}$ are pairwise disjoint, simply connected $C^3$-curves with strictly positive curvature. The billiard flow is now defined on $Y := \{ \mathbb{T}^2 \setminus \bigcup_i \textrm{Interior}(\mathcal{C}_i)\}$ through the linear flow on the interior of $Y$, with hard collisions on the boundary. Moreover, we say that the billiard has finite horizon if the time a particle can travel without non-tangential collision is uniformly bounded.
Billiards of this type were studied extensively by  Sinai and  Chernov, and, following the seminal work of Young \cite{Young--Statistical-Properties-Of-Dynamical--AM2-1998}, they have become a standard example illustrating the applicability of Young tower constructions.

This hypothesis now implies that   $X :=  \bigcup_i ( \mathcal{C}_i \times [-\pi/2,\pi/2])$ is a Poincaré section for the billiard flow and that, in particular, the first return or billiard map $f: X \to X$ is well defined and the flow invariant Lebesgue measure on $Y \times \mathbb{T}^1$ induces a finite probability measure $m$ on $X$.

To prove exponential decay of correlations for $f$, Young constructed continuous families
 $\Gamma^u$ and $\Gamma^s$ of local unstable and stable leaves, respectively, satisfying, among other properties, the following conditions.
\begin{enumerate}[label=(Y\arabic*),ref=(Y\arabic*)]
 \item \label{item:Y1} $\gamma^s \cap \gamma^u$ is a singleton for all $\gamma^u \in \Gamma^u$  and $\gamma^s \in \Gamma^s$. Moreover, the angle of intersection is uniformly bounded from below. In particular, for each $ x \in \Lambda := \bigcup \gamma^s \cap \gamma^u$, there is a unique pair $(\gamma^s(x) ,\gamma^u(x))$ with $x = \gamma^s(x)  \cap \gamma^u(x)$.
 \item \label{item:Y2} $m_{\gamma^u} (\Lambda) >  0$ for all $\gamma^u \in \Gamma^u$.
 \item \label{item:Y3} There is a countable partition $(\Lambda_i)$  of $\Lambda $ such that $m_{\gamma^u} (\Lambda \setminus \bigcup_i \Lambda_i) =0$ for all $\gamma^u \in \Gamma^u$. Moreover, for each $i$, there is $R_i \in \N$
 such that  $\Gamma^s = \{ \gamma^s(f^{R_i}(x)) : x \in \Lambda_i \}$ and,  for all $x \in \Lambda_i$,   $f^{R_i}(\gamma^s(x)) \subset \gamma^s(f^{R_i}(x))$ and  $f^{R_i}(\gamma^u(x)) \supset \gamma^u(f^{R_i}(x))$.
 \item \label{item:Y4} There are uniform estimates for the contraction along $\gamma^s$, a certain backward contraction along $\gamma^u$ and the distortion of $Df$ along $\gamma^u$.
\end{enumerate}
 For the details in \ref{item:Y4}, we refer to (P3)-(P5) in \cite[p. 590]{Young--Statistical-Properties-Of-Dynamical--AM2-1998}.

\medskip

Now assume as above that  $\mathcal{G}=(\V,\E)$ is a Schreier graph associated to a subgroup $H$ of a  discrete group  $G$ with generating set  $\mathfrak{g} $. Furthermore, assume that $\kappa: X \to \mathfrak{g}$, $x \mapsto \kappa_x$ is a map which is constant on each $\mathcal{C}_i \times [-\pi/2,\pi/2]$. The graph extension of $f$ by $\kappa$ is now defined by
\[
T: X \times \V \to X \times \V, (x,Hg) \mapsto (f(x),Hg\kappa_x),
\]
which is an invertible map with a non-uniformly hyperbolic base transformation. Observe that one may identify $T$ with the Poincaré map for the billiard on the set of disjoint copies $\bigcup Y_{Hg}$  of $Y$ where each collision in $\mathcal{C}_i \subset Y_{Hg}$ comes with a teleportation to
$ Y_{Hg\kappa_x}$ whenever $x \in \mathcal{C}_i \times [-\pi/2,\pi/2]$.
In complete analogy to the above, we also define
$T_0: T: X \times G/H_0 \to X \times G/H_0, (x,H_0g) \mapsto (f(x),H_0g\kappa_x),$
where $H_0$ is the normal core of $H$.
In this context, we obtain the following.

\begin{theorem} \label{theo:sinai-billiard}
Assume that $(f,X,m)$, $\mathcal{G}$ and $T$ are as above and that the underlying billiard satisfies the above regularity conditions and has finite horizon. Furthermore, assume that the following holds with respect to the map $\kappa$ and the base $\Lambda$ of a Young tower.
\begin{enumerate}
 \item  \label{item:transitivity-sinai-billiards}
For all open subsets $U,V \subset X$ and $g \in G$, there exist $n \in \N$ and $x \in X$ such that $x \in U \cap f^{-n}(V) \neq \emptyset$ and $(\kappa_x \cdots \kappa_{f^{n-1}(x)}) g^{-1} \in \bigcap_{h \in G} hHh^{-1}$.
\item\label{item:non-exp-decay-of-returns-for-sinai-billiard}
We have that
 $ \limsup_{n \to \infty} \sqrt[n]{ m \left(\left\{ x \in \Lambda \cap f^{-n} \Lambda :  \kappa_x \cdots \kappa_{f^{n-1}(x)} \in H  \right\}\right) } =1$.
\end{enumerate}
Then $\mathcal{G}$ is $m$-amenable.
\end{theorem}

\begin{proof}
We begin with the definition of the Young tower. Set $\Delta := \{ (x,\ell) : x\in \Lambda, 0 \leq \ell < R(x)\}$ and $F: \Delta \to \Delta$, $(x,l) \mapsto (x,l+1)$ if $l+1 < R(x)$ or $(f^{R(x)}(x),0)$ else. As shown by Young, $\Delta$ comes with a Markov partition, $\Lambda \times \{0\}$ is an element of this partition, and $m(\{x \in \Lambda: R(x) > \ell\})$ decays exponentially. Using the contraction along $\gamma_s$, it is now possible to collapse the stable directions in $\Lambda$ by considering the equivalence relation $x \sim y$ if $x,y$ are in the same stable leaf. As pointed out by Young, one now obtains maps $\overline{F}: \overline{\Delta} \to \overline{\Delta}$ and $\overline{f}^R: \overline{\Lambda} \to \overline{\Lambda}$ acting on equivalence classes, as well as the partition $(\overline{\Lambda_i})$ of $\overline{\Lambda}$.
It now follows from \ref{item:Y3} that the partition $(\overline{\Lambda_i})$ has full branches.  Furthermore, the Jacobian is not affected by collapsing along stable leaves and has the required regularity (\cite[Lemma 1]{Young--Statistical-Properties-Of-Dynamical--AM2-1998}). This shows that $\overline{f}^R$ is a full Gibbs-Markov map. On the other hand, as $\overline{\Lambda} \times \{0\}$ is an element of the Markov partition of $\overline{F}$, it follows that  $\overline{f}^R$ is Gibbs-Markov map, embedded in $\overline{F}$.
In order to show that this embedding is adequate, observe that $\overline{f}^R$ is the first return to the base $\overline{\Lambda} \times \{0\}$. Moreover, it follows from the control of the derivatives in \ref{item:Y4} (see
\cite[(P4)(b)]{Young--Statistical-Properties-Of-Dynamical--AM2-1998}) that  we may choose $(C_n)$ to be a constant sequence in Definition \ref{def:embedded Gibbs-Markov structure}. Hence, we have an adequate embedding.

Furthermore, it follows from Lemma 1 in \cite{Young--Statistical-Properties-Of-Dynamical--AM2-1998} that sliding along unstable leaves is equivariant with respect to the measure on $\gamma^u$. Hence, on the level of equivalence classes, we obtain that we have an adequate embedding with exponential tails and that, on the level of decay rates, $R_{\overline{\Lambda}}(\overline{F})= R_{\Lambda}(F)$. This proves the theorem by applying
Theorem \ref{theo:Non-Markovian}.

\end{proof}

}

\subsection{Non-normal subgroups of Kleinian groups}  \label{subsec:kleinian}
A further application
is related to non-periodic covers of a certain class of hyperbolic manifolds, and gives an independent proof of the main result in \cite{Coulon-Dougall-Tapie--Twisted-Patterson-sullivan-Measures-And--PA2018} in a special case. 
That is, as in  \cite{Stadlbauer-Stratmann--Infinite-Ergodic-Theory-For--ETDS2005}, we refer to $G$ as an essentially free Kleinian group if $G$ acts on the standard hyperbolic space $\H$ of dimension $n$ and admits a Poincar\'e fundamental polyhedron $F$ with faces $f_1, f_2, \ldots f_{2n}$ and associated generators $g_1,g_2 \ldots g_{2n}$ of $G$ with $g_{i}(f_i) = f_{i+n}$, $g^{-1}_{i}(f_{i+n}) = f_{i}$ and $g_i^{-1}=g_{i+n}$ for $i=1,\ldots n$, such that the following conditions are satisfied. In the following, we refer to 
 $\overline{(\cdot)}_{\overline{\H}}$ as the closure in $\overline{\H}$. 
\begin{enumerate} 
\item If $\overline{(f_i)}_{\overline{\H}} \cap \overline{(\bigcup_{j\neq i}f_j)}_{\overline{\H}} = \emptyset $ for some $i =1,2,\ldots n$, then  $g_i,g_{i+n}$ are hyperbolic transformations, and $\overline{(f_{i+n})}_{\overline{\H}} \cap \overline{(\bigcup_{j\neq i+n}f_j)}_{\overline{\H}} = \emptyset $,    
\item if $\overline{(f_i)}_{\overline{\H}}\cap \overline{(f_j)}_{\overline{\H}}$ is a single point $p$ for some $j =1,2,\ldots 2n$, then $p$ is a parabolic fixed point of some $g\in G$, 
\item if $f_i \cap f_j \neq \emptyset$ for some $j =1,2,\ldots 2n$, then $g_ig_j = g_jg_i$.
\end{enumerate}
In fact, this class was defined in \cite{Stadlbauer-Stratmann--Infinite-Ergodic-Theory-For--ETDS2005} only in dimensions two and three, but the proofs in there generalize verbatim to arbitrary dimensions. Moreover, it is worth noting that 
the  class comprises all non-cocompact, geometrically finite Fuchsian groups, the class of Schottky groups, and moreover gives rise to geometrically finite hyperbolic manifolds which may have cusps of arbitrary rank.

As shown in \cite{Stadlbauer--The-Return-Sequence-Of--FM2004,Stadlbauer-Stratmann--Infinite-Ergodic-Theory-For--ETDS2005}, it is then possible to construct a Markov map $\theta$ acting on the conical limit set $X:=L_r(G)$ of $G$ equipped with an invariant and ergodic measure $\mu$, which is equivalent to Patterson's measure $m$ such that 
the geodesic flow on the sphere bundle of $\H/G$ is measure theoretically conjugated to a special flow over the natural extension of $(X,\theta,\mu)$. As shown in \cite{Stadlbauer-Stratmann--Infinite-Ergodic-Theory-For--ETDS2005}, there are the following three distinct situations. If there are no  parabolic elements, then an iterate of $\theta$ is uniformly expanding, $\mu$ is finite and $d\mu/dm$ is a Lipschitz continuous function bounded from above and below. 
If $G$ has parabolic elements, then $d\mu/dm$ is always unbounded but the finiteness of $\mu$ depends on two parameters, the abscissa of convergence  $\delta$ of the Poincaré series of $G$ and the maximal rank $k_{\hbox{\tiny max}}$ of the parabolic subgroups. Namely, the measure is finite if and only if $2 \delta > k_{\hbox{\tiny max}} +1$.

By employing the relations indicated in the diagram \eqref{diagram},  one now obtains a generalisation of Theorem 6.1 in \cite{Stadlbauer--An-Extension-Of-Kestens--AM2013} to subgroups with non-trivial normal core. Moreover, if the group has no parabolic elements or equivalently, is convex-cocompact, then there is a nontrivial intersection with results by Brooks, Dougall and Coulon, Dal'Bo \& Sambusetti where the same result was obtained for normal subgroups (\cite{Brooks--The-Bottom-Of-The--JRAM1985}), for normal subgroups and spaces of pinched negative curvature (\cite{Dougall--Critical-Exponents-Of-Normal--AM2019}) as well as for arbitrary subgroups and CAT(-1) spaces (\cite{Coulon-DalBo-Sambusetti--Growth-Gap-In-Hyperbolic--GFA2018}). Recently, these results were generalized in a recent preprint to \emph{strongly positively recurrent} groups with a \emph{growth gap at infinity}. 
Moreover, there are one-sided results by Roblin and Pollicott who showed that amenability implies $\delta(G) = \delta(H)$ for pinched negative curvature (\cite{Roblin--A-Fatou-Theorem-For--IJM2005}) and $h(\phi_H) \geq h(\phi_G) $
for any compact surface with a transitive geodesic (\cite{Pollicott--Amenable-Covers-For-Surfaces--AM2017}), where $h$ refers to the topological entropy of the geodesic flow. For the other direction, one should remark the result by  
Falk \& Matsuzaki in \cite{Falk-Matsuzaki--The-Critical-Exponent-The--CGD2015} who showed that $\delta=1$ implies amenability of the graph given by the pants decomposition of a hyperbolic surface of first kind. 

\begin{theorem} \label{theo-fuchsian}
Assume that $G$ either is an essentially free Kleinian group or a geometrically finite Fuchsian group and let $\mathfrak{g}$ refer to the set of generators given by the associated fundamental polyhedron. Moreover, assume that $H$ is a subgroup of $G$ such that $H_0 = \bigcap_{g \in G} gHg^{-1}$ is non-trivial. Then the Schreier graph of $H$ associated with $\mathfrak{g}$ is amenable if and only if $\delta(G) = \delta(H)$.
\end{theorem}

\begin{proof} If $G$ is an elementary Kleinian group, then the theorem holds as any subgroup is ame\-nable and $\delta(G)=0$. Hence, without loss of generality, we assume that $G$ is not elementary.  Moreover, we first consider the case of an essentially free Kleinian group
 $G$ and start with the construction of the associated embedded Gibbs-Markov map. In order to do so, first observe that each element $a \in \alpha$ of the Markov partition of $T$ corresponds to an element $g_a$ of $G$ and that, as the elements in $\alpha$ come in pairs, there exists an involution $\alpha \to \alpha$, $a \mapsto a^\dagger$  with $(a^\dagger)^\dagger=a$ and $g_{a^\dagger} = g_a^{-1}$. It is then easy to see that this involution  extends to finite words by  $(w_1 \ldots w_n)^\dagger := (w_n^\dagger \ldots w_1^\dagger)$ and also satisfies $g_{w^\dagger}= g_{w}^{-1}$  for $w=(w_1 \ldots w_n) $ and $g_w = g_{w_1} \cdots g_{w_n}$.  

\medskip 
\noindent\textsc{Step 1: The embedded symmetric Gibbs-Markov map.}
We now employ the involution for the construction of a symmetric embedded Gibbs-Markov map which only depends on $G$. Therefore, choose a word $w=(w_1 \ldots w_n)$ such that $w_1 = w_n^\dagger$ and define $\Omega:= [w] \cup [w^\dagger] \subset [w_1]$. Furthermore, as $w_1 = w_n^\dagger$, we have, for any finite word $v$, that $wvw$ is admissible if and only if $wvw^\dagger$ is admissible, which then implies that, with a slight abuse of notation, $([wvw]\cup [wvw^\dagger]) = [w^\dagger v^\dagger w^\dagger] \cup [w^\dagger v^\dagger w]$. Hence, the involution also extends to the partition $\beta$ of $\Omega$ of the first return to $\Omega$, that is to 
\begin{align*} \beta := &\left\{ [w v w] \cup [w v w^\dagger] : \exists v \in \cW^\infty \hbox{ such that }  \theta^{|wv|} \hbox{ is the first return to } \Omega \right\} \\
& \cup \left\{ [w^\dagger v w] \cup [w^\dagger v w^\dagger] : \exists v \in \cW^\infty \hbox{ such that }  \theta^{|wv|} \hbox{ is the first return to } \Omega \right\}.  
\end{align*}
As $G$ is non-elementary, we may choose $w$ such that $g_w$ is hyperbolic.   This implies that $\Omega$ is bounded away from the parabolic points and, in particular, that $d\mu/dm \asymp 1$. As $m([w v w])$ depends only, up to a constant, from $\exp\left(-\delta(G) d\left(\mathbf{o},g_{wvw}(\mathbf{o})\right)\right)$, with $d$ referring to the  hyperbolic distance, we have $\mu([u]) \asymp \mu([u^\dagger])$ for each $u \in \beta^k$ and $k \in \N$. It now immediately follows from this that 
the first return to $\Omega$ is symmetric as defined in \eqref{eq:weakly-symmetric} with respect to $N_n =0$ and $C_n =C$ for some constant $C$. Moreover, as shown in \cite{Stadlbauer-Stratmann--Infinite-Ergodic-Theory-For--ETDS2005}, the first return has the Gibbs-Markov property.  

\medskip 
\noindent\textsc{Step 2: Uniform loops and transitivity.}
We now proceed with the construction of the graph extensions and use the non-triviality of $H_0$ in order to obtain transitivity and uniform loops  by using words with $g_{\cdot} \in H_0$ as a kind of spacer. 

First observe that the limit sets $L(G)$ and $L(H_0)$ coincide by normality. Furthermore, the set of fixed points of loxodromic elements in $H_0$ is dense in $L(H_0)$. This leads to the observation that, for a given $a \in \alpha$, there always exists $v \in \cW^\infty$ such that $av$ is admissible and $g_v$ is a loxodromic element in $H_0$. By the same argument, for each $b \in \alpha$, there exists $w \in \cW^\infty$ such that $wb^\dagger$ is admissible and $g_w$ also is a loxodromic element in $H_0$. Moreover, one may choose $v,w$ such that $g_v \neq g_w$. In particular, $ g_v g_w^{-1} \in H_0 \setminus \left\{\id \right\}$ and, after a possible canceling of letters in $(vw^\dagger)$, there is $u \in \cW^\infty$ such that $aub$ is admissible and $g_u \in H_0$.  

The graph extension $T$ is now defined by $T:(x,Hg) \mapsto (\theta(x),Hgg_a)$, for $x \in [a]$, and topological transitivity is equivalent to prove that, for any $a,b \in \alpha$ and  $g,h \in G$, there exists $w \in \cW^\infty$ such that $awb$ is admissible and such that $Hgg_{w} = Hh$. This word $w$ might be constructed with the above as follows. Choose $v \in \cW^\infty$ such that $g_v = g^{-1}h$ and then $u_1,u_2  \in \cW^\infty$ such that $au_1v$ and $vu_2b$ are admissible and $g_{u_1},g_{u_2} \in H_0$. By normality of $H_0$ in $G$, we hence have that $Hgg_{u_1vu_2} = Hgg_{u_1}g_vg_{u_2}= Hgg_{u_1}g^{-1}hg_{u_2} = Hh$. 
Hence, $w:= (u_1vu_2)$ satisfies the required properties and $T$ is transitive. 

The proof of uniform loops is similar and depends again on the existence of such $u$ and the following choice of the word $w$ in the construction of the embedded Gibbs-Markov map. Choose $v\in \cW^\infty$ and $u \in \cW^\infty $ such that $v u v^\dagger$ is admissible and $g_u \in H_0$. Now define $\Omega$ and $\beta$ as  above for $w:= v u v^\dagger$. Then $g_{w u^\dagger } \in H_0$ and $\sigma^{2|vu|}([w u^\dagger w^\dagger] \cup  [w u^\dagger w]) = \Omega$. In an analogous way, it is possible to define for each $g \in \mathfrak{g}$ a cylinder $v$ with respect to $\beta$ such that $H_0g = H_0g_v$. By the same construction as in the proof of Theorem \ref{theo:Schreier - embedded GM structure}, one then obtains $\eta : \Omega \to \N$ such that $x \mapsto \sigma^{\eta(x)}(x)$ defines an adequately embedded Gibbs-Markov map $S$ which satisfies the three required conditions. Moreover, by choosing $\eta$ in a symmetric way, the Gibbs-Markov map is symmetric.  

\medskip 
\noindent\textsc{Step 3: The abscissa of convergence.} 
Observe that it is not required in Proposition \ref{prop:mu-amenable} that $\mu$ is $\theta$-invariant as the proof of the equivalence of $\mu$- and $\nu$-amenability is based on F{\o}lner sets. Following the equivalences in Diagram \eqref{diagram}, it hence follows that  $\mathcal{G}$ is $m$-amenable if and only if $R(S)=1$. Moreover, as $\mathfrak{g}$ is finite, the notions of $m$-amenability and amenability as defined by Gerl coincide in this situation.
The assertion of the theorem would now follow if $R(S)=1$ if and only if $\delta(G)=\delta(H)$. However, the proof of Theorem 6.1 in \cite{Stadlbauer--An-Extension-Of-Kestens--AM2013} applies verbatim here, as it  is a consequence of the polynomial contribution of the parabolic subgroups to the Poincaré series.

\medskip 
\noindent\textsc{Step 4: Compact surfaces.}
It remains to show the theorem for cocompact Fuchsian groups. However, by using the construction of Adler \& Flatto in \cite{Adler-Flatto--Geodesic-Flows-Interval-Maps--BAMSNS1991},
the above proof can be adapted easily.
\end{proof}

\begin{remark}\label{rem:non-necessary} 
This remark is related to Theorem \ref{theo-fuchsian} and the condition of exponential tails in Proposition \ref{prop:pressure_T_vs_S}. 
Recall that it follows from the top representation of horospheres as in \cite{Stratmann-Velani--The-Patterson-Measure-For--PLMS31995} that the parabolic gap condition holds, i.e. the abscissa of convergence of any parabolic subgroup is strictly smaller than the abscissa of the whole group, which allows to prove Theorem \ref{theo-fuchsian} without touching 
the condition of exponential tails. Furthermore, as the coding map associated with a geometrically finite Fuchsian or essentially free Kleinian group might have parabolic fixed points, it follows from the expansion of Patterson's measure around a parabolic point (see, e.g., the \emph{global measure formula} in \cite{Stratmann-Velani--The-Patterson-Measure-For--PLMS31995}), that the induced map might have polynomial tails. Hence, the condition of exponential tails is not necessary for Theorem \ref{theo:main theorem - embedded GM structure}. However, if the parabolic gap condition is not satisfied, e.g. for some surfaces with cusps of variable curvature as constructed in \cite{Dalbo-Peigne-Sambusetti--Convergence-And-Counting-In--AIFG2017}), then Theorem \ref{theo-fuchsian} probably does not hold. On the other hand, as the so-called \emph{growth gap at infinity} generalizes the parabolic gap condition to the context of CAT(-1) spaces, the results in \cite{Coulon-Dougall-Tapie--Twisted-Patterson-sullivan-Measures-And--PA2018} show that this generalized gap condition is sufficient for applications in geometry.
\end{remark}

{\color{black}
\begin{remark} We also would like to present a different application of our results based on a graph-theoretic construction of a cover. So assume that $G$ is an essentially free Kleinian group or geometrically finite Fuchsian group. Then the underlying fundamental domain defines a group presentation of $G$ in terms of generators and relations and, in particular, defines a Cayley graph of $G$. It then remains to fix a subgraph of this Cayley graph and check whether the corresponding union of translates of the fundamental domain, together with the induced side pairings, satisfies the conditions of Poincaré's polyhedron theorem.

By the general theory of covering spaces, it follows that the associated Kleinian group, say $N$, is a subgroup of $G$, and that the chosen subgraph is the Schreier graph of $N$. However, due to the constructive approach, this allows to construct explicit examples of non-regular covers with or without  $\delta(G) = \delta (N)$.
\end{remark}
}

\begin{remark}\label{rem:variable-curvature} As a closing remark with respect to this class of applications, we would like to point out that recent advances in the coding of the geodesic flow on convex-compact CAT(-1)-spaces (\cite{Constantine-Lafont-Thompson--Strong-Symbolic-Dynamics-For--JLP-M2020}) allow to adapt the proof of Theorem \ref{theo-fuchsian} to the setting of variable curvature. Namely, based on the symbolic representation from
\cite{Constantine-Lafont-Thompson--Strong-Symbolic-Dynamics-For--JLP-M2020} for the geodesic flow on the compact space, the construction of the group extension in \cite[Chapter 6]{Bispo-Stadlbauer--The-Martin-Boundary-Of--IJM2023} adapts verbatim as above to graphs as each letter of the coding is associated to an element of the convex-cocompact isometry group. As the proof of Theorem \ref{theo-fuchsian} adapts verbatim to this setting, one therefore obtains a different proof of a slightly weaker version of the main result of \cite{Coulon-DalBo-Sambusetti--Growth-Gap-In-Hyperbolic--GFA2018}.
\end{remark}

\section{Random walks on graphs and semigroups: the amenability criteria of Day and Gerl}
\label{subsec:Day}
The results of Kesten and Day provide amenability criteria through random walks with independent increments on groups and semigroups, respectively. That is, the spectral radius  of the associated Markov operator is equal to one if and only if the semigroup is amenable. Furthermore, if the random walk is symmetric, then the exponential growth of the return probability in time $n$ vanishes. 

These results are related to  Theorem \ref{theo:main_result} and Proposition \ref{prop:symmetry} in here through the Cayley graph of the semigroup which is  constructed as follows. 
Let  $\mathcal{S}$ be a discrete semigroup such that there exist 
$\mathbf{o}\in \mathcal{S}$ and a set $\mathfrak{g}$, such that each element in  $\mathcal{S}$ can be written as $\mathbf{o}\gamma_1 \cdots \gamma_n$, for $n\in \N\cup\{0\}$ and $\gamma_i \in \mathfrak{g}$. In this situation, $\mathcal{G} =\left\{\mathcal{S},\left\{(g,g\gamma) : g \in \mathcal{S}, \gamma \in \mathfrak{g} \right\} \right\}$ is referred to as the Cayley graph of $\mathcal{S}$ with root $\mathbf{o}$  and generator set  $\mathfrak{g}$. We now assume as in Day (\cite{Day--Convolutions-Means-And-Spectra--IJM1964}), that $\mathcal{S}$ satisfies the right cancellation property and that there exist right units. That is, $gh=\tilde{g}h$ implies that $g=\tilde{g}$ and there exists $u$ such that $gu=g$ for all $g \in \mathcal{S}$, or in other words, the map $\kappa_h : \mathcal{S} \to \mathcal{S}$, $g \mapsto gh$ is injective and $\kappa_u = \id$.
Now assume that $(X,\te,\mu)$ is a full Gibbs-Markov map and that $\iota: X \to \mathfrak{g}$ is onto and constant on atoms of $\alpha$. Then 
\[ T: X\times \mathcal{S} \to X\times \mathcal{S}, \quad (x,g)\mapsto \left(\theta(x),\kappa_{\iota(x)}(g)\right) \]
defines a topological Markov chain. Moreover, as $\iota$ is onto and $\mathfrak{g}$ generates $\mathcal{S}$, there exists $n \in \N$ and $(w_1 \ldots w_n) \in \cW^n$ such that $\mathbf{o} \iota(w_1) \cdots \iota(w_n)$ is a right unit. 
In particular, $T^n$ has uniform loops. However, as $\kappa$ is injective but not necessarily surjective, one has to require in addition that 
$\mathcal{S} h = \mathcal{S}$ for all $h \in \mathfrak{g}$ in order to obtain a nearest neighbour cocycle as in Definition \ref{defn:graph_extension}. We remark that this property appears to be essential for many arguments in here as it provides independence of the number of  preimages of $T$ from the second coordinate. Furthermore, it is worth noting that the condition is related to the embeddability of $\mathcal{S}$ in a group. 

As a consequence of the above, there is a non-empty intersection with the results of Kesten and Day. If $\mathcal{S}$ is a group, then Theorem \ref{theo:main_result} and Proposition \ref{prop:symmetry} are generalisations of the amenability criteria of Kesten (\cite{Kesten--Full-Banach-Mean-Values--MS1959}) to measures with the Gibbs-Markov property as Kesten's 
results are covered by the special case that  $\mu$ is a Bernoulli probability measure. That is, for a given probability measure on $\alpha$, the measure of a cylinder $[w_1,\ldots,w_n] \in \alpha_n$ is given by $\mu([w_1,\ldots,w_n]) = \prod_{k=1}^n \mu([w_k])$. 
In particular, our results might be seen as amenability criteria through stationary, exponentially $\psi$-mixing increments.   Analogously, the results of Day in \cite{Day--Convolutions-Means-And-Spectra--IJM1964} (for groups, see also \cite{Derriennic-Guivarch--Theoreme-De-Renouvellement-Pour--CRASPSA1973}) are generalised under the additional hypothesis that $\kappa$ is surjective. 
 
In the setting of random walks on graphs, we now recall the result by Gerl on strong isoperimetric inequalities in \cite{Gerl--Amenable-Groups-And-Amenable--1988}. In there, translated to the setting of graph extensions above, he considers an extension of a Gibbs-Markov map with full branches and with respect to a finite alphabet by a locally finite graph where $\mu$  is an invariant and reversible Markov measure, that is $d\mu\circ\theta/d\mu$ is constant on the atoms of $\alpha$ and the symmetry condition in \eqref{eq:weakly-symmetric} holds for $C_n=1$ and $N_n=0$. In this context, he shows that 
$R(T)=1$, $\rho(\widehat{T})=1$ and the existence of $a>0$ such that 
\[ a|K| \leq  |\partial K|, \quad \forall K \subset \V \hbox{ finite}\]
are equivalent. The latter property is referred to as a strong isoperimetric inequality and, as can easily be verified, is equivalent  to non-amenability. Furthermore, the symmetry condition in \cite{Gerl--Amenable-Groups-And-Amenable--1988} automatically implies that $T^2$ always has uniform loops. Hence, as above, the results for general Gibbs-Markov measures in  Theorem \ref{theo:main_result} and Proposition \ref{prop:symmetry} can be seen as generalisations to exponentially $\psi$-mixing increments. 


\subsection*{Acknowledgements}
The first author acknowledges financial support by JSPS KAKENHI 21K03269 and 24K06777. The second author acknowledges financial support by Coordenação de Aperfeiçoamento de Pessoal de Nível Superior - Brasil (CAPES) and the third author by the Fundação de Amparo à Pesquisa do Estado do Rio de Janeiro (FAPERJ) through grants E-26/210.388/2019 \& E-26/204.324/2024, in part, by CAPES - Finance Code 001 through a visiting grant of the third author of the PrInt program.
%


\begin{thebibliography}{10}

\bibitem{Aaronson--An-Introduction-To-Infinite--1997}
J.~Aaronson.
\newblock {\em An introduction to infinite ergodic theory}, volume~50 of {\em
  Mathematical Surveys and Monographs}.
\newblock American Mathematical Society, Providence, RI, 1997.

\bibitem{Aaronson-Denker-Urbanski--Ergodic-Theory-For-Markov--TMS1993}
J.~Aaronson, M.~Denker, and M.~Urbański.
\newblock Ergodic theory for {M}arkov fibred systems and parabolic rational
  maps.
\newblock {\em Trans. Am. Math. Soc.}, 337(2):495--548, 1993.

\bibitem{Adler-Flatto--Geodesic-Flows-Interval-Maps--BAMSNS1991}
R.~Adler and L.~Flatto.
\newblock Geodesic flows, interval maps, and symbolic dynamics.
\newblock {\em Bull. Amer. Math. Soc. (N.S.)}, 25(2):229--334, 1991.

\bibitem{Bispo-Stadlbauer--The-Martin-Boundary-Of--IJM2023}
S.~R.~P. Bispo and M.~Stadlbauer.
\newblock The Martin boundary of an extension by a hyperbolic group.
\newblock {\em Israel J. Math.} 255, page 1--62, 2023.

\bibitem{Brooks--The-Bottom-Of-The--JRAM1985}
R.~Brooks.
\newblock The bottom of the spectrum of a Riemannian covering.
\newblock {\em J. Reine Angew. Math.}, 357:101--114, 1985.

\bibitem{Constantine-Lafont-Thompson--Strong-Symbolic-Dynamics-For--JLP-M2020}
D.~Constantine, J.-F. Lafont, and D.~J. Thompson.
\newblock Strong symbolic dynamics for geodesic flows on {C}{A}{T}$(-1)$ spaces
  and other metric anosov flows.
\newblock {\em Journal de l'École polytechnique — Mathématiques},
  7:201--231, 2020.

\bibitem{Coulon-DalBo-Sambusetti--Growth-Gap-In-Hyperbolic--GFA2018}
R.~Coulon, F.~Dal'Bo, and A.~Sambusetti.
\newblock Growth gap in hyperbolic groups and amenability.
\newblock {\em Geometric and Functional Analysis}, 28(5):1260--1320, Oct.
  2018.

\bibitem{Coulon-Dougall-Tapie--Twisted-Patterson-sullivan-Measures-And--PA2018}
R.~Coulon, R.~Dougall, B.~Schapira, and S.~Tapie.
\newblock {\em Twisted Patterson-Sullivan measures and applications to amenability and coverings,}
\newblock volume 1539 of {\em Memoirs of the American Mathematical Society}, American Mathematical Society, Providence, RI, 2025.

\bibitem{Dalbo-Peigne-Sambusetti--Convergence-And-Counting-In--AIFG2017}
F.~Dal'bo, M.~Peigné, J.-C. Picaud, and A.~Sambusetti.
\newblock Convergence and counting in infinite measure.
\newblock {\em Ann. Inst. Fourier (Grenoble)}, 67(2):483--520, 2017.

\bibitem{Day--Convolutions-Means-And-Spectra--IJM1964}
M.~M. Day.
\newblock Convolutions, means, and spectra.
\newblock {\em Illinois J. Math.}, 8:100--111, 1964.

\bibitem{Derriennic-Guivarch--Theoreme-De-Renouvellement-Pour--CRASPSA1973}
Y.~Derriennic and Y.~Guivarc'h.
\newblock Théorème de renouvellement pour les groupes non moyennables.
\newblock {\em C. R. Acad. Sci. Paris Sér. A-B}, 277:A613--A615, 1973.

\bibitem{Dougall--Critical-Exponents-Of-Normal--AM2019}
R.~Dougall.
\newblock Critical exponents of normal subgroups, the spectrum of group
  extended transfer operators, and {K}azhdan distance.
\newblock {\em Advances in Mathematics}, 349:316--347, 2019.

\bibitem{Falk-Matsuzaki--The-Critical-Exponent-The--CGD2015}
K.~Falk and K.~Matsuzaki.
\newblock The critical exponent, the hausdorff dimension of the limit set and
  the convex core entropy of a Kleinian group.
\newblock {\em Conform. Geom. Dyn.}, 16(159--196), 2015.

\bibitem{Folner--On-Groups-With-Full--MS1955}
E.~F{\o}lner.
\newblock On groups with full Banach mean value.
\newblock {\em Math. Scand.}, 3:243--254, 1955.

\bibitem{Gerl--Amenable-Groups-And-Amenable--1988}
P.~Gerl.
\newblock Amenable groups and amenable graphs.
\newblock In {\em Harmonic analysis ({L}uxembourg, 1987)}, volume 1359 of {\em
  Lecture Notes in Math.}, page 181--190. Springer, Berlin, 1988.

\bibitem{Gerl--Random-Walks-On-Graphs--JTP1988}
P.~Gerl.
\newblock Random walks on graphs with a strong isoperimetric property.
\newblock {\em J. Theoret. Probab.}, 1(2):171--187, 1988.

\bibitem{Jaerisch--Group-Extended-Markov-Systems--PMS2015}
J.~Jaerisch.
\newblock Group-extended {M}arkov systems, amenability, and the
  {P}erron-{F}robenius operator.
\newblock {\em Proc. Am. Math. Soc.}, 143:289--300, 2015.

\bibitem{Kessebohmer--Large-Deviation-For-Weak--N2001}
M.~Kesseböhmer.
\newblock Large deviation for weak Gibbs measures and multifractal spectra.
\newblock {\em Nonlinearity}, 14(2):395--409, 2001.

\bibitem{Kesten--Full-Banach-Mean-Values--MS1959}
H.~Kesten.
\newblock Full {B}anach mean values on countable groups.
\newblock {\em Math. Scand.}, 7:146--156, 1959.
{\color{black}
\bibitem{Liverani-Saussol-Vaienti--A-Probabilistic-Approach-To--ETDS1999}
C.~Liverani, B.~Saussol and S.~Vaienti.
A probabilistic approach to intermittency.
\emph{Ergodic Theory Dynam. Systems}, 19(3):   {671--685}, {1999}.
}






\bibitem{MauUrb03} R. D. Mauldin, M. Urba\'nski. {\it Graph directed Markov systems: Geometry and Dynamics of Limit Sets.} Cambridge Tracts in Mathematics  {\bf  148}, 2003.

\bibitem{Pinheiro--Expanding-Measures--AIHPANL2011}
V.~Pinheiro.
\newblock Expanding measures.
\newblock {\em Ann. Inst. H. Poincaré Anal. Non Linéaire}, 28:889--939,
  2011.

\bibitem{Pollicott--Amenable-Covers-For-Surfaces--AM2017}
M.~Pollicott.
\newblock Amenable covers for surfaces and growth of closed geodesics.
\newblock {\em Adv. Math.}, 319:599--609, 2017.

\bibitem{Roblin--A-Fatou-Theorem-For--IJM2005}
T.~Roblin.
\newblock A Fatou theorem for conformal densities with applications to Galois
  coverings in negative curvature. (un théorème de Fatou pour les densités
  conformes avec applications aux revêtements galoisiens en courbure
  négative.).
\newblock {\em Isr. J. Math.}, 147:333--357, 2005.

\bibitem{Ruelle--The-Thermodynamic-Formalism-For--CMP1989}
D.~Ruelle.
\newblock The thermodynamic formalism for expanding maps.
\newblock {\em Comm. Math. Phys.}, 125(2):239--262, 1989.

\bibitem{Sarig--Thermodynamic-Formalism-For-Countable--ETDS1999}
O.~M. Sarig.
\newblock Thermodynamic formalism for countable Markov shifts.
\newblock {\em Ergodic Theory Dyn. Syst.}, 19(6):1565--1593, 1999.

\bibitem{Sarig--Existence-Of-Gibbs-Measures--PMS2003}
O.~M. Sarig.
\newblock Existence of Gibbs measures for countable Markov shifts.
\newblock {\em Proc. Am. Math. Soc.}, 131(6): 1751-1758, 2003.




\bibitem{Stadlbauer--The-Return-Sequence-Of--FM2004}
M.~Stadlbauer.
\newblock The return sequence of the {B}owen-{S}eries map for punctured
  surfaces.
\newblock {\em Fund. Math.}, 182(3):221--240, 2004.

\bibitem{Stadlbauer--An-Extension-Of-Kestens--AM2013}
M.~Stadlbauer.
\newblock An extension of {K}esten's criterion for amenability to topological
  {M}arkov chains.
\newblock {\em Adv. Math.}, 235:450--468, 2013.

\bibitem{Stadlbauer-Stratmann--Infinite-Ergodic-Theory-For--ETDS2005}
M.~Stadlbauer and B.~O. Stratmann.
\newblock Infinite ergodic theory for {K}leinian groups.
\newblock {\em Ergod. Th. Dynam. Sys.}, 25(4):1305--1323, 2005.

\bibitem{Stadlbauer-Varandas-Zhang--Quenched-And-Annealed-Equilibrium--ETDS-2023}
M.~Stadlbauer, P.~Varandas, and X.~Zhang.
\newblock Quenched and annealed equilibrium states for random Ruelle expanding
  maps and applications.
\newblock {\em Ergodic Theory Dyn. Syst.}, (43):3150--3192, 2023.

\bibitem{SU-PsM}
B. Stratmann, M. Urba\'nski,
Pseudo-Markov Systems and Infinitely 
Generated Schottky Groups, \emph{Amer. J. Math.} 129: 1019--1062,  2007.

\bibitem{Stratmann-Velani--The-Patterson-Measure-For--PLMS31995}
B.~Stratmann and S.~L. Velani.
\newblock The Patterson measure for geometrically finite groups with parabolic
  elements, new and old.
\newblock {\em Proc. London Math. Soc. (3)}, 71(1):197--220, 1995.

{\color{black}
\bibitem{Thaler--Transformations-On-01-With--IJM1983}
M.~Thaler,
Transformations on $[0,1]$ with infinite invariant measure,
\emph{Israel J. Math.} 46: 67--96, 1983.


\bibitem{Young--Statistical-Properties-Of-Dynamical--AM2-1998}
L.-S.~Young.
\newblock {{Statistical properties of dynamical systems with some hyperbolicity}},
\newblock
\emph{Ann. of Math. (2)} {147}: {585--650},
{1998}. }

\bibitem{Yuri--Thermodynamic-Formalism-For-Certain--ETDS1999}
M.~Yuri.
\newblock Thermodynamic formalism for certain nonhyperbolic maps.
\newblock {\em Ergodic Theory Dyn. Syst.}, 19(5):1365--1378, 1999.

\end{thebibliography}
\end{document}